\documentclass[runningheads,11pt]{article}
    \usepackage{amsmath}
    \usepackage{booktabs} % For pretty tables
    \usepackage{caption} % For caption spacing
    \usepackage{subcaption} % For sub-figures
    \usepackage{graphicx}
    \usepackage{pgfplots}
    \usepackage[all]{nowidow}
    \usepackage[utf8]{inputenc}
    \usepackage{tikz}
    \usetikzlibrary{er,positioning,bayesnet}
    \usepackage{multicol}
    \usepackage{algpseudocode,algorithm,algorithmicx}
    \usepackage{hyperref}
    \usepackage[inline]{enumitem} % Horizontal lists
    \usepackage[a4paper,bottom=2.5cm, left=2.5cm, right=2.5cm]{geometry}
    \usepackage{sectsty}    % \usepackage{secdot}
    \usepackage{titlesec}
    \usepackage{amsfonts} 
    \usepackage{amssymb}
    \usepackage{setspace}
    \usepackage{nicematrix}
    \usepackage{amsthm}
    \usepackage[affil-it]{authblk}
    \usepackage{titletoc} % Chapter title formatting in ToC
    \usepackage[parfill]{parskip}
    \renewcommand{\thesection}{\arabic{section}} 
    \renewcommand{\thesubsection}{\arabic{section}.\arabic{subsection}}
    
    \newcommand*{\eqdef}{\mathrel{\vcenter{\baselineskip0.5ex \lineskiplimit0pt
                \hbox{\scriptsize.}\hbox{\scriptsize.}}}=}
    \newcommand\ee{\large\text{\fontencoding{U}\fontfamily{boondoxuprscr}\fontshape{n}\selectfont e}}
    \titleformat{\section}{\normalsize\centering\mdseries\scshape}{\thesection}{1em}{}
    \titleformat{\subsection}[runin]
    {\normalfont\normalsize\bfseries}{\thesubsection}{1em}{}
    \titleformat{\subsubsection}[runin]
    {\normalfont\normalsize\bfseries}{\thesubsubsection}{1em}{}
    
    \DeclareMathAlphabet{\mathdutchcal}{U}{dutchcal}{m}{n}

    \definecolor{blue}{HTML}{1F77B4}
    \definecolor{orange}{HTML}{FF7F0E}
    \definecolor{green}{HTML}{2CA02C}
    \pgfplotsset{compat=1.14}

    \newcommand{\R}{\mathbb{R}}
    \newcommand{\C}{\mathbb{C}}
    \newcommand{\N}{\mathbb{N}}
    \newcommand{\Lop}{Lopatinski\u{\i}\;}

    \DeclareMathOperator{\ima}{\mathrm{Im}}
    
    \DeclareMathOperator{\real}{\mathrm{Re}}
    
    \DeclareMathOperator{\Op}{\mathrm{Op}}

    \DeclareMathOperator{\WR}{\mathcal{WR}}
    \DeclareMathOperator{\supp}{\mathrm{supp}}
    
    \makeatletter
    \newcommand{\mat}[1]{{\mathpalette\mat@{#1}}}
    \newcommand{\mat@}[2]{%
      \begingroup
      \sbox\z@{$\m@th#1\underline{#2}$}%
      \dimen@=\dp\z@ \advance\dimen@ -2\mat@dimen{#1}%
      \dp\z@=\dimen@
      \sbox\z@{$\m@th\underline{\box\z@}$}%
      \box\z@
      \endgroup
    }
    \newcommand\mat@dimen[1]{%
      \fontdimen8
      \ifx#1\displaystyle\textfont\else
      \ifx#1\textstyle\textfont\else
      \ifx#1\scriptstyle\scriptfont\else
      \scriptscriptfont\fi\fi\fi 3
    }
    \makeatother
    
    \newtheorem{Assumption}{Assumption}[section]
    \newtheorem{Theorem}{Theorem}[section]
    \newtheorem{Proposition}{Proposition}[section]
    \newtheorem{Definition}{Definition}[section]
    \newtheorem{Remark}{Remark}[section]
    \newtheorem{Lemma}{Lemma}[section]
    \newtheorem{Corollary}{Corollary}[section]

    \DeclareMathAlphabet{\mathdutchcal}{U}{dutchcal}{m}{n}    
    \DeclareFontFamily{U}{mathx}{}
    \DeclareFontShape{U}{mathx}{m}{n}{<-> mathx10}{}
    \DeclareSymbolFont{mathx}{U}{mathx}{m}{n}
    \DeclareMathAccent{\widehat}{0}{mathx}{"70}
    \DeclareMathAccent{\widecheck}{0}{mathx}{"71}

    \newcommand{\vertiii}[1]{{\left\vert\kern-0.25ex\left\vert\kern-0.25ex\left\vert #1 
    \right\vert\kern-0.25ex\right\vert\kern-0.25ex\right\vert}}
    
    \newtheorem{innercustomgeneric}{\customgenericname}
    \providecommand{\customgenericname}{}
    \newcommand{\newcustomtheorem}[2]{%
      \newenvironment{#1}[1]
      {%
       \renewcommand\customgenericname{#2}%
       \renewcommand\theinnercustomgeneric{##1}%
       \innercustomgeneric
      }
      {\endinnercustomgeneric}
    }
    
    \newcustomtheorem{customthm}{Theorem}
    \newcustomtheorem{customproposition}{Proposition}
    
    \newcommand{\smallsim}{\smallsym{\mathrel}{\sim}}
    \makeatletter
    \newcommand{\smallsym}[2]{#1{\mathpalette\make@small@sym{#2}}}
    \newcommand{\make@small@sym}[2]{%
      \vcenter{\hbox{$\m@th\downgrade@style#1#2$}}%
    }
    \newcommand{\downgrade@style}[1]{%
      \ifx#1\displaystyle\scriptstyle\else
        \ifx#1\textstyle\scriptstyle\else
          \scriptscriptstyle
      \fi\fi
    }
    \makeatother
\begin{document}
    \setstretch{1.125}
    % \title{ENERGY ESTIMATES FOR THE $\mathcal{WR}$ CLASS}
    \title{ENERGY ESTIMATES FOR THE WR CLASS}
    \author{SANTIAGO CORREA}
    \affil{\normalfont University of Göttingen}
    % \email{david-santiago.correa-cardeno@mathematik.uni-goettingen.de}\\
    % }
    \date{}
    \maketitle   
    \let\clearpage\relax

    \setcounter{tocdepth}{1}
    % \tableofcontents
    \begin{abstract}
    In this paper, we discuss energy estimates for a particular class of linear hyperbolic boundary value problems known as weakly regular of real type. Such class, also called $\WR$ in the literature, is relevant in many physical situations like the formation of shock waves in isentropic gas dynamics. In this and other $\WR$ examples, the failure of the uniform \Lop condition plays a major role since it is associated with a loss of regularity in the scale of Sobolev spaces, eventually leading to energy inequalities that are ill-suited for nonlinear problems when solved by iteration. In the course of this work, we derive a priori estimates for the model case that are comparable to existing results, but using a more robust approach that we can extend to scalar or system problems with variable coefficients. The result is a significant step towards the ultimate goal of having estimates that can be combined with Piccard's or Newton's iterative scheme to analyze nonlinear situations.
\end{abstract}
\section{Introduction}
    \subsection{Background.}
        Let $\mathbb{R}^d_{+}=\{x=(y,x_d)\in \mathbb{R}^d: x_d>0\}$ be the half-space and suppose that $L(t,x,D_t,D_x)$ is a first-order linear differential operator of the form
         \begin{equation*}
            L(t,x,D_t,D_x) = D_t + \sum_{j=1}^d A_j(t,x) D_j.
        \end{equation*}
        Here, $D_j\eqdef-i\partial_{x_j}$ and $A_1(t,x),\, \cdots, A_d(t,x)$ are $n\times n$ matrix-valued functions with real entries depending on $(t,x)\in \R\times\overline{\mathbb{R}}^d_{+}\simeq \R_{+}^{1+d}$. In addition, denote by $\mathcal{M}_{m\times n}(\R)$ the set of matrices with dimensions $m\times n$ and real entries. 
        
        Consider the boundary value problem
        \begin{align}\label{PureBVP}
            \left\{\begin{aligned}
                Lu(t,x)&=f(t,x) &\quad (t,x)&\in \R\times\R^d_{+},\\
                    Bu|_{x_d=0}(t,y)&=g(t,y) &\quad (t,y)&\in \R\times\R^{d-1},
                    \\
            \end{aligned}\right. 
        \end{align}
        where $L$ is a hyperbolic operator, $B\in C^\infty(\R\times\R^{d-1}, \mathcal{M}_{p\times n}(\mathbb{R}))$, $\,\cdot\,|_{x_d=0}:
        C^\infty(\overline{\mathbb{R}}^d_{+})\to C^\infty(\R^{d-1})$ means restriction to the boundary (properly extended to larger function spaces), and the source terms $f, g$ are chosen from suitable function spaces to be specified later. We assume the boundary to be non-characteristic for $L$, i.e., $\det{A_d(t,x)}\neq0$ for every $(t,x)\in \R\times\overline{\mathbb{R}}^d_{+}$ and that the number $p$ of boundary conditions is the number of positive eigenvalues of $A_d$. 
        
        It has been known since the seminal works of \Lop\hspace{-1mm}, Kreiss, and Sakamoto that the weak Lopatinski condition is both sufficient and necessary for the $C^\infty$ well-posedness of \eqref{PureBVP}, while the uniform Lopatinski condition is necessary and sufficient for the $H^s$ well-posedness of  \eqref{PureBVP} without loss of derivatives (see \cite{benzoni-gavage_serre_2007}, \cite{kreiss1970initial}, and \cite{sakamoto1970mixed} for details). However, this lossless scenario is the exception rather than the rule, and it is observed in many interesting examples that the \Lop condition is met weakly but not uniformly \footnote{this failure being characterized in a distinctive way.}, thus raising the question of whether it is possible to systematically classify such examples. The answer to this issue proves to be positive, and fits indeed in a more general perspective: besides the two stable classes of hyperbolic boundary value problems $(L,B)$ where either the weak \Lop condition fails (strongly unstable) or the uniform \Lop holds (strongly stable), Benzoni--Gavage, Rousset, Serre, and Zumbrun have identified in \cite{benzoni-gavage_rousset_serre_zumbrun_2002} a third stable class that they named \textit{weakly regular of real type}, or $\WR$ for short, for which the \Lop condition degenerates to the first order as one approaches the so-called hyperbolic region. 
        % Moreover, the authors also characterized the transition regions between these three classes.
        
        When it comes to energy estimates, strongly unstable and strongly stable boundary value problems are well understood. In the former, there is no hope for any satisfactory theory; in the latter, the uniform \Lop condition has been shown to be equivalent to an energy estimate of the type
        \begin{align}\label{StrongEstimate}
            \gamma\int_{\R_{+}^{1+d}}{{e^{-2\gamma t}\vert{u}\vert^2}}\,dt\,dx&+\int_{\R^d}{{e^{-2\gamma t}\vert{u|_{x_d=0}}\vert^2}}\,dt\,dy\\\nonumber&\leq C\left(\frac{1}{\gamma}\int_{\R_{+}^{1+d}}{{e^{-2\gamma t}\vert{Lu}\vert^2}}\,dt\,dx+\int_{\R^d}{{e^{-2\gamma t}\vert{Bu|_{x_d=0}}\vert^2}}\,dt\,dy\right),
        \end{align}
        \normalsize
        where we have the remarkable feature that both the input $f, g$ and the solution are estimated in the same norm. Though formula \eqref{StrongEstimate} cannot apply to the $\WR$ class by its very definition, it is certainly possible to deduce energy inequalities with loss of derivatives. For instance, Coulombel obtains in \cite{coulombel2002weak} and \cite{coulombel_2004} a priori estimates while studying the linear stability of multidimensional shock waves for hyperbolic systems of conservation laws assuming that Majda's uniform condition is violated (see \cite{majda1983stability}). Roughly speaking, his strategy is based upon modifying the original Kreiss construction of a microlocal symmetrizer from which he gets estimates of the form 
        \begin{equation}\label{EE_Cou}
            \gamma \vertiii{u}_{0,\gamma}^2+\Vert{u|_{x_d=0}}\Vert^2_{0,\gamma}\lesssim\frac{1}{\gamma^3}\vertiii{Lu}_{1,\gamma}^2+\frac{1}{\gamma^2}\Vert{Bu}\Vert^2_{1,\gamma},
        \end{equation}
        with 
        \begin{equation*}
            \vertiii{u}^2_{s,\gamma}\eqdef \int_{0}^{\infty}{\Vert{u(\cdot,x_d)}\Vert^2_{s,\gamma}}\,dx_d, \quad \mathrm{and} \quad \Vert{u}\Vert^{2}_{s.\gamma}\eqdef\frac{1}{(2\pi)^d}\int_{0}^{\infty}{(\gamma^2+\vert{\xi}\vert^2)^s\vert{\hat{u}(\xi)}\vert^2}\,d\xi. 
        \end{equation*}

        Formula \eqref{EE_Cou} holds true for both constant and variable coefficients, and is even satisfied when the coefficients have limited regularity (since Coulombel uses paradifferential calculus). Furthermore, \eqref{EE_Cou} is optimal on the scale of Sobolev spaces as shown in \cite{coulombel2010geometric} using geometric optics expansions. Alternatively, Serre proposes in \cite{benzoni-gavage_rousset_serre_zumbrun_2002} another approach to derive energy estimates for the $\WR$ class in which a specific operator is applied to the solution and the data alike. His method encompasses two different scenarios, namely, boundary value problems for second-order scalar hyperbolic operators with constant and variable coefficients, and boundary value problems for first-order systems with constant coefficients. In the scalar case, Serre investigates the wave operator $L\eqdef-\partial_d-c^2\Delta_x$ in a half-space supplemented by a boundary condition $B\eqdef-\partial_d-\beta-\partial_t-v\nabla_y$. Then, through an intricate factorization, $(L,B)$ is decomposed into a new boundary value problem $(L,E)$ satisfying the uniform \Lop condition, plus a Cauchy problem for the pseudodifferential equation $Pu=z$, where $P$ is a ``filter'' operator that vanishes at points on the boundary where the \Lop is violated. By an appropriate choice of the parameters on which $P$ depends, one concludes that
            \begin{equation*}
                \gamma\Vert{\nabla_{t,x}Pu}\Vert_{L_{\gamma}^2(Q)}+\Vert{\nabla_{t,x}Pu}\Vert^2_{L_{\gamma}^2(\partial Q)}\leq C\left(\frac{1}{\gamma}\Vert{Pf}\Vert^2_{L^2_{\gamma}(Q)}+\Vert{Rg}\Vert^2_{L_{\gamma}^2(\partial Q)}\right).
            \end{equation*}
        % or
        % \begin{equation*}
        %     \gamma\Vert{\nabla_{t,x}u}\Vert_{L_{\gamma}^2(Q)}+\Vert{\nabla_{t,x}u}\Vert^2_{L_{\gamma}^2(\partial Q)}\leq \frac{C}{\gamma^2}\left(\frac{1}{\gamma}\Vert{Pf}\Vert^2_{L^2_{\gamma}(Q)}+\Vert{Rg}\Vert^2_{L_{\gamma}^2(\partial Q)}\right).
        % \end{equation*}
        The variable coefficient case for a general second-order hyperbolic operator $L$ follows the same principle, except that one needs to account for lower order terms at each step, as the composition formulae are no longer exact. 
        
        The treatment of first-order systems only \textcolor{black}{pertains} to the model problem (i.e., when $P$ has constant coefficients) and is done by controlling the Fourier transform of $u(0)$ using a pseudodifferential operator with symbol $p(\tau,\eta)=\pi_{+}+\Delta(\tau,\eta)\pi_{-}$, with $\pi_{-}, \pi_{+}$ being the projectors on the stable and unstable subspaces, and $\Delta(\tau,\eta)$ being the \Lop determinant. In the end, using Plancherel's theorem we have
        % \small
        \begin{align}\label{EE_BGS}
            &\gamma\int_{\Omega\times\R}{{e^{-2\gamma t}\Vert{P_{\gamma}u}\Vert^2}}\,dx\,dt+\int_{\partial\Omega\times\R}{{e^{-2\gamma t}\Vert{P_{\gamma}\gamma_0u}\Vert^2}}\,dx\,dt\\\nonumber&\leq C\left(\frac{1}{\gamma}\int_{\Omega\times\R}{{e^{-2\gamma t}\Vert{P_{\gamma}Lu}\Vert^2}}\,dx\,dt+\int_{\partial\Omega\times\R}{{e^{-2\gamma t}\vert{P_{\gamma}\gamma_0(A^d)^{-1}M^TBu}\vert^2}}\,dx\,dt\right),
        \end{align}
        where $\gamma_0$ is the trace operator and $M$ is some matrix such that $(A^d)^{-1}M^TB$ is a projector. In essence, though Serre's philosophy serves as the starting point for this paper, his techniques rely on ad hoc steps that barely admit any generalization beyond the model problem.
    \subsection{Notation.}   
        For future reference, we take $(\tau,\xi)=(\tau,\eta,\xi_d)$ as the set of covariables of $(t,x)=(t,y,x_d)$  and define 
        \begin{equation*}
            A(t,x,\eta, \xi_d)=\sum_{j=1}^{d-1}{\eta_j A_{j}(t,x)+\xi_dA_d}.
        \end{equation*}
        The symbol of $L(t,x,D_t,D_x)$ and its characteristic polynomial are then
         \begin{equation*}
            L(t,x,\tau,\eta,\xi_d)=\tau+A(t,x,\eta, \xi_d),
        \end{equation*}
        \begin{equation}\label{PriSym}
           \pi_L(t,x,\tau,\xi)\eqdef\det L(t,x,\tau,\eta,\xi_d)=\det{(\tau+A(t,x,\eta, \xi_d))}.
        \end{equation}
        The frequency set and its projection onto $\{\gamma=0\}$ are characterized by
        \begin{align*}
            \Xi\eqdef \{\zeta=(\tau-i\gamma,\eta)\in \C\times\R^{d-1}\setminus\{0,0\}: \gamma\geq 0\} \quad \mathrm{and} \quad \Xi_0\eqdef \Xi\cap\{\gamma=0\},
        \end{align*}
        whereas the space-time-frequency set and its projection onto $\{\gamma=0\}$ are given by
        \begin{align*}
            \mathbb{X}\eqdef \{(t,y,x_d,\tau,\eta,\gamma): (t,y, x_d)\in \R_{+}^{1+d}, (\tau-i\gamma,\eta)\in \Xi\} \quad \mathrm{and} \quad \mathbb{X}_0 \eqdef \mathbb{X}\cap\{\gamma=0\}.
        \end{align*}
        Let $X=(t,y,x_d,\eta,\tau,\gamma)$ be here and everywhere a generic point in $\mathbb{X}$. As we shall see in due course, the symbols in this work are classical and therefore reducible to homogeneous pieces in $(\tau-i\gamma,\eta)$, so it may be advantageous to focus from time to time on the sphere 
        \begin{equation*}
            S^d  \eqdef \{(\tau-i\gamma, \eta)\in \Xi: \gamma^2 +\tau^2+|\eta|^2=1\},
        \end{equation*}
        or when required, on the set
        \begin{equation*}
            \mathbb{X}_S \eqdef \{X \in \mathbb{X}: (\tau,\eta,\gamma) \in S^d\}.
        \end{equation*}
        In the same vein, we define 
        \begin{equation*}
            \mathbb{Y}\eqdef \{X \in \mathbb{X}: X=(t,y,0,\eta,\tau,\gamma)\} \qquad \mathrm{and} \qquad \mathbb{Y}_S \eqdef \{X \in \mathbb{Y}: (\tau,\eta,\gamma) \in S^d\}.
        \end{equation*}
        Finally, the energy estimates we are interested in are modeled over weighted-in-time spaces $e^{\gamma t}L^2$ consisting of functions $u$ such that $e^{-\gamma t}u\in L^2$.
        
    \subsection{Assumptions and main result.}   
        In the sequel, we shall deal with constantly \textit{symmetrizable hyperbolic operators} in the $t$-direction, the precise meaning of which is shown below.
        \begin{Definition}\label{HyperbolicityGeneral}
            The operator $L$ is symmetrizable hyperbolic of constant multiplicities if there are positive integers $\alpha_1,\,\cdots,\alpha_q$ and real analytic, pairwise different functions
            \begin{equation}
                \lambda_1(t,x,\xi),\,\cdots,\lambda_q(t,x,\xi)
            \end{equation}
            defined on $\R^{1+d}_{+}\times\R^d$ such that: 
            \begin{enumerate}[label=\normalfont(\alph*)]
                \item $\pi_L(t,x,\tau,\xi)=\prod_{k=1}^{q}(\tau+\lambda_k(t,x, \xi))^{\alpha_k}$ and
                    % \begin{equation}\label{Char(L)}
                    %     \pi_L(t,x,\tau,\xi)=\prod_{k=1}^{q}(\tau+\lambda_k(t,x, \xi))^{\alpha_k},
                    % \end{equation}
                    % and
                \item $\lambda_1(t,x,\xi),\,\cdots,\lambda_q(t,x,\xi)$ are semi-simple when understood as the eigenvalues of $A(t,x,\eta, \xi_d)$.
            \end{enumerate}
             Notably, when $\alpha_k=1$ for every $k\in \{1,\,\cdots,q\}$, the eigenvalues are simple and $L$ is said to be strictly hyperbolic.
            % and $\lambda_1(t,x,\xi),\,\cdots,\lambda_q(t,x,\xi)$, when understood as the eigenvalues of $A(t,x,\eta, \xi_d)$, are semi-simple. Notably, when $\alpha_k=1$ for every $k\in \{1,\,\cdots,q\}$, the eigenvalues are simple and $L$ is said to be strictly hyperbolic.
        \end{Definition}
        For the sake of clarity, let us summarize the structural hypothesis for Problem \ref{PureBVP}.
        % We supplement Problem \ref{PureBVP} with some structural hypotheses.
        \begin{Assumption}\label{Assumption1} \phantom{}
            \begin{enumerate}[label=\normalfont(\roman*)]
                \item $L$ is symmetrizable of constant multiplicities.
                \item The boundary $\partial\R^{1+d}_{+}\simeq\R^{d}$ is non-characteristic for $L$, meaning that $A_d$ is nonsingular for all $(t,y,0)\in \R^{d}$. 
                \item $B$ is assumed to be everywhere of maximal rank $p$, with $p$ being the number of positive eigenvalues of $A_d$ \footnote{the number of incoming characteristics.}.
            \end{enumerate}
        \end{Assumption}
        In the study of boundary value problems, we take into account the special role of $x_d$ and recast the differential equation in \eqref{PureBVP} as 
        \begin{equation*}
            D_{d}{u}+e^{\gamma t}A_d^{-1}\left(D_t-i\gamma +\sum_{j=1}^{d-1}A_j(t,x)D_j\right)e^{-\gamma t}u=A_d^{-1}f(t,x).
        \end{equation*}
        Thus, we get the equivalent problem
        \begin{equation}\label{(P,B)_Dif}
            \left\{\begin{aligned}
                 P_{\gamma}u(t,x)\eqdef\left(D_d+ \mathdutchcal{A}_{\gamma}(t,y,x_d,D_t,D_y)\right)u(t,x)&=A_d^{-1}f(t,x) &\quad &\quad (t,x)\in &&\R_{+}^{1+d},\\
                 B_{\gamma}(t,y)u(t,y)|_{x_d=0}&=g(t,y) &\quad &\quad (t,y)\in &&\R^d,
            \end{aligned}\right.
        \end{equation}
        with 
        \begin{equation*}
            \mathdutchcal{A}_{\gamma}(t,y,x_d,D_t,D_y)\eqdef e^{\gamma t}A_d^{-1}\left(D_t-i\gamma I_n+\sum_{j=1}^{d-1}A_j(t,x)D_j\right)e^{-\gamma t} \quad \textrm{and} \quad  B_{\gamma}(t,y)\eqdef B(t,y).
        \end{equation*}
        To cover both the scalar and the system case, we shall examine a more comprehensive situation hereafter. More precisely, we postulate that a reasonable generalization of \eqref{(P,B)_Dif} is therefore the boundary value problem 
        \begin{equation}\label{SystemPDO}
            \left\{\begin{aligned}
                 P_{\gamma}u(t,x)\eqdef\left(D_d+ \mathdutchcal{A}_{\gamma}(t,y,x_d,D_t,D_y)\right)u(t,x)&=f(t,x) &\quad &\quad (t,x)\in \R_{+}^{1+d},\\
                 B_{\gamma}(t,y)u(t,y)|_{x_d=0}&=g(t,y) &\quad &\quad (t,y)\in \R^d,
            \end{aligned}\right.
        \end{equation}
        where $\mathdutchcal{A}_{\gamma} \in \mathrm{OPS}_{\gamma}^1(\R^{1+d}_{+}\times[\gamma_0, +\infty))$ is a classical pseudodifferential operator whose principal part $a(t,y, x_d,\tau,\eta,\gamma)$ admits an asymptotic expansion
        \begin{equation*}
            a\sim \sum^{\infty}_{j=0} a_{1-j},
        \end{equation*}
        each $a_{1-j}$ being an $n\times n$ homogeneous matrix of degree $1-j$. In the same spirit, $B_{\gamma}\in \mathrm{OPS}_{\gamma}^{0}$ is a classical pseudo-differential operator with parameter having a $p\times n$ matrix $b$ as principal part. Meanwhile, the source data $f$ and $g$ are chosen at least in $L^2_{\gamma}$. 
        
        \begin{Definition}\label{Hyperbolicity}
            Let $\underline{X}=(\underline{t},\underline{x},\underline{\tau},\underline{\eta},\underline{\gamma}) \in\mathbb{X}_{S}$ and set $\pi_P(X,\xi_d)=\det(\xi_dI_n+a(X))$. 
            \begin{enumerate}[label=\normalfont(\roman*)]
                \item \label{hypG1} When $\underline{\gamma}\neq 0$, $\pi_P(\underline{X},\xi_d)\neq 0$ for every $\xi_d \in \R$. This may be rephrased by saying that $a(X)$ has no real eigenvalues when $\underline{\gamma}\neq 0$.
                \item \label{hypG2} If $\underline{\xi_d}\in \R$ is such that $\pi_P(\underline{X},\underline{\xi_d})=0$, there exist $\alpha\in\N$ together with smooth functions $\lambda(t,x,\eta,\xi_d)$ and $e(X,\xi_d)$ defined locally around $(\underline{t},\underline{x},\underline{\eta},\underline{\xi_d})$ and $(\underline{X},\underline{\xi}_d)$, respectively, such that they are holomorphic in $\xi_d$, 
                \begin{equation}\label{EigenvaluesSymbol}
                    \pi_P(X,\xi_d)=e(X,\xi_d)(\tau-i\gamma+\lambda(t,x,\eta,\xi_d))^{\alpha},
                \end{equation}
                and $e(X,\xi_d)$ is nonvanishing at $(\underline{X},\underline{\xi_d})$. Moreover,  $\lambda$ is real when $\xi$ is real and there is a smooth matrix-valued function $\Pi(X)$ on a neighborhood of $\underline{X}$, holomorphic with respect to $\xi_d$, such that $\Pi$ is a projector of rank $\alpha$ for which $\ker(\xi_dI_n+a(X))=\Pi(X)\C^n$ when $\tau-i\gamma+\lambda(t,x,\eta,\xi_d)=0$. 
            \end{enumerate}
        \end{Definition}
        Whenever $\alpha=1$ in \eqref{EigenvaluesSymbol}, we say that $P_{\gamma}$ is \textit{strictly hyperbolic}.
        \begin{Assumption}\label{Assumption2} \phantom{}
            \begin{enumerate}[label=\normalfont(\roman*)]
                \item \label{Assumption2ONe} $P_{\gamma}$ is hyperbolic as in Definition \ref{Hyperbolicity}.
                \item \label{Assumption2Two} The symbols $a(X)$ and $b(X)$ are independent of $(t,x)$ outside a certain compact set $K$. We shall designate this as property \textbf{(C)}. 
                \item $b(X)$ is everywhere of maximal rank $p=\dim\mathbb{E}^{-}(X)$.
            \end{enumerate}
        \end{Assumption}
        We now outline the main results of this paper. To start with, we explore the model problem 
         \begin{equation}\label{ModelProblemIntro}
            \left\{\begin{aligned}
                 P_{\gamma}u(t,x)\eqdef\left(D_{d}+\mathdutchcal{A}_{\gamma}(D_t,D_y)\right)u(t,x)&=f(t,x) &\quad &\quad (t,x)\in \R^{1+d}_{+},\\
                 Bu|_{x_d=0}(t,y)&=g(t,y) &\quad &\quad (t,y)\in \R^{d},
            \end{aligned}\right.
        \end{equation}
        subject to Assumptions \ref{Assumption1} and \ref{Assumption2}. Although we obtain energy estimates for \eqref{ModelProblemIntro} that are essentially the same as those in \cite{serre2005solvability} and \cite{benzoni-gavage_serre_2007}, our approach involving the construction of a symmetrizer $\Sigma_{\gamma}$ is more robust,
        firstly because it extends the validity of \eqref{EE_BGS} beyond the forward cone $\Gamma$ (see \cite{benzoni-gavage_serre_2007}, Chapter 8), and secondly because it unveils a necessary condition to keep in mind when generalizing $\Sigma_{\gamma}$ to variable coefficients. In brief, we prove the existence of two families of pseudo-differential operators, namely $\Delta_{\gamma}=\Op_{\gamma}(\delta)$ and $\Sigma_{\gamma}=\Op_{\gamma}(\sigma)$, so that:
        \begin{itemize}[label=$\circ$]
            \item the set of points where the \Lop condition fails is included in the characteristic set of $\Delta_\gamma$,
            \item  $\Sigma_{\gamma}$ is self-adjoint,
            \item  if 
            \begin{equation*}
                L_{\Delta}\eqdef\{v \in \mathcal{S}'(\R^{1+d}_+,\,\R^n): \Delta_{\gamma} v \in L_{\gamma}^2(\R^{1+d}_+,\,\C^n)\},
            \end{equation*}
            then for every $v_1,v_2\in L^2_{\Delta}$ there exists a positive constant $C$ satisfying 
            \begin{equation*}
                \langle{\Sigma_{\gamma} v_1,v_2}\rangle\leq C\vert{\Delta_{\gamma} v_1}\vert\vert{\Delta_{\gamma} v}\vert,
            \end{equation*}
            \item  there is a positive constant $c$ such that
            \begin{equation*}
                \ima{\langle\Sigma_{\gamma}\mathdutchcal{A}_{\gamma}v,v\rangle}\geq c\gamma \vert{\Delta_{\gamma} v}\vert^2
            \end{equation*}
            for each $v\in L^2_{\Delta}$,
            \item there exists an operator $Q_{\gamma}$ and positive constants $\alpha,\beta$ for which 
            \begin{equation*}
                \langle{\Sigma_{\gamma}v(0),v(0)}\rangle\geq \alpha\vert{\Delta_{\gamma} v(0)}\vert^2-\beta \vert{Q_{\gamma}Bv(0)}\vert^2
            \end{equation*}
            holds true for every $v\in L^2_{\Delta}$.
        \end{itemize}
        
        For $\WR$ problems, it is well known that the intersection of the stable subspace and the kernel of $B$ at a point $\zeta_0$ where the \Lop condition fails is a one-dimensional subspace of $\C^n$. We call this line the \textit{critical direction} and denote it by $\ell(\zeta_0)$. Interestingly, Proposition \ref{KrilovDegeneracy} explains how $\Sigma_{\gamma}$ degenerates around $\ell$: if $a(\zeta)$ is the principal symbol of $A(D_t,D_y)$ and $\underline{\Delta}(\zeta_0)=0$, then  $v\mapsto\langle{\sigma(\zeta_0)v,v}\rangle$ vanishes on the Krylov space of $\ell(\zeta_0)$ with respect to $a(\zeta_0)$, i.e., on the smallest invariant subspace of $a(\zeta_0)$ containing $\ell(\zeta_0)$. 
        
        The previous case study prepares the ground to attack Problem \eqref{(P,B)_Dif} in its full \textcolor{black}{generality}. In fact, if we suppose that $(P_{\gamma},B_{\gamma})$ satisfies the definition of the $\WR$ class, we show that there exist symbols $\tilde{\ee}_0$ and $\delta$ with the following features:
        \begin{itemize}[label=$\circ$]
            \item $\tilde{\ee}_0$ and $\delta$ are homogeneous of degree $0$,
            \item $\tilde{\ee}_0\in GL_n(\C)$ and 
            \begin{equation*}
                \dot{a}_1\eqdef \tilde{\ee}_0^{-1}a_1\tilde{\ee}_0
            \end{equation*}
            is diagonal with entries $a_{1,1},\,\cdots, a_{1,n}$,
            \item there exits $\mathdutchcal{s}\leq p$ so that $\delta$ is diagonal with respect to the basis $\tilde{\ee}_0$ and given by
            \begin{equation}
                \delta=\mathrm{diag}(\delta^{-}, I_{n-p}),
            \end{equation}
            where
            \begin{equation*}
                \delta^{-} \eqdef\begin{pNiceMatrix}
                    \begin{pmatrix}
                \delta_{1}^{-}&\hdots&0\\
                \vdots&\ddots&\vdots\\
                0&\hdots&\delta_{\mathdutchcal{s}}^{-}\\
            \end{pmatrix}&      \\
                    &I_{p-\mathdutchcal{s}}
                \end{pNiceMatrix}
            \end{equation*}    
            is such that each entry $\delta^{-}_j$ is the solution of the transport equation
            \begin{align}\label{TransportEqIntro}
                \left\{\begin{aligned}
     		       \partial_{d}\delta_{j}^{-}+\{\delta_{j}^{-},a_{1,j}\}&=0,\\
     		       \delta_{j}^{-}|_{x_d=0}&=\underline{\Delta}, \\
		        \end{aligned}\right.
            \end{align}  
            \item if $x_d=0$, there exist smooth matrix-valued functions $q$ and $m$ with dimensions $p\times p$ and $p\times n$, respectively, so that if $\dot{b}\eqdef b\tilde{\ee}_0$, then
            \begin{equation}
                q\dot{b}=m\delta,
            \end{equation}
            \item when nontrivial, $\ker{\delta}$ is an $\mathdutchcal{s}-$dimensional invariant subspace of $a_1$ containing the critical direction $\ell$.
        \end{itemize}
        If $\Delta_{\gamma}=\Op_{\gamma}(\delta)$ and $Q_{\gamma}=\Op_{\gamma}(q)$, we claim that it is possible to find a family $\Sigma_{\gamma}$ of $C^1$ operator-valued mappings depending on $x_d$ so that for $\gamma$ sufficiently large,
        \begin{itemize}[label=$\circ$]
            \item  $\Sigma_{\gamma}(x_d)$ is self-adjoint.
            \item for every $v_1,v_2\in L^2_{\Delta}$, there is a positive constant $C$ satisfying 
            \begin{equation*}
                \langle{\Sigma_{\gamma}(x_d) v_1,v_2}\rangle\leq C\vert{\Delta_{\gamma}(x_d) v_1}\vert\vert{\Delta_{\gamma}(x_d) v}\vert.
            \end{equation*}
            \item There is a positive constant $c$, independent of $x_d$, such that
            \begin{equation*}
                \langle{\partial_d \Sigma_{\gamma}(x_d) v,v}\rangle+2\ima{\langle\Sigma_{\gamma}(x_d)\mathdutchcal{A}_{\gamma}(x_d)v,v\rangle}\geq c\gamma \vert{\Delta_{\gamma}(x_d) v}\vert^2
            \end{equation*}
            for each $v\in L^2_{\Delta}$. 
            \item There exist positive constants $\alpha$ and $\beta$ for which
            \begin{equation*}
                \langle{\Sigma_{\gamma}(0)v(0),v(0)}\rangle\geq \alpha\vert{\Delta_{\gamma}(0) v(0)}\vert^2-\beta \vert{Q_{\gamma} B_{\gamma}v(0)}\vert^2
            \end{equation*}
            holds true for every $v\in L^2_{\Delta}$.
        \end{itemize}
        In contrast to the standard case where the uniform \Lop condition is satisfied, lower order terms in a $\WR$ problem are problematic and require some attention. To deal with them, we have to resort to finer tools, including some technical lemmas and the following statement.
        \begin{Proposition}
            The norms $\Vert{\Delta_{\gamma}\,\cdot\,}\Vert$ lie between $L_{\gamma}^2$ and $H_{\lambda}^{-1}$ for $\gamma$ larger than $\gamma_0\geq 1$, that is to say, there exist positive constants $C_1$ and $C_2$ such that  
            \begin{equation}
                C_1\Vert{\,\cdot\,}\Vert_{\gamma,-1}\leq\frac{1}{\gamma_0} \Vert{\Delta_{\gamma}\,\cdot\,}\Vert \leq C_2\Vert{\,\cdot\,}\Vert_{\gamma},
            \end{equation}
            for every $\gamma\in (\gamma_0,+\infty)$.
        \end{Proposition}
        Due to Proposition \ref{DeltaHminus}, we are able to recover Coulombel's estimates with loss of one derivative. Overall, by combining $\Sigma_{\gamma}$ with a partition of unity, we can prove the main theorem of this \textcolor{black}{work.} 
        \begin{Theorem}\label{main}
            Let
            \begin{equation*}
            \left\{\begin{aligned}
                     P_{\gamma}u_{\gamma}(t,x)\eqdef\left(D_{d}+\mathdutchcal{A}_{\gamma}(t,y,x_d,D_t,D_y)\right)u(t,x)&=f(t,x) &\quad &\quad (t,x)\in \R^{1+d}_{+},\\
                     B_{\gamma}(t,y)u(t,y)|_{x_d=0}&=g(t,y) &\quad &\quad (t,y)\in \R^{d},
                \end{aligned}\right.
            \end{equation*}
            where  $\mathdutchcal{A}_{\gamma} \in \mathrm{OPS}_\gamma^{1}(\R^{1+d}_{+}\times [1, +\infty))$ and $B_{\gamma}\in \mathrm{OPS}_\gamma^{0}(\R^{d}\times [1, +\infty))$ are classical pseudodifferential operators with matrix-valued symbols $a(X)$ and $b(X)$ of dimensions $n\times n$ and $p\times n$, respectively. Assume that $P_{\gamma}$ is hyperbolic in the sense of Definition \ref{Hyperbolicity}, $P_{\gamma}$ and $B_{\gamma}$ satisfy Property \textbf{(C)}, and $p=\dim\mathbb{E}^{-}(X)$. Then there exist
            \begin{enumerate}[label=\normalfont(\roman*)]
               \item $\gamma_0\geq 1$,
               \item a family of pseudodifferential operators $\Delta_{\gamma}(t,x,D_t,D_y)\in\mathrm{OPS}_\gamma^{0}(\R^{1+d}_{+}\times [\gamma_0, +\infty))$,
               \item function spaces 
                   \begin{align*}
                      L^2_{\Delta}&\eqdef\{v \in \mathcal{S}'(\R^{1+d}_+,\,\R^n): \Delta_{\gamma}v \in L_{\gamma}^2(\R^{1+d}_+,\,\C^n)\},\\
                      H^s_{\Delta}&\eqdef\{v \in \mathcal{S}'(\R^{1+d}_+,\,\R^n): \Lambda^s_{\gamma} v \in L_{\Delta}^2(\R^{1+d}_+,\,\C^n)\},
                   \end{align*} 
               \item and a positive constant $C$ such that,
            \end{enumerate}
            % \begin{enumerate}[label=\normalfont(\alph*)]
            %     \item a
            %     \item b
            % \end{enumerate}
            if $f\in L_{\gamma}^2(\R^{1+d}_{+})$ and $g\in L_{\gamma}^{2}(\R^{d})$, then for all $\gamma\geq \gamma_0$ and every $u\in \mathcal{D}(\R^{1+d}_{+})$ the following estimate holds
            \begin{align}
                \gamma\Vert{\Delta_{\gamma} u }\Vert_{0,\gamma}^2+\vert{\Delta_{\gamma} u(0)}\vert_{0,\gamma}^2&\leq C\left(\frac{1}{\gamma}\Vert{f}\Vert_{0,\gamma}^2+\vert{g}\vert_{0,\gamma}^2\right).
            \end{align}
            More generally, if $f\in H_{\gamma}^s(\R^{1+d}_{+})$ and $g\in H_{\gamma}^{s}(\R^{d})$,
            \begin{align}
                \gamma\Vert{\Delta_{\gamma} u }\Vert_{s,\gamma}^2+\vert{\Delta_{\gamma} u(0)}\vert_{s,\gamma}^2&\leq C\left(\frac{1}{\gamma}\Vert{f }\Vert_{s,\gamma}^2+\vert{g}\vert_{s,\gamma}^2\right).
            \end{align}
            
            Suppose, in addition, that the nullity of the principal symbol of $\Delta_{\gamma}$ is independent of $X\in \mathbb{X}_{S}$. If $f\in L_{\Delta}^2(\R^{1+d}_{+})$ and $g\in L_{\Delta}^{2}(\R^{d})$, there exists a pseudodifferential operator $Y_{\gamma}(t,y,D_t,D_y)\in \mathrm{OPS}_\gamma^{0}(\R^{d}_{+}\times [\gamma_0, +\infty))$ so that, for all $\gamma\geq \gamma_0$ and every $u\in \mathcal{D}(\R^{1+d}_{+})$,
            \begin{align}
                \gamma\Vert{\Delta_{\gamma} u }\Vert_{0,\gamma}^2+\vert{\Delta_{\gamma} u(0)}\vert_{0,\gamma}^2&\leq C\left(\frac{1}{\gamma}\Vert{\Delta_{\gamma}f}\Vert_{0,\gamma}^2+\vert{\Delta_{\gamma}Y_{\gamma}g}\vert_{0,\gamma}^2\right).
            \end{align}
            More generally, if $f\in H_{\Delta}^s(\R^{1+d}_{+})$ and $g\in H_{\Delta}^{s}(\R^{d})$,
            \begin{align}
                \gamma\Vert{\Delta_{\gamma} u }\Vert_{s,\gamma}^2+\vert{\Delta_{\gamma} u(0)}\vert_{s,\gamma}^2&\leq C\left(\frac{1}{\gamma}\Vert{\Delta_{\gamma}f }\Vert_{s,\gamma}^2+\vert{\Delta_{\gamma}Y_{\gamma}g}\vert_{s,\gamma}^2\right).
            \end{align}
            % In addition, if the kernel of the principal symbol $\sigma(\Delta_{\gamma})$ is constant as $X$ ranges on $\in\mathbb{X}_{S}$, then
            
            % \textcolor{blue}{In particular, if the dimension of the Krylov space is constant, $f\in L_{\Delta}^2(\R^{1+d}_{+})$, and $g\in L_{\Delta}^{2}(\R^{d})$, there exists $Y_{\gamma}(t,y,D_t,D_y)\in \mathrm{OPS}_\gamma^{0}(\R^{d}_{+}\times [\gamma_0, +\infty))$ such that for all $\gamma\geq \gamma_0$ and every $u\in \mathcal{D}(\R^{1+d}_{+})$, the following estimate holds
            % \begin{align}
            %     \gamma\Vert{\Delta_{\gamma} u }\Vert_{0,\gamma}^2+\vert{\Delta_{\gamma} u(0)}\vert_{0,\gamma}^2&\leq C\left(\frac{1}{\gamma}\Vert{\Delta_{\gamma}f}\Vert_{0,\gamma}^2+\vert{\Delta_{\gamma}Y_{\gamma}g}\vert_{0,\gamma}^2\right).
            % \end{align}
            % More generally, if $f\in H_{\gamma}^s(\R^{1+d}_{+})$ and $g\in H_{\gamma}^{s}(\R^{d})$,
            % \begin{align}
            %     \gamma\Vert{\Delta_{\gamma} u }\Vert_{s,\gamma}^2+\vert{\Delta_{\gamma} u(0)}\vert_{s,\gamma}^2&\leq C\left(\frac{1}{\gamma}\Vert{\Delta_{\gamma}f }\Vert_{s,\gamma}^2+\vert{\Delta_{\gamma}Y_{\gamma}g}\vert_{s,\gamma}^2\right).
            % \end{align}
            % }

            % Y_{\gamma}(t,y,D_t,D_y)&\in \mathrm{OPS}_\gamma^{0}(\R^{d}_{+}\times [\gamma_0, +\infty))
       \end{Theorem}
\section{Recap on the WR class}\label{Recap}
    \subsection{The block structure condition.}
        We begin with a fundamental idea in the construction of symmetrizers:  the \textit{block structure condition}. This notion was originally introduced by Kreiss in \cite{kreiss1970initial} for the strictly hyperbolic case, and later adapted by Métivier in \cite{metivier_2000} to the wider class of constantly hyperbolic operators.
        \begin{Definition}\label{BSC}
            Let $\underline{X}\in \mathbb{X}$. The matrix $a(X)$ verifies the block structure condition if there exists a neighborhood $\mathcal{V}$ of $\underline{X}$ in $\mathbb{X}$, an integer $\mathdutchcal{q}\geq 1$, a partition $n=\nu_1+\,\cdots+\nu_\mathdutchcal{q}$ with $\nu_i\geq 1$, and a smooth nonsingular map $\ee_0(X)$ defined on $\mathcal{V}$ such that for every $X\in \mathcal{V}$
            \begin{equation*}
                \ee_0^{-1}(X)a(X)\ee_0(X)=\mathrm{diag}(a_{1}(X),\,\cdots, a_{\mathdutchcal{q}}(X)),
            \end{equation*}
            with blocks $a_{k}(X)$ of size $\nu_k\times\nu_k$ that fall exactly into one of the following categories:
            \begin{enumerate}[label=\normalfont(\roman*)]
                \item \label{one} The spectrum of $a_{k}(X)$ is contained in $\C\setminus\R$.
                \item \label{two} $\nu_j=1$, $a_{j}(X)\in \R$ when $\gamma=0$, and $i\partial_{\gamma}{a_{j}(X)}\in \R\setminus\{0\}$.
                \item \label{three} $\nu_j>1$, $a_{j}(X)$ has real coefficients when $\gamma=0$, and there is $\mu_j\in \R$ such that
                \begin{equation*}
                    a_{j}(\underline{X})=\begin{pNiceMatrix}
                        \mu_j & 1 & 0 & \cdots & 0 \\
                        0 & \mu_j & 1 & \ddots & \vdots \\
                        \vdots& \ddots & \ddots & \ddots &  0 \\
                        \vdots & &\ddots &\ddots &  1 \\
                        0 & \cdots & \cdots& 0 &  \mu_j
                    \end{pNiceMatrix}.   
                \end{equation*}
            \end{enumerate}
            Additionally, the entry at the lower left corner of $i\partial_{\gamma}a_{k}(X)$ is nonvanishing and real.
        \end{Definition} 
        \begin{Proposition}\label{BSCProp}
            Let $\underline{X}\in \mathbb{X}$. If $a(X)$ is hyperbolic as in Definition \ref{Hyperbolicity}, then $a(X)$ satisfies the block structure condition.
        \end{Proposition}
        The proof of this fact is lengthy and highly technical, so we feel free to skip it. If interested, the reader may find useful the discussions in Chapter 7 in \cite{Chazarain2011} (for strictly hyperbolic operators), Chapter 5 in \cite{benzoni-gavage_serre_2007}, and the remarkable paper \cite{metivier_2000} by Métivier.
    
        Based on Definition \ref{BSC} and Proposition \ref{BSCProp}, $\mathbb{X}_0$ may be divided into four regions as indicated next.
        \begin{Definition}\label{Classification}
            \phantom{1}
            \begin{itemize}[label=$\circ$]
                \item The set of elliptic points $\mathcal{E}$ consists of those $X\in \mathbb{X}_0$ for which Definition \ref{BSC} is satisfied with blocks of type \ref{one} exclusively (complex conjugate pairs).
                \item The set of hyperbolic points $\mathcal{H}$ consists of those $X\in \mathbb{X}_0$ for which Definition \ref{BSC} is satisfied with blocks of type \ref{two} exclusively.
                \item The set of mixed points $\mathcal{E}\mathcal{H}$ consists of those $X\in \mathbb{X}_0$ for which Definition \ref{BSC} is satisfied with blocks of type \ref{one} and block \ref{two}, but no blocks of type \ref{three}.
                \item The set of glancing points $\mathcal{G}$ consists of those $X\in \mathbb{X}_0$ for which Definition \ref{BSC} is satisfied with at least one block of type \ref{three}.
            \end{itemize}
        \end{Definition}
        An interesting consequence of the block structure condition is given hereunder. 
        % \textcolor{red}{A proof can be found in \cite{metivier_2000}}. 
        \begin{Proposition}\label{Continuation}
            The stable subspace $\mathbb{E}^{-}(X)$ defines a smooth vector bundle over $\mathbb{X}\cap\{\gamma>0\}$ that extends into a continuous vector bundle (again denoted by $\mathbb{E}^-(X)$) over $\mathbb{X}$ with the same rank.  
        \end{Proposition} 
         According to Definition \ref{Classification}, the hyperbolic region $\mathcal{H}$ comprises the elements $X\in \mathbb{X}$ such that $a(X)$ is diagonalizable with purely real eigenvalues, meaning that $\mathcal{H}$ is necessarily confined to the frequency boundary $\{\gamma=0\}$ since $a(X)$ is known to have no real eigenvalues when $\gamma>0$ (see Hersh's lemma). 
         
         We complement Proposition \ref{Continuation} with a definition and a result that we state without proof (see Lemma 3.1 in \cite{coulombel2010geometric}).
         \begin{Definition}
            A complex vector space is of real type if it possesses a basis consisting entirely of real vectors.
        \end{Definition}
        \begin{Proposition}\label{ERealType}
            Let $X \in \mathbb{X}_0$. If $X\in \mathcal{H}$, then $\mathbb{E}^{-}(X)$ is of real type.
        \end{Proposition}
    \subsection{The \Lop condition.}   
        We open this section with a central concept for the forthcoming discussion.
        \begin{Definition}\label{LopCond}
            A boundary value problem $(P_{\gamma},B_{\gamma})$ satisfies the weak \Lop condition if for every $X \in \mathbb{Y}_S\cap \{\gamma>0\}$, 
            \begin{align*}
                 \mathbb{E}^{-}(X)\cap \ker b(X)=\{0\}.
            \end{align*}
        \end{Definition}
        The weak \Lop condition turns out to be necessary for the $H^s$ well-posedness of \eqref{SystemPDO}, albeit in general with loss of derivatives. As a matter of fact, such a loss of regularity can be avoided in the presence of a stronger statement: the uniform \Lop condition (see Theorem \ref{ExistenceUniquenessRegularity} below).
        \begin{Definition}\label{UnifLopCond}
            A boundary value problem $(P_{\gamma},B_{\gamma})$ satisfies the uniform \Lop condi-\\tion if for every $X \in \mathbb{Y}_S$,
            \begin{align}\label{LopEquation}
                 \mathbb{E}^{-}(X)\cap \ker b(X)=\{0\},
            \end{align}
            that is to say, equation \eqref{LopEquation} holds up to the frequency boundary \footnote{This extension relies on Proposition \ref{Continuation}.} $\{\gamma= 0\}$.
        \end{Definition}
        In practice, we do not directly verify Definitions \ref{LopCond} and \ref{UnifLopCond} when checking the \Lop condition for a given $(P_{\gamma},B_{\gamma})$, but use a special map defined as 
        \begin{equation}\label{LopDet}
            \underline{\Delta}(X)\eqdef \det\left(b(X)\textbf{e}_1(X),\,\cdots, b(X)\textbf{e}_p(X)\right),
        \end{equation}
        where $\textbf{e}(X)=\{\textbf{e}_1(X), \, \cdots, \textbf{e}_p(X)\}$ is a smooth, homogeneous basis for $\mathbb{E}^{-}(X)$. This function is called the \textit{\Lop determinant} and has the following properties:
        \begin{itemize}[label=$\circ$]
            \item for $\gamma>0$, $X\mapsto\underline{\Delta}(X)$ is smooth with respect to $(t,x)$, holomorphic in $\tau-i\gamma$ and real analytic in $\eta$,
            \item $\underline{\Delta}$ is homogeneous of degree $0$,
            \item $\underline{\Delta}(X)$ vanishes only at points $X\in \mathbb{X}$ where the \Lop condition is violated.
        \end{itemize}

        \begin{Theorem}[Theorem 6.10, \cite{Chazarain2011}]\label{ExistenceUniquenessRegularity}
            Consider the boundary value problem \eqref{SystemPDO} under Assumption \ref{Assumption1}. If the uniform \Lop condition is satisfied, it is possible to find a constant $\gamma_0\geq 1$ such that the following assertion holds for every $\gamma\geq \gamma_0$\,: if $f\in L^2_{\gamma}(\R_{+}^{1+d})$ and $g(t,x)\in L^2_{\gamma}(\R^{d})$, there is a unique $u(t,x)\in L^2_{\gamma}(\R_{+}^{1+d})$ of \eqref{SystemPDO} with $u|_{x_d=0}$ in $L^2_{\gamma}(\R^d)$ so that $u$ satisfies the energy estimate
            \begin{equation*}
                \gamma \Vert{u}\Vert^2_{0,\gamma}+\vert{u(0)}\vert^2_{0,\gamma}\leq c\left(\frac{1}{\gamma}\Vert{f}\Vert^2_{0,\gamma}+\vert{g}\vert^2_{0,\gamma}\right).
            \end{equation*}
            for some $C$ which only depends on $\gamma_0$. Moreover, there exists $\gamma_k\geq \gamma_0$ so that if $\gamma\geq \gamma_k$, it is true that for $f\in H^k_{\gamma}(\R_{+}^{1+d})$ and $g\in H^k_{\gamma}(\R^{d})$, there is a unique solution $u\in H^k_{\gamma}(\R_{+}^{1+d})$ of \eqref{SystemPDO} whose trace on $\R^d$ belongs to $H^k_{\gamma}(\R^d)$, and so that $u$ satisfies 
            \begin{equation*}
                \gamma \Vert{u}\Vert^2_{s,\gamma}+\vert{u(0)}\vert^2_{s,\gamma}\leq c\left(\frac{1}{\gamma}\Vert{f}\Vert^2_{s,\gamma}+\vert{g}\vert^2_{s,\gamma}\right).
            \end{equation*}
        \end{Theorem}
  
    \subsection{The $\mathcal{WR}$ class.}    
        We now proceed to define the $\WR$ class and examine its main properties.
        \begin{Definition}\label{WR}
            The boundary value problem $(P_{\gamma},B_{\gamma})$ is of class $\WR$ if the following conditions are met:
            \begin{enumerate}[label=$\circ$]
                \item The weak \Lop condition holds, 
                \item The level set $\underline{\Delta}^{-1}(0)$ is non-void and contained in the hyperbolic region $\mathcal{H}$. Moreover, $\partial_{\tau}\underline{\Delta}(X)\neq 0$ whenever $\underline{\Delta}(X)=0$.
            \end{enumerate}
        \end{Definition}
        We shall utilize an equivalent characterization of the the $\WR$ class that fits better our purposes.
        \begin{Proposition}\label{WR_2.2}
            The pair $(P_{\gamma},B_{\gamma})$ defines a $\WR$ boundary value problem if and only if:
            \begin{enumerate}[label=$\circ$]
                \item \label{WLC} for every $X\in \mathbb{Y}_S\cap\{\gamma>0\}$, it is true that $\mathbb{E}^{-}(X)\cap \ker b(X) =\{0\}$. In other words, the weak \Lop condition is fulfilled.
                \item \label{WR_condition} The critical set $\Gamma\eqdef\{X\in \mathbb{Y}_S: \mathbb{E}^{-}(X)\cap \ker b(X)\neq \{0\}\}$ is nonempty and included in the hyperbolic region $\mathcal{H}$. Furthermore, for every $\underline{X}\in \Gamma$, there exist a conic neighborhood $\mathcal{V}$ of $\underline{X}$, and mappings
                \begin{itemize}[label=$\triangleright$]
                    \item $\underline{e}_1, \, \cdots, \underline{e}_{p} \in  C^{\infty}(\mathcal{V},\C^{n})$, 
                    \item $\mathfrak{p}(X)\in C^{\infty}(\mathcal{V},\textrm{GL}_p(\C))$,
                    \item $\omega\in C^{\infty}(\mathcal{V},\R)$ 
                \end{itemize}
                such that, for every $X\in \mathcal{V}$, $\underline{\ee}^{-}(X)=\{\underline{e}_1(X), \, \cdots, \underline{e}_{p}(X)\}$ is a basis for $\mathbb{E}^{-}(X)$, 
                \begin{equation*}
                    b^{-}(X)\eqdef b(X)(\underline{e}_1(X),\,\cdots, \underline{e}_p(X))=\mathfrak{p}(X)\begin{pNiceMatrix}
                    \lambda^{-1}(\zeta)(\gamma+i\omega(X))&0&\cdots&0\\
                    0&1&\cdots&0 \\
                    \vdots&\vdots &\ddots &\vdots&\\
                    0& 0&\cdots &1\\
                \end{pNiceMatrix}.
                \end{equation*}
                and $\partial_{\tau}\omega(X)\neq0$ when $\omega(X)=0$.
            \end{enumerate}
        \end{Proposition}
        The latter equivalence corresponds to Proposition B.1 in \cite{coulombel2010geometric}, whose proof we complete by including the construction of $\underline{\ee}^{-}(X)$ following the analysis in \cite{sable1988existence} and \cite{ohkubo1975structures}. 
    
        \begin{Proposition}\label{STBasis} If $(P_{\gamma},B_{\gamma})$ belongs to the $\WR$ class, there exist a neighborhood $\mathcal{V}$ and two mappings $\mathfrak{p},c\in C^{\infty}(\mathcal{V}, GL_p(\C))$ homogeneous of degree $0$ in $\zeta=(\tau-i\gamma,\eta)\simeq(\tau,\eta,\gamma)$, such that the factorization 
        \begin{equation}\label{B_minus}
            b^{-}(X)=\mathfrak{p}(X)\begin{pNiceMatrix}
                \lambda^{-1}(\zeta)\left(\gamma+i\omega(X)\right)&0&\cdots &0\\
                0&1&\cdots &0\\
                \vdots & \vdots &\ddots &\vdots \\
                0&0&\cdots &1\\
                \end{pNiceMatrix}c^{-1}(X)
        \end{equation}
        holds true.
        \end{Proposition}
        \begin{proof}
            Let 
            \begin{equation*}
                \Gamma=\{X\in \mathbb{Y}_S: \mathbb{E}^{-}(X)\cap \ker b(X)\neq \{0\}\}=\{X\in \mathbb{Y}_S: \underline{\Delta}(X)=0\},
            \end{equation*}
            and suppose that $\underline{X}\in \Gamma$ is such that $\underline{\Delta}(\underline{X})=0$ and $\partial_{\tau}\underline{\Delta}(\underline{X}) \neq 0$. Since $\underline{\Delta}(X)$ is holomorphic in $\rho\eqdef\tau-i\gamma$ and homogeneous of degree $0$ with respect to $\zeta$, the implicit function theorem characterizes the zeros of $\underline{\Delta}(X)$ in a conic neighborhood $\mathcal{V}$ of $\underline{X}$ through an equation $\rho=\nu(t, y, \eta)$, where $\nu(t, y, \eta)$ is  a smooth, homogeneous function of degree 1 in $\eta$. If $\ee(X)=\{e_1(X),\,\cdots,e_p(X)\}$ is any basis for $\mathbb{E}^{-}(X)$, set 
            \begin{equation*}
                b^{-}(X)\eqdef (b(X)e_1(X),\,\cdots, b(X)e_p(X))
            \end{equation*}
            and denote each $b(X)e_i(X)$ by $b^{-}_i(X)$ in the sequel. We claim that $b^{-}(X)$ has a nonsingular cofactor matrix $h(X)$ of order $p-1$ on $\mathcal{V}$. Indeed, shrinking $\mathcal{V}$ if necessary, $\partial_{\tau}\underline{\Delta}(\underline{X}) \neq 0$ guarantees that  
            \begin{equation}\label{rank}
                \textrm{rank}\,{{b}^{-}(X)}\geq  p-1 \quad \mathrm{on} \quad \mathcal{V},
            \end{equation}
            the equality being realized only when $X\in \mathcal{V}\cap \Gamma$. In practice, there is no loss of generality in assuming that $h(X)$ is the resulting block after deleting the first column and the first row of $b^{-}(X)$. Let $b'_{1}(X)$ be the vector $b^{-}_1(X)$ without its first entry. In such case, for every $X\in \mathcal{V}$, the linear system $h(X)d(X)=b'_{1}(X)$ possesses a unique solution
            \begin{equation*}
                \begin{pNiceMatrix}
                    d_2(X)\\
                    \vdots\\
                    d_p(X)
                \end{pNiceMatrix}=h^{-1}(X){b}'_{1}(X),
            \end{equation*}
            whose entries $d_2(X),\,\cdots, d_p(X)$ are smooth in $(t,x)$ and homogeneous of degree $0$ in $(\rho, \eta)$. Suppose now that
            \begin{equation*}
                k(X)= b^{-}_1(X)-\sum_{i=2}^{p}{d_i(X)b^{-}_i(X)}.
            \end{equation*}
            Except for the first component of $k(X)$, which only vanishes when $X\in\{\rho=\nu(t, y, \eta)\}$, all other entries are identically zero by construction.
            Thus, owing to the Weierstrass preparation theorem, there exists a nonvanishing function $z_1(X)\equiv z(t,y,\rho,\eta)$ in $\mathcal{V}$, analytic with respect to $\rho$ and homogeneous of degree $-1$ in $(\rho,\eta)$, such that
            \begin{equation*}
                k_1(X)=(\rho-\nu(t, y, \eta))z_1(X) \quad \textrm{for} \quad X\in\mathcal{V}.
            \end{equation*}
            As a result, if we put
            \begin{equation*}
                z(X)\eqdef\begin{pNiceMatrix}
                    z_1(X)\\
                    \vdots\\
                    z_p(X)
                \end{pNiceMatrix},
            \end{equation*}
            with $z_2(X)=\cdots=z_p(X)=0$ along $\mathcal{V}$, we can write $k(X)=(\rho-\nu(t, y, \eta))z(X)$ and with it, 
            \begin{equation*}
                b^{-}_1(X)=(\rho-\nu(t, y, \eta))z(X)+\sum_{i=2}^{p}{d_i(X)b^{-}_i(X)}.
            \end{equation*}
            Let
            \begin{equation}
                c(X)=\begin{pNiceMatrix}
                    1&0&\cdots &0\\
                    -d_2(X)&1&\cdots &0\\
                    \vdots & \vdots &\ddots &\vdots \\
                    -d_{p}(X)&0&\cdots &1\\
                \end{pNiceMatrix}.
            \end{equation} 
            Observe that $c(X)$ is homogeneous of order $0$ in $(\rho,\eta)$ and that
            \begin{align*}
                b^{-}(X)&=\begin{pNiceMatrix}
                            b^{-}_1(X) &b^{-}_2(X)&\cdots& b^{-}_p(X)
                            \end{pNiceMatrix}\\&=\begin{pNiceMatrix}
                            (\rho-\nu(t, y, \zeta))z(X) &b^{-}_2(X)&\cdots& b^{-}_p(X)
                            \end{pNiceMatrix}c^{-1}(X),
            \end{align*}
            after a straightforward computation. Notably, the latter amounts to writing in matrix notation 
            \begin{align*}
                b^{-}(X)&=\mathfrak{p}(X)\begin{pNiceMatrix}
                    \lambda^{-1}(\zeta)i(\rho-\nu(t, y, \eta))&0&\cdots &0\\
                    0&1&\cdots &0\\
                    \vdots & \vdots &\ddots &\vdots \\
                    0&0&\cdots &1\\
                \end{pNiceMatrix}c^{-1}(X)\\&=\mathfrak{p}(X)\begin{pNiceMatrix}
                    \lambda^{-1}(\zeta)(\gamma+i(\tau-\nu(t, y, \eta)))&0&\cdots &0\\
                    0&1&\cdots &0\\
                    \vdots & \vdots &\ddots &\vdots \\
                    0&0&\cdots &1\\
                \end{pNiceMatrix}c^{-1}(X),
            \end{align*}
            where
            \begin{equation*}
                \mathfrak{p}(X)\eqdef \begin{pNiceMatrix}
                            -iz(X)\lambda(\zeta) &b^{-}_2(X)&\cdots& b^{-}_p(X)
                            \end{pNiceMatrix}
            \end{equation*}
            is homogeneous of degree $0$ in $\zeta$ and nonsingular. Lastly, setting $\omega(t,y,\tau,\eta) \eqdef \tau-\nu(t, y, \eta)$, it is easily seen that $\partial_{\tau}\omega(t,y,\tau,\eta)\neq 0$ and that 
            \begin{align}\label{Kerb}
                b^{-}(X)&=\mathfrak{p}(X)\begin{pNiceMatrix}
                    \lambda^{-1}(\zeta)(\gamma+i\omega(t,y,\tau,\eta)))&0&\cdots &0\\
                    0&1&\cdots &0\\
                    \vdots & \vdots &\ddots &\vdots \\
                    0&0&\cdots &1\\
                \end{pNiceMatrix}c^{-1}(X).
            \end{align}
        \end{proof}
        \begin{Corollary}
            Let $\omega(X)\equiv\omega(t,y,\tau,\eta)$ and $\mat{\Delta}(X)\eqdef(\gamma+i\omega(X))/\lambda(\zeta)$. Under the assumptions of Proposition \ref{STBasis}, there is a basis 
            \begin{equation*}
                \underline{\ee}^{-}(X)=\{\underline{e}_1(X),\,\cdots,\underline{e}_p(X)\},
            \end{equation*}
            for which  
            \begin{equation}
                b^{-}(X)=\mathfrak{p}(X)\begin{pNiceMatrix}
                    \mat{\Delta}(X)&0&\cdots &0\\
                    0&1&\cdots &0\\
                    \vdots & \vdots &\ddots &\vdots \\
                    0&0&\cdots &1\\
                    \end{pNiceMatrix}.
            \end{equation}
        \end{Corollary}
        \begin{proof}
           Notice that \eqref{B_minus} attains its simplest form when choosing 
            \begin{align}\label{d-coefficients}
                \underline{e}_1(X)&\eqdef e_1(X)-d_2(X)e_2(X)-\cdots-d_p(X)e_p(X),\\
                \nonumber\underline{e}_2(X)&\eqdef e_2(X),\\
                &\nonumber\phantom{=}\vdots\\
                \nonumber\underline{e}_p(X)&\eqdef e_p(X),
            \end{align}
            as a basis for $\mathbb{E}^{-}(X)$.
        \end{proof}
    
        \begin{Remark}
            The one-dimensional subspace spanned by $\underline{e}_1(X)$ is special, as it points in the direction in which the \Lop condition degenerates when $X\in\Gamma$. In what follows, we shall write $\ell(X)$ to denote this subspace and refer to it as the \textit{critical direction}.
        \end{Remark}
    
\section{The constant coefficients case}\label{ConstantCoefficientCase}
    \subsection{Construction of a $\WR$ symmetrizer and a priori estimates for the model problem.} 
        We first investigate the boundary value problem
        \begin{equation}\label{SystemPDOConstant}
            \left\{\begin{aligned}
                 Pu(t,x)\eqdef\left(D_{d}+\mathdutchcal{A}(D_t,D_y)\right)u(t,x)&=f(t,x) &\quad &\quad (t,x)\in \R^{1+d}_{+},\\
                 Bu|_{x_d=0}(t,y)&=g(t,y) &\quad &\quad (t,y)\in \R^{d},
            \end{aligned}\right.
        \end{equation}
        subject to Assumptions \ref{Assumption1} and \ref{Assumption2} (when nontrivial), where $\mathdutchcal{A}(D_t,D_y)$ is a first-order differential operator with symbol $a(\zeta)$. We introduce the notions of a \Lop family of operators and a $\WR$ symmetrizer.
        
        \begin{Definition}\label{LopMultiplier}
            Let $\delta(\zeta)\in C^{\infty}(\Xi, \mathcal{M}_{n\times n}(\C))$ be such that
            \begin{enumerate}[label=\normalfont(\roman*)]
                \item $\delta(\zeta)$ is homogeneous of degree $0$ in $\zeta$,
                \item \label{kerTrivial}  If $\zeta\in S^d$, $\ker{\delta(\zeta)}$ is trivial provided that $\gamma>0$,
                \item when $\zeta\in \Gamma$, $\ker{\delta(\zeta)}$ is nontrivial and $\ell(\zeta)\subseteq \ker{\delta(\zeta)}$.
            \end{enumerate}
            We shall call $\Delta_{\gamma}\eqdef \mathrm{Op}_{\gamma}(\delta)\in \mathrm{OPS}^{0}(\R^{1+d}_{+}\times [1, +\infty))$ a \Lop family of operators.
            %  A collection of pseudodifferential operators $\Delta_{\gamma}(t,x,D_t,D_y)\in \mathrm{OPS}^{0}(\R^{1+d}_{+}\times [\gamma_0, +\infty))$ is a \Lop family of operators if $\Delta_{\gamma}\eqdef \mathrm{Op}_{\gamma}(\delta)$.
        \end{Definition}
        \begin{Definition}\label{FS_C}
           Let 
           \begin{equation*}
               L^2_{\Delta}=\{v \in \mathcal{S}'(\R^{1+d}_+,\,\R^n): \Delta_{\gamma} v \in L_{\gamma}^2(\R^{1+d}_+,\,\C^n)\}.
           \end{equation*}
           A $\WR$ symmetrizer for Problem \ref{SystemPDOConstant} is a family of pseudodifferential operators $\Sigma_{\gamma}\in \mathrm{OPS}^{0}(\R^{1+d}_{+}\times [1, +\infty))$ so that
            \begin{enumerate}[label=\normalfont(\roman*)]
                \item \label{FS0} $\Sigma_{\gamma}$ is Hermitian, 
                \item \label{FS1} for every $v_1,v_2\in L^2_{\Delta}$, there is a positive constant $C$ satisfying
                    \begin{equation*}
                        \langle{\Sigma_{\gamma}v_1,v_2}\rangle\leq C\vert{\Delta_{\gamma}v_1}\vert\vert{\Delta_{\gamma}v_2}\vert,
                    \end{equation*}
                \item  \label{FS2} there exists a positive constant $c$ such that
                    \begin{equation*}
                        \ima{(\Sigma_{\gamma}\mathdutchcal{A}v,v)}\geq c\gamma\vert{\Delta_{\gamma}v}\vert^2,
                    \end{equation*}
                    for each $v\in L^2_{\Delta}$,
                \item \label{FS3} there are positive constants $\alpha$ and $\beta$, together with a family of pseudodifferential operators $Q_{\gamma}\in\mathrm{OPS}^{0}(\R^{d}_{+}\times [1, +\infty))$ such that
                  \begin{equation*}
                        \langle{\Sigma_{\gamma}v(0),v(0)}\rangle\geq \alpha\vert{\Delta_{\gamma}v(0)}\vert^2-\beta\vert{Q_{\gamma}Bv(0)}\vert^2.
                    \end{equation*}
            \end{enumerate}
        \end{Definition}
            
        If a $\WR$ symmetrizer exists, we expect the symbol $\sigma(\zeta)$ of $\Sigma_{\gamma}$ to be somewhat degenerate on the critical set $\Gamma$. More precisely, we have
        \begin{Proposition}\label{KrilovDegeneracy}
            If $\underline{\zeta}\in \Gamma$, $v\mapsto\langle{\sigma(\underline{\zeta})v,v}\rangle$ vanishes on the Krylov space of 
            \begin{equation*}
                \ell(\underline{\zeta})= \mathbb{E}^{-}(\underline{\zeta})\cap\ker{b(\underline{\zeta})}
            \end{equation*}
            with respect to $\mathdutchcal{A}_{\gamma}$, that is to say, on the smallest invariant subspace of $a(\underline{\zeta})$ containing $\ell(\underline{\zeta})$. 
        \end{Proposition}
        \begin{proof}
            We prove initially that $v\mapsto\langle{\sigma(\zeta)v,v}\rangle$ restricted to $\mathbb{E}^{-}(\zeta)$ is positive definite for $\gamma >0$.  To do so,  let $u\in \mathbb{E}^{-}(\zeta)$ be such that $u\neq 0$ and consider the initial value problem
            \begin{align*}
                \left\{\begin{aligned}
         		       D_s v+a(\zeta)v&=0,\\
         		       v(0)&=u, \\
    		    \end{aligned}\right.
            \end{align*}  
            whose solution is well known and equal to 
            \begin{equation}\label{solution}
                v(s)=e^{-ia(\zeta)s}u
            \end{equation}
            in the sense of matrices. Then, 
            \begin{align}\label{KryIVP}
                \nonumber\frac{d}{ds}{\langle{\sigma(\zeta) v,v}\rangle}=\langle{\sigma(\zeta)\partial_{s}v,v}\rangle+\langle{\sigma(\zeta) v,\partial_{s}v}\rangle&=i\langle{\sigma(\zeta)D_{s}v,v}\rangle-i\langle{\sigma(\zeta) v,D_{s}v}\rangle\\\nonumber&=-2\ima{\langle{D_sa(\zeta)v,v}\rangle}\\\nonumber&=2\ima{\langle{\sigma(\zeta) a(\zeta)v,v}}\rangle\\&\geq 2c\gamma\vert{\delta(\zeta)v}\vert^2>0.
            \end{align}
            The first inequality in \eqref{KryIVP} is due to \ref{FS2} in Definition \ref{FS_C} and Plancherel's theorem, while the second one is due to Condition \ref{kerTrivial} in Definition \ref{LopMultiplier}. Since $u\in \mathbb{E}^{-}(\zeta)$, $v\to 0$ decreases exponentially fast as $s\to \infty$ and hence, when integrating over $\R^{+}$, 
            \begin{equation*}
                -\langle{\sigma(\zeta) v(0),v(0)}\rangle=\int_{0}^{\infty}{\frac{d}{ds}{\langle{\sigma(\zeta) v(s),v(s)}\rangle}}\,ds\geq 2C\gamma\int_{0}^{\infty}{\vert{\delta(\zeta) v(s)}\vert^2}\,ds>0.
            \end{equation*}
            Put it differently, $\langle{{\sigma(\zeta) u,u}\rangle}< 0$, as claimed. Back to the original assertion, let us fix $\underline{\zeta}\in \Gamma$. Since $\mathbb{E}^{-}(\underline{\zeta})\cap\ker{b(\zeta)}$ is nontrivial, for every $v\in \mathbb{E}^{-}(\underline{\zeta})\cap\ker{b(\zeta)}$ such that $v\neq 0$, it must happen on one hand that $\langle{\sigma(\underline{\zeta})v,v}\rangle\leq0$ by the opening argument of this proof and by continuity in $\gamma$, and on the other that $\langle{\sigma(\underline{\zeta})v,v}\rangle\geq0$ because of a combination of Plancherel's theorem and Part \ref{FS3} in Definition \ref{FS_C}. This ensures that the restriction of $\sigma(\underline{\zeta})$ to $\ker{b(\zeta)}$ is non-negative. Together, both facts indicate that $\sigma(X)|_{\ell(X)}=0$. 
            
            In order to prove that $\sigma(\zeta)$ certainly vanishes in a larger subspace, we argue in a similar fashion as above and integrate \eqref{KryIVP} from $0$ to a positive real number $t$,
            \begin{equation*}
                \langle{\sigma(\underline{\zeta}) v(t),v(t)}\rangle-\langle{\sigma(\underline{\zeta}) v(0),v(0)}\rangle=\int_{0}^{t}{\frac{d}{ds}{\langle{\sigma(\underline{\zeta}) v(s),v(s)}\rangle}}\,ds\geq 0,
            \end{equation*}
            or equivalently, 
            \begin{equation}\label{noninc}
                \langle{\sigma(\underline{\zeta}) v(t),v(t)}\rangle\geq\langle{\sigma(\underline{\zeta}) v(0),v(0)}\rangle=\langle{\sigma(\underline{\zeta}) u,u}\rangle.
            \end{equation}
            If we pick $u \in\ell(\underline{\zeta})$,  the right-hand side of \eqref{noninc} is automatically zero and it is safe to say that
            \begin{equation}\label{KrylovCond_2}
                \langle{\sigma(\underline{\zeta}) v(t),v(t)}\rangle\geq 0.
            \end{equation}
            Let $K_{\ell}(\underline{\zeta})$ be the smallest invariant subspace of $a(\underline{\zeta})$ containing $\ell(\underline{\zeta})$. As $\ell(\underline{\zeta})$ is included in $\mathbb{E}^{-}(\underline{\zeta})$  and $\mathbb{E}^{-}(\underline{\zeta})$ is invariant under $a(\underline{\zeta})$, we necessarily have that $K_{\ell}(\underline{\zeta})\subseteq\mathbb{E}^{-}(\underline{\zeta})$. Furthermore, as the solution of a first-order autonomous differential equation whose initial value belongs to an invariant space remains within the invariant space, it is true that
            \begin{equation}\label{KryPos}
                \langle{\sigma(\underline{\zeta}) v(t),v(t)}\rangle\geq 0.
            \end{equation}
            The reverse inequality may be inferred from two facts, namely, that $\langle{\sigma(\underline{\zeta})\,\cdot\,,\,\cdot\,}\rangle\leq 0$ when restricted to $\mathbb{E}^{-}(\underline{\zeta})$ (already verified!), and that $v(t)\in \mathbb{E}^{-}(\underline{\zeta})$ for every $t$.  
        \end{proof}
        \begin{Remark}\normalfont
            The latter means, roughly speaking, that in general there is no hope that $\sigma(\zeta)$ can kill only the critical direction $\ell(X)$, but a larger subspace containing $\ell(X)$.
        \end{Remark}
        % \subsection{Construction of a $\WR$ symmetrizer and energy estimates for the model problem}  
        Since the upcoming analysis is intended to motivate future results, we shall keep things simple and assume at this early stage that $K_{\ell}(\zeta)=\mathbb{E}^{-}(\zeta)$, leaving the more general case $K_{\ell}(\zeta)\subset\mathbb{E}(\zeta)$ for the next section where problems with variable coefficients are explored. Also, we shall assume for ease of exposition that there are no glancing points.
        
        \begin{Theorem}
            Suppose there exists a $\WR$ symmetrizer for Problem \ref{SystemPDOConstant}. If $f\in L^2_{\Delta}$ and  $Q_{\gamma}g\in L^2_{\gamma}$, there exist constants $\gamma_0\geq 1$ and $C>0$ such that for all $\gamma\geq \gamma_0$ and for every $u\in \mathcal{D}(\overline{\R}^{1+d}_{+})$, the following estimate holds 
            \begin{align}
                \gamma\int_{\R^{1+d}_{+}}{e^{-2\gamma t}\vert{\Delta_{\gamma} u}\vert^2}\,dt\, dx &+\int_{\R^{d}}{e^{-2\gamma t}\vert{\Delta_{\gamma} u(0)}\vert^2}\,dt\, dy 
                \\&\leq C\left(\frac{1}{\gamma}\int_{\R^{1+d}_{+}}{e^{-2\gamma t}\vert{\Delta_{\gamma} f}\vert^2}\,dt\, dx \right. \nonumber\left. +\int_{\R^{d}}{e^{-2\gamma t}\vert{Q_{\gamma}Bu(0)}\vert^2\,dt\, dy}\right).
            \end{align}
            Moreover, there exists an operator $Y_{\gamma}$ such that its symbol $y(\zeta)$ is a projector, $Y_{\gamma}g\in L^{2}_{\gamma}$, and 
            \begin{align}
                \gamma\int_{\R^{1+d}_{+}}{e^{-2\gamma t}\vert{\Delta_{\gamma} u}\vert^2}\,dt\, dx &+\int_{\R^{d}}{e^{-2\gamma t}\vert{\Delta_{\gamma} u(0)}\vert^2}\,dt\, dy 
                \\&\leq C\left(\frac{1}{\gamma}\int_{\R^{1+d}_{+}}{e^{-2\gamma t}\vert{\Delta_{\gamma} f}\vert^2}\,dt\, dx \right. \nonumber\left. +\int_{\R^{d}}{e^{-2\gamma t}\vert{\Delta_{\gamma}Y_{\gamma}Bu(0)}\vert^2\,dt\, dy}\right).
            \end{align}
        \end{Theorem}
        \begin{proof}
           Let us see how Conditions \ref{FS0} to \ref{FS3} in Definition \ref{FS_C} imply energy estimates for the $\WR$ class in the current situation. To shorten the notation, we shall often omit the independent variables and the parameter $\gamma$ in the calculations ahead. We begin by expanding the term $d\langle{\Sigma u,u}\rangle/dx_d$ as shown,
            \begin{align}
                \nonumber\frac{d}{dx_d}\langle{\Sigma u,u}\rangle&=\langle{\Sigma\partial_{d}u,u}\rangle+\langle{\Sigma u,\partial_{d}u}\rangle\\
                \nonumber&=2\real{\langle{\Sigma iD_{d}u,u}\rangle}\\
                \nonumber&=2\real{\langle\Sigma i(f-\mathdutchcal{A}u),u\rangle}\\
                \nonumber&=2\ima{\langle{\Sigma au,u}\rangle}-2\ima{\langle{\Sigma f,u}\rangle}.
            \end{align}
            Keeping in mind that $u\in \mathcal{D}(\R^{1+d})$ vanishes at infinity, an integration over $[0,\infty)$ with respect to $x_d$ produces
            \begin{align*}
                \langle{\Sigma u(0), u(0)}\rangle=&-2\int_{0}^{\infty}{\ima{\langle{\Sigma Au,u}\rangle}}\,dx_d+2\int_{0}^{\infty}{\ima{\langle{\Sigma f,u}\rangle}}\,dx_d.
            \end{align*}
            To bound both integrals, we exploit Definition \ref{FS_C} directly. For example, from Condition \ref{FS1},
            \begin{align*}
                2\ima{\langle\Sigma f,u\rangle}&\leq 2|\langle\Sigma f,u\rangle|\leq C_1 \vert{\Delta f}\vert\vert{\Delta u}\vert
            \end{align*}
            for some positive constant $C_1$, whereas from (\ref{WRFS1} it is clear that
            \begin{align}\label{EE_Step1}
                \langle{\Sigma u(0), u(0)}\rangle&\leq -c\gamma\int_{0}^{\infty}{\vert{\Delta u}\vert^2}\,dx_d+C_1\int_{0}^{\infty}{\vert{\Delta f}\vert\vert{\Delta u}\vert}\,dx_d.
            \end{align}
            We can control \eqref{EE_Step1} from below by means of \ref{FS3} in Definition \ref{FS_C}, and from above via Young's inequality, so that
            \begin{align*}
                \alpha|\Delta u(0)|^2-\beta|QBu(0)|^2&\leq  \left(-c\gamma+\varepsilon\gamma\right)\int_{0}^{\infty}{\vert{\Delta u}\vert}\,dx_d+\frac{C_1}{4\varepsilon\gamma}\int_{0}^{\infty}{\vert{\Delta f}\vert^2}\,dx_d.
            \end{align*}
            Taking a sufficiently small value for $\varepsilon$ yields
            \begin{align*}
                \alpha|\Delta u(0)|^2-\beta|QBu(0)|^2&\leq  -C_2\gamma \int_{0}^{\infty}{\vert{\Delta u}\vert^2}\,dx_d+\frac{C_3}{\gamma}\int_{0}^{\infty}{\vert{\Delta f}\vert^2}\,dx_d,
            \end{align*}
            for some constant $C_2$, or what is the same,   
            \begin{align}\label{ConstantCoefficients}
                \gamma\int_{\R^{1+d}_{+}}{e^{-2\gamma t}\vert{\Delta u}\vert^2}\,dt\, dx &+\int_{\R^{d}}{e^{-2\gamma t}\vert{\Delta u(0)}\vert^2}\,dt\, dy 
                \\&\leq C\left(\frac{1}{\gamma}\int_{\R^{1+d}_{+}}{e^{-2\gamma t}\vert{\Delta f}\vert^2}\,dt\, dx \right. \nonumber\left. +\int_{\R^{d}}{e^{-2\gamma t}\vert{QBu(0)}\vert^2\,dt\, dy}\right),
            \end{align}
            for some $C>0$. 
                
            To conclude, we show how to recover estimates analogous to those in \cite{benzoni-gavage_serre_2007}, Chapter 8. In there, each term is ``filtered'' by $\Delta$, even the one containing the boundary matrix $B$. For this purpose, let
           \begin{equation*}
               m\eqdef \begin{pmatrix}
                    I_p &qB^{+}
                \end{pmatrix},
           \end{equation*}
           with
           \begin{equation*}
               B^{+}\eqdef (Be_{p+1},\,\cdots, Be_n),
           \end{equation*}
           and observe that 
            \begin{align*}
                qB=\begin{pmatrix}
                        qB^{-} &qB^{+}
                    \end{pmatrix}&=\begin{pmatrix}
                        I_p &qB^{+}
                    \end{pmatrix}\begin{pmatrix}
                        \delta^{-}&\\
                        &I_{n-p}\\
                    \end{pmatrix}\\&=m\delta.
            \end{align*}
            If we take any $(n-p)\times n$ matrix $\mathdutchcal{x}$ such that $\mathdutchcal{x}$ is surjective and
            \begin{equation*}
                \begin{pmatrix}
                    B\\
                    \mathdutchcal{x}\\
                \end{pmatrix}
            \end{equation*}
            is nonsingular, there exist matrices $\mathdutchcal{y}$ and $\mathdutchcal{d}$ with respective dimensions $n\times p$ and $n\times (n-p)$, so that
            \begin{equation}\label{BXDY}
                I_n=\begin{pmatrix}
                    b\\
                    \mathdutchcal{x}\\
                \end{pmatrix}\begin{pmatrix}
                    \mathdutchcal{y}& \mathdutchcal{d}
                \end{pmatrix}=\begin{pmatrix}
                    b\mathdutchcal{y} &b\mathdutchcal{d}\\
                    \mathdutchcal{x}\mathdutchcal{y} &\mathdutchcal{x}\mathdutchcal{d}\\
                \end{pmatrix}.
            \end{equation}    
            From \eqref{BXDY}, we may deduce that $b\mathdutchcal{y}=I_p$, and then that 
            \begin{equation*}
                qb=qb\mathdutchcal{y}b=m\delta \mathdutchcal{y}b.
            \end{equation*}
            Finally, if $Y\eqdef\mathrm{Op}(y)$ and $M\eqdef\mathrm{Op}(m)$, by invoking Plancherel's theorem we conclude that $QB=M\Delta YB$ and
            \begin{align}\label{Delta_Estimates_CC}
                \gamma\int_{\R^{1+d}_{+}}{e^{-2\gamma t}\vert{\Delta u}\vert^2}\,dt\, dx &+\int_{\R^{d}}{e^{-2\gamma t}\vert{\Delta u(0)}\vert^2}\,dt\, dy 
                \nonumber\\&\lesssim\left(\frac{1}{\gamma}\int_{\R^{1+d}_{+}}{e^{-2\gamma t}\vert{\Delta f}\vert^2}\,dt\, dx \right. \nonumber\left. +\int_{\R^{d}}{e^{-2\gamma t}\vert{\Delta QBu(0)}\vert^2\,dt\, dy}\right)\\&\lesssim \left(\frac{1}{\gamma}\int_{\R^{1+d}_{+}}{e^{-2\gamma t}\vert{\Delta f}\vert^2}\,dt\, dx \right. \left. +\int_{\R^{d}}{e^{-2\gamma t}\vert{\Delta YBu(0)}\vert^2\,dt\, dy}\right),
            \end{align}
            as desired.
        \end{proof}
        \begin{Theorem}
            Let $(P,B)$ be as in \ref{SystemPDOConstant}. If $(P,B)$ is of $\WR$ type, there exists a $\mathcal{WR}$ symmetrizer satisfying Definition \ref{FS_C}.
        \end{Theorem}
        \begin{proof}
            It suffices to argue symbolically thanks to Plancherel's formula. That said, we need to construct a symbol $\sigma$ with the following properties:
            \begin{enumerate}[label=\normalfont(\roman*)]
                \item \label{ZeroMic} $\sigma$ is Hermitian, 
                \item \label{ZeroOneMic} for every $v_1,v_2\in \C^n$, there is a positive constant $C$ satisfying
                    \begin{equation*}
                        \langle{\sigma v_1,v_2}\rangle\leq C\vert{\delta v_1}\vert\vert{\delta v_2}\vert,
                    \end{equation*}
                \item  \label{OneMic} there exists a positive constant $c$ such that
                    \begin{equation*}
                        \ima{(\sigma a v,v)}\geq c\gamma\vert{\delta v}\vert^2,
                    \end{equation*}
                    for each $v\in \C^n$,
                \item \label{TwoMic} there are positive constants $\alpha$ and $\beta$, together with a symbol $q$ such that
                  \begin{equation*}
                        \langle{\sigma v(0),v(0)}\rangle\geq \alpha\vert{\delta v(0)}\vert^2-\beta\vert{q bv(0)}\vert^2.
                    \end{equation*}
            \end{enumerate}
            
            Let $\underline{\zeta}\in\mathcal{H}$. The block structure condition indicates that $a(\xi)$ is smoothly diagonalizable in a neighborhood $\mathcal{V}$ of $\underline{\zeta}$ with eigenvalues $a_{1}(\zeta),\,\cdots, a_{n}(\zeta)$ \footnote{counted according to their multiplicities} and eigenvectors $e_1(\zeta),\,\cdots, e_n(\zeta)$ ordered as the columns of a nonsingular matrix $\ee_0(\zeta)$, so that 
            \begin{equation}
                \dot{a}_1(\zeta)=\ee_0^{-1}(\zeta)a_1(\zeta)\ee_0(\zeta)=
                \begin{pNiceMatrix}
                    a_{1,1}(\zeta)&0&\cdots &0\\
                    0&a_{1,2}(\zeta)&\cdots &0\\
                    \vdots & \vdots &\ddots &\vdots \\
                    0&0&\cdots &a_{1,n}(\zeta)\\
                \end{pNiceMatrix}.
            \end{equation}
            If $\kappa_j(\underline{\zeta})\eqdef-i\partial a_{j}(\underline{\zeta})/\partial \gamma$, Taylor's theorem with respect to $\gamma$ shows that 
            \begin{equation}\label{Taylor}
                a_{j}(\zeta)=a_j(\underline{\zeta})+i\gamma \kappa_j(\underline{\zeta})+\gamma^2 w_j(\zeta),
            \end{equation}
            where $w_j(\zeta)$ is a smooth function.  Notice that \eqref{Immra} is bounded away from zero because of the continuity of $\kappa_j(\underline{\zeta})$ and the fact that $\underline{\kappa}_j\eqdef\kappa(\underline{\zeta})\in \R\setminus\{0\}$ (also ensured by Proposition \ref{BSC}). Actually, we may say a little bit more about $\underline{\kappa}_j$. By definition, 
            \begin{equation*}
                \underline{\kappa}_{j}=\real{\left(-i\frac{\partial a_j}{\partial \gamma}(\underline{\zeta})\right)}=\ima{\left(\frac{\partial a_j}{\partial \gamma}(\underline{\zeta})\right)}=\ima{\left(\lim_{\gamma\to 0^{+}}\frac{a_j(\zeta)-a_j(\underline{\zeta})}{\gamma}\right)}=\lim_{\gamma\to 0^{+}}\frac{\ima{a_j(\zeta)}}{\gamma},
            \end{equation*}
            showing that $\underline{\kappa}_j$ and $\ima{a(\zeta)}$ have the same sign for $\zeta=(t,y,x_d,\tau,\eta,\gamma)$ sufficiently close to $\underline{\zeta}$.
            Let 
            \begin{equation}
                r=\begin{pNiceMatrix}
                        r_{1}&0&\cdots &0\\
                        0&r_{2}&\cdots &0\\
                        \vdots & \vdots &\ddots &\vdots \\
                        0&0&\cdots &r_{n}\\
                    \end{pNiceMatrix},
            \end{equation}
            with every $r_j\in \R$, $1\leq j\leq n$ to be chosen later. If 
            \begin{equation}\label{delta_diag}
                \delta(\zeta)\equiv\mathrm{diag}\left(\delta_1(\zeta),\,\cdots, \delta_n(\zeta)\right)\eqdef\mathrm{diag}\left(\mat{\Delta}(\zeta)I_p, I_{n-p}\right),
            \end{equation}
            we seek $\sigma$ in the form $\sigma= \delta^*r\delta$. Given that $a$, $\delta$, and $r$ are all diagonal, we have
            % For instance, Condition \ref{FS0} is immediate from
            % \begin{equation*}
            %     \langle{\Sigma v_1, v_2}\rangle=\langle{\Op{(r)}\Delta v_1,\Delta v_2}\rangle\leq C \vert{\Delta v_1}\vert\vert{\Delta v_2}\vert,
            % \end{equation*}
            % which is valid for arbitrary test functions $v_1,v_2$. 
            \begin{equation}\label{im_ra_1}
                \ima{\overline{\delta}_jr_{j}a_{j}\delta_j(\zeta)}=\ima\left({r_j a_{j}(\underline{\zeta})\vert{\delta_j(\zeta)}\vert^2+i\gamma r_j \kappa_j(\underline{\zeta})\vert{\delta(\zeta)}\vert^2+r_j\gamma^2w_j(\zeta)\vert{\delta_j(\zeta)}\vert^2}\right),
            \end{equation}
            and since $a_j(\underline{\zeta})\in \R$ (see Definition \ref{BSC}), we may discard $r_ja_j(\underline{\zeta})$ in \eqref{im_ra_1} and write 
            \begin{equation}\label{Immra}
                \ima{r_ia_j}(\zeta)=\gamma r_i \real{\kappa_{j}}(\underline{\zeta})+\gamma^2r_i\ima{w_j(\zeta)}.
            \end{equation}                    
            Now, considering that
            \begin{equation*}
                \mathbb{E}^{-}(\zeta)=\bigoplus_{\ima a_j(\zeta)<0}\ker{(a(\zeta)-a_j(\zeta))}=\{v\in \C^{n}: v_j=0 \quad \mathrm{if}\quad \underline{\kappa}_j>0\},
            \end{equation*}
            it is straightforward to check that Condition \eqref{OneMic} is satisfied when 
            \begin{align}
                r_j=\left\{\begin{aligned}
         		       -1  \quad \textrm{for} \quad \underline{\kappa}_j>0,\\
         		      \rho \quad \textrm{for} \quad \underline{\kappa}_j<0,\\
        		\end{aligned}\right. 
            \end{align}
            with $\rho$ being a positive constant to be specified when meeting Condition \ref{TwoMic}. To check the last item, let us suppose that ${v}^-$ (resp. $v^+$) is the projection component of $v\in \C^{n}$ onto ${\mathbb{E}}^{-}(\zeta)$ (resp. ${\mathbb{E}}^{+}(\zeta)$) so that
            \begin{equation*}
                v=\begin{pmatrix}
                    v^{-}\\v^{+}
                \end{pmatrix}.
            \end{equation*}  
            A direct calculation gives
            \begin{align}\label{w_minus_plus}
                \langle{\sigma(\zeta)v,v}\rangle=\langle{\delta^*(\zeta)r\delta(\zeta)v,v}\rangle&=-|\mat{\Delta}(\zeta)v^-|^2+\rho|v^+|^2\\&=\phantom{-}|\mat{\Delta}(\zeta)v^-|^2+\rho|v^+|^2-2|\mat{\Delta}(\zeta)v^-|^2 \nonumber,
            \end{align}
            which suggests using equation \eqref{B_minus} to link \eqref{w_minus_plus} with the boundary matrix $B$. Indeed, borrowing $\mathfrak{p}(\zeta)$ an $c(\zeta)$ from Proposition \ref{STBasis}, one has
            \begin{align}\label{DeltaDecomp}
               \nonumber\mat{\Delta}I_p&=c(\zeta)(\mat{\Delta}I_p )c^{-1}(\zeta)\\&=c(\zeta)\begin{pmatrix}
                1&0&\hdots &0\\
                0&\mat{\Delta}(\zeta) &\hdots&0 \\
                \vdots&\vdots &\ddots &\vdots\\
                0&0&\hdots&\mat{\Delta}(\zeta)\\ 
            \end{pmatrix}\begin{pmatrix}
                \mat{\Delta}(\zeta)&0&\hdots &0\\
                0&1 &\hdots&0 \\
                \vdots&\vdots &\ddots &\vdots\\
                0&0&\hdots&1\\ 
            \end{pmatrix}c^{-1}(\zeta)\\&=c(\zeta)\begin{pmatrix} \nonumber
                1&0&\hdots &0\\
                0&\mat{\Delta}(\zeta) &\hdots&0 \\
                \vdots&\vdots &\ddots &\vdots\\
                0&0&\hdots&\mat{\Delta}(\zeta)\\ 
            \end{pmatrix}\mathfrak{p}^{-1}(\zeta)\mathfrak{p}(\zeta)\begin{pmatrix}
                \mat{\Delta}(\zeta)&0&\hdots &0\\
                0&1 &\hdots&0 \\
                \vdots&\vdots &\ddots &\vdots\\
                0&0&\hdots&1\\ 
            \end{pmatrix}c^{-1}(\zeta)\\&=q(\zeta)b^{-}(\zeta),
           \end{align}
           where $q(\zeta)\eqdef c(\zeta)\mathrm{diag}(1,\mat{\Delta}(\zeta),\,\cdots, \mat{\Delta}(\zeta))\mathfrak{p}^{-1}(\zeta)$. Armed with \eqref{DeltaDecomp} and 
           \begin{equation*}
               b^{+}(\zeta)\eqdef (b(\zeta)e_{p+1}(\zeta),\,\cdots, b(\zeta)e_n(\zeta)),
           \end{equation*}
           we see that the first term in \eqref{w_minus_plus} unfolds as
            \begin{align}\label{Delta_minus_w_minus}
                |\mat{\Delta}(\zeta)v^-|^2&=\left|{\begin{pmatrix}
                    \mat{\Delta}(\zeta)I_p  & 0\\
                    0 & I_{n-p} 
                \end{pmatrix}\begin{pmatrix}
                    v^- \\ 0
                \end{pmatrix}}\right|^2\nonumber \\&=\left|\begin{pmatrix}
                    \mat{\Delta}(\zeta)I_p  &  q(\zeta)b^{+}(\zeta)\\
                    0 & I_{n-p} 
                \end{pmatrix}\begin{pmatrix}
                    v^- \\ 0
                \end{pmatrix}\right|^2\nonumber \\&=\left|\begin{pmatrix}
                    \mat{\Delta}(\zeta)I_p  &  q(\zeta)b^{+}(\zeta)\\
                    0 & I_{n-p} 
                \end{pmatrix}\begin{pmatrix}
                    v^- \\ v^+
                \end{pmatrix}-\begin{pmatrix}
                    \mat{\Delta}(\zeta)I_p  &  q(\zeta)b^{+}(\zeta)\\
                    0 & I_{n-p} 
                \end{pmatrix}\begin{pmatrix}
                    0 \\ v^+ 
                \end{pmatrix}\right|^2\nonumber\\&=\left|\begin{pmatrix}
                    q(\zeta)b^{-}(\zeta)  &  q(\zeta)b^{+}(\zeta)\\
                    0 & I_{n-p} 
                \end{pmatrix}\begin{pmatrix}
                    v^- \\ v^+
                \end{pmatrix}-\begin{pmatrix}
                    \mat{\Delta}(\zeta)I_p  &  q(\zeta)b^{+}(\zeta)\\
                    0 & I_{n-p} 
                \end{pmatrix}\begin{pmatrix}
                    0 \\ v^+ 
                \end{pmatrix}\right|^2\nonumber\\&\leq\left(\left|\begin{pmatrix}
                    q(\zeta)b(\zeta)v \\ v^+ 
                \end{pmatrix}\right|+C|v^+|\right)^2\leq  C\left(|q(\zeta)b(\zeta)v|^2+|v^+|^2\right), 
             \end{align}
            where the final line accounts for the convexity of the power function $x \mapsto x^2$. In the end, combining \eqref{w_minus_plus} and \eqref{Delta_minus_w_minus}, we get
             \begin{align*}
                 \langle{\sigma(\zeta)v,v}\rangle&\geq |\mat{\Delta}(\zeta)v^-|^2 +(\rho-2C)|v^+|^2-2C|q(\zeta)b(\zeta)v|^2,
           \end{align*}
           from which the result follows from choosing $\rho$ such that $\alpha\eqdef\rho-2C=1$, $\beta\eqdef 2C$.
        \end{proof}
\section{The variable coefficients case}\label{Sec_variablecoefficients}
        % Our starting point will be once again Problem \eqref{SystemPDO}.
        % \begin{equation}\label{SystemPDO}
        %     \left\{\begin{aligned}
        %          P_{\gamma}u_{\gamma}(t,x)\eqdef\left(D_{d}+\mathdutchcal{A}_{\gamma}(t,y,x_d,D_t,D_y)\right)u(t,x)&=f(t,x) &\quad &\quad (t,x)\in \R^{1+d}_{+},\\
        %          B_{\gamma}(t,y)u(t,y)|_{x_d=0}&=g(t,y) &\quad &\quad (t,y)\in \R^{d},
        %     \end{aligned}\right.
        % \end{equation}
        % where $\mathdutchcal{A}_{\gamma} \in \mathrm{OPS}^{1}(\R^{1+d}_{+}\times [\gamma_0, +\infty))$ is a classical pseudodifferential operator whose symbol $a\in S_{\gamma}^{1}(\R^{1+d}_{+}\times \R^{d}\times [\gamma_0, +\infty))$ is a matrix-valued function of dimension $n\times n$ that admits an asymptotic expansion
        % \begin{equation*}
        %     a\sim \sum^{\infty}_{j=0} a_{1-j},
        % \end{equation*}
        % each $a_{1-j}$ being homogeneous of degree $1-j$. Likewise, $B_{\gamma}\in \mathrm{OPS}^{0}(\R^{d}\times [\gamma_0, +\infty))$ is a classical pseudodifferential operator with a $p\times n$ principal part $b(X)\in S^{0}(\R^{d}\times \R^{d}\times [\gamma_0, +\infty))$. In addition, the source data $f$ and $g$ are chosen at least in $L^2_{\gamma}$. We supplement Problem \ref{SystemPDO} with Assumption \ref{Assumption2}. 

    \subsection{Construction of a $\mathcal{WR}$ symmetrizer and a priori estimates for the general problem.}
        We shall generalize the notion of a \Lop multiplier to fit Problem \ref{SystemPDO}, taking into account Proposition \ref{KrilovDegeneracy}.
        \begin{Theorem}\label{SymWR}
            Suppose that $(P_{\gamma},B_{\gamma})$ is a $\mathcal{WR}$ boundary value problem furnished with Assumptions \ref{Assumption1} and \ref{Assumption2}.  Let $\underline{X}\in\mathcal{H}$. Then there exist symbols $\tilde{\ee}_0(X)$ and $\delta(X)$ defined in some conic neighborhood $\mathcal{V}$ of $\underline{X}$ such that for every $X\in \mathcal{V}$,
            \begin{enumerate}[label=\normalfont(\roman*)]
                \item $\tilde{\ee}_0(X)$ and $\delta(X)$ are homogeneous of degree $0$,
                \item $\tilde{\ee}_0(X)\in GL_n(\C)$ and 
                \begin{equation*}
                    \dot{a}_1(X)\eqdef \tilde{\ee}_0^{-1}(X)a_1(X)\tilde{\ee}_0(X)
                \end{equation*}
                is diagonal with entries $a_{1,1}(X),\,\cdots, a_{1,n}(X)$,
                \item there is $\mathdutchcal{s}\leq p$ so that $\delta(X)$ is diagonal with respect to the basis $\tilde{\ee}_0(X)$ and given by \begin{equation*}
                    \delta(X)=\mathrm{diag}(\delta^{-}(X), I_{n-p}),   
                \end{equation*}
                where
                \begin{equation*}
                    \delta^{-}(X) \eqdef\begin{pNiceMatrix}
                        \begin{pmatrix}
                    \delta_{1}^{-}(X)&\hdots&0\\
                    \vdots&\ddots&\vdots\\
                    0&\hdots&\delta_{\mathdutchcal{s}}^{-}(X)\\
                \end{pmatrix}&      \\
                        &I_{p-\mathdutchcal{s}}
                    \end{pNiceMatrix}
                \end{equation*}    
                is such that each $\delta^{-}_j$ is the solution of the transport equation
                \begin{align}\label{TransportEq}
                    \left\{\begin{aligned}
         		       \partial_{d}\delta_{j}^{-}+\{\delta_{j}^{-},a_{1,j}\}&=0,\\
         		       \delta_{j}^{-}|_{x_d=0}&=\mat{\Delta}, \\
    		        \end{aligned}\right.
                \end{align}  
                % \item \textcolor{red}{each $\delta^{-}_i$ solves the transport equation} 
                \item \label{deltaBoundary} when $X\in \mathbb{Y}_S\cap \mathcal{H}$, there exist matrices $q(X)$ and $m(X)$ depending smoothly on $X\in \mathcal{V}\cap \mathbb{Y}$ with dimensions $p\times p$ and $p\times n$, respectively, so that if $\dot{b}(X)\eqdef b(X)\tilde{\ee}_0(X)$, there holds
                \begin{equation}\label{BoundaryRelation}
                    q(X)\dot{b}(X)=m(X)\delta(X),
                \end{equation}
                \item $\ker{\delta(X)}\neq \{0\}$ if and only if $X\in \Gamma$. When nontrivial, $\ker{\delta(X)}$ is an $\mathdutchcal{s}-$dimensional invariant subspace of $a_1(X)$ containing the critical direction $\ell(X)$.
            \end{enumerate}
       \end{Theorem}
        Prior to entering into the proof of Theorem \ref{SymWR}, we shall state an auxiliary result. 
        % \textcolor{red}{We continue with a statement which  will serve useful in  proving Theorem \ref{SymWR}}.
        \begin{Lemma}\label{LemDecomp}
            Let $V$ be a finite dimensional vector space. Suppose $T\in \mathrm{End}(V)$ is diagonalizable with distinct eigenvalues $\lambda_1, \,\cdots, \lambda_q$, and corresponding eigenspaces $V_{\lambda_1},\,\cdots,V_{\lambda_q}$. Then every $T-$invariant subspace $W$ can be decomposed as
            \begin{equation}\label{W_decomp}
                W=(W\cap V_{\lambda_1})\oplus\cdots\oplus(W \cap V_{\lambda_q}).
            \end{equation}
        \end{Lemma}
        \begin{proof}
            First of all, note that $W_i\eqdef W \cap V_{\lambda_i}$ is the intersection of two subspaces of $V$, so formula \eqref{W_decomp} is meaningful. To prove the assertion, we need to show that $W_i\cap W_j=\{0\}$ for $i\neq j$ and that $W=W_1\oplus W_2\oplus\cdots\oplus W_k$. Let $w\in W$. To verify the first part, let us assume that $w\in W_i\cap W_j=\{0\}$ for different indices $i$ and $j$. The latter necessarily implies that $w\in E_i\cap E_j=\{0\}$, and eventually that $w=0$ since $\lambda_i\neq \lambda_j$. Next, let us write $w_i$ for the projection onto $V_{\lambda_i}$. Clearly, $w_i\in W_i$, and consequently $w$ can be expressed as $w=w_1+\cdots+w_k$. Lastly, if $w\in W_i\cap \breve{W}$ with
            \begin{equation*}
                \breve{W}\eqdef W_1\oplus W_2\oplus W_{i-1}\oplus W_{i+1}\oplus W_{k},
            \end{equation*}
            then $w=0$ for $W_i\cap W_j=\{0\}$ with $i\neq j$. In conclusion, $W=W_1\oplus \cdots\oplus W_{k}$, as desired. 
        \end{proof}
        \begin{proof}[Proof of Theorem \ref{SymWR}]
            For simplicity's sake, we shall split the argument into several steps.
            
            \noindent\textit{Step 1. We classify points in $\mathcal{H}$.} Our strategy focuses on defining $\delta(X)$ initially for $X=(t,y,x_d,\tau,\eta,\gamma)$ with $x_d$ small, and then extend $\delta(X)$ in a constant way to larger values of $x_d$. That being said, let us assume that $\underline{X}\in\mathcal{H}\cap \mathbb{Y}$. We distinguish two cases, namely when $\underline{X}$ belongs to $\Gamma$ and when it does not. In the latter, we choose $\tilde{\ee}_0(X)$ as predicted by the block structure condition (see Definition \ref{BSC} and Proposition \ref{BSCProp}) and notice that the uniform \Lop condition is fulfilled. As a result, each $\delta_j(X)$, \textemdash which is guaranteed to exist locally by Picard-Lindelöf's theorem\textemdash,  never vanishes in a small neighborhood $\mathcal{V}$ of $\underline{X}$, for $\mat{\Delta}(X)$ never vanishes in $\mathcal{V}\cap \mathbb{Y}$ either. Hence, $\delta(X)$ is nonsingular and it follows that equation \eqref{BoundaryRelation} holds by choosing $q(X)=I_p$ and $m(X)=\dot{b}(X)\delta^{-1}(X)$. The remaining and most interesting case occurs therefore around points $\underline{X}\in \mathbb{Y}\cap \mathcal{H}$ where the \Lop determinant vanishes to the first order. We devote the rest of the proof to examine this situation. 
            
            \noindent\textit{Step 2. We find a suitable basis $\tilde{\ee}(X)$.}  Recall that $a_1(X)$ is smoothly diagonalizable around $\underline{X}$ in view of Proposition \ref{BSC},  meaning that for every $X$ in a neighborhood $\mathcal{V}$ of $\underline{X}$ there exist eigenvalues 
            \begin{equation*}
                {a_{1,1}}(X),\, \cdots , a_{1,n}(X)
            \end{equation*}
            (counted according to their multiplicities) and eigenvectors $e_1(X),\,\cdots,e_{n}(X)$ organized as the columns of a nonsingular matrix $\ee_0(X)$, so that 
            \begin{equation*}
                \ee_0^{-1}(X)a_1(X)\ee_0(X)=
                    \begin{pNiceMatrix}
                        a_{1,1}(X)&0&\cdots &0\\
                        0&a_{1,2}(X)&\cdots &0\\
                        \vdots & \vdots &\ddots &\vdots \\
                        0&0&\cdots &a_{1,n}(X)\\
                    \end{pNiceMatrix}.
            \end{equation*}
            For $X\in \mathcal{V}$, let $K_{\ell}(X)$ be the Krylov space of $\ell(X)$ with respect to $a_1(X)$. Admitting that the first $p$ columns of $\ee_0(X)$ span the stable subspace $\mathbb{E}^{-}(X)$ and that repeated eigenvalues are adjacent,  we explain how to pick a different basis for $\mathbb{E}^{-}(X)$ that interacts nicely with $K_{\ell}(X)$. To this end, let $\mu_{1}(X),\,\cdots,\mu_{q}(X)$ be pairwise different eigenvalues of $a_1(X)$ with multiplicities $\alpha_1,\,\cdots, \alpha_q$. For every $k\in\{1,\,\cdots,q\}$, we can find a positive integer $i_k\leq n$ such that
            \begin{equation*}
                \mu_{k}(X)=a_{1,i_k}(X)=a_{1,i_k+1}(X)=\cdots= a_{1,i_k+\alpha_k-1}(X),
            \end{equation*}
            with associated eigenspace $V_{k}(X)=\mathrm{span}\{{e_{i_k}(X),\,\cdots,e_{i_k+\alpha_k-1}(X)}\}$. With this at hand, Lemma \ref{LemDecomp} suggests that $K_{\ell}(X)$ can be decomposed as
            \begin{equation*}
                K_{\ell}(X)=K_{\ell}(X)\cap V_{1}(X)\oplus\cdots\oplus K_{\ell}(X) \cap V_{q}(X),
            \end{equation*}
            where $K_{\ell}(X)\cap V_{k}(X)$ is trivial for every $V_{k}(X)\subset\mathbb{E}^{+}(X)$, since $K_{\ell}(X)\subseteq\mathbb{E}^{-}(X)$ for every $X\in \mathcal{V}$. If $K_{\ell}(X)\cap V_{k}(X) \neq \{0\}$, we can choose $\tilde{e}_{k}(X) \in K_{\ell}(X)\cap V_{k}(X)$ and use it to replace an existing element in $\ee(X)$ in such a way that the resulting set $\tilde{\ee}_0(X)$ is still a basis of $\C^n$. Thus, if $\mathdutchcal{s}=\mathdutchcal{s}(X)$ is the number of non-zero coefficients from $d_1(X),\,\cdots, d_p(X)$ in \eqref{d-coefficients}, after rearranging components if necessary, the new basis $\tilde{\ee}_0(X)$ consists of eigenvectors of $a_1(X)$ whose first $s$ elements span $K_{\ell}(X)$. To put it differently,
            \begin{align*}
                \C^n&=\mathrm{span}\{{\tilde{e}_1(X),\,\cdots, \tilde{e}_\mathdutchcal{s}(X)}\}\oplus \,\mathrm{span}\{{\tilde{e}_{\mathdutchcal{s}+1}(X),\,\cdots, \tilde{e}_n(X)} \}\\&=K_{\ell}(X)\oplus\mathrm{span}\{{\tilde{e}_{\mathdutchcal{s}+1}(X),\,\cdots, \tilde{e}_n(X)}\}.
            \end{align*}
            Looking ahead to further stages of this proof, it is of primary interest to us that only one value of $\mathdutchcal{s}$ is chosen for the whole neighborhood $\mathcal{V}$. This is generally the case, save possibly when a coefficient $d_i$ vanishes point-wise at $X$ while not being identically zero in $\mathcal{V}$. In these circumstances, $K_\ell(X)\eqdef \mathrm{span}\{{\tilde{e}_1(X),\,\cdots, \tilde{e}_\mathdutchcal{s}(X)}\}$ with $\mathdutchcal{s}$ being the largest value we encounter while $X$ ranges on $\mathcal{V}$ (although $K_{\ell}(X)$ is no longer the smallest subspace containing $\ell(X)$ for every $X\in \mathcal{V}$, this will suffice for our plans). 

            \noindent\textit{Step 3. We define $\delta(X)$.} Having picked an appropriate basis in the previous passage, we can define
            \begin{equation*}
                \delta(X)=\mathrm{diag}(\delta^{-}(X), I_{n-p}),
            \end{equation*}
            with
            \begin{equation}\label{delta_dot}
                \delta^{-}(X) \eqdef\begin{pNiceMatrix}
                    \begin{pmatrix}
                        \delta_{1}^{-}(X)&\hdots&0\\
                        \vdots&\ddots&\vdots\\
                        0&\hdots&\delta_{\mathdutchcal{s}}^{-}(X)\\
                    \end{pmatrix}&      \\
                    &I_{p-\mathdutchcal{s}}
                \end{pNiceMatrix}.
            \end{equation}
            Each $\delta_{i}^{-}(X)$ in \eqref{delta_dot} solves locally the transport equation
            \begin{align}\label{DeltaExt}
                \left\{\begin{aligned}
         		       \partial_{d}\delta_{j}^{-}+\{\delta_{j}^{-},a_{1,j}\}&=0,\\
         		       \delta_{j}^{-}|_{x_d=0}&=\mat{\Delta}, \\
    		    \end{aligned}\right.
            \end{align}  
            whose characteristic curves coincide with the integral curves of the Hamiltonian system  
            \begin{align}\label{bicharacteristics}
                \left\{\begin{aligned}
         		       &\frac{dy_k}{dx_d}=\phantom{-}\frac{\partial a_{1,j}}{\partial \eta_k},\\
         		       &\frac{d\eta_k}{dx_d}=-\frac{\partial a_{1,j}}{\partial y_k}, \\
         		       &(y_0,\cdots, y_{d-1},0,\eta_0,\cdots,\eta_{d-1}, \gamma)\in \mathcal{V}\cap \mathbb{Y},
    		    \end{aligned}\right.
            \end{align}
            provided that we interpret $x_0$ as $t$ and $\eta_0$ as $\tau$. The preceding set of equations does not impose any restriction on $\gamma$, so we are free to complement \eqref{bicharacteristics} with the natural assumption that $d\gamma/dx_d=0$ along the bicharacteristic curves. As before, the existence of such $\delta_{j}^{-}(X)$ in (perhaps a smaller) $\mathcal{V}$ is justified by Picard-Lindelöf's theorem.
            
            \noindent\textit{Step 4. We link $\delta(X)$ and $b^{-}(X)$.} To investigate the behavior of $\delta(X)$ on the boundary $\{x_d=0\}$, we place ourselves at any point $X\in \mathcal{V}\cap\mathbb{Y}$ and take 
            \begin{equation}
                c(X)=\begin{pNiceMatrix}
                    1&0&\cdots &0\\
                    -d_2(X)&1&\cdots &0\\
                    \vdots & \vdots &\ddots &\vdots \\
                    -d_{p}(X)&0&\cdots &1\\
                \end{pNiceMatrix}
            \end{equation} 
            as in the proof of Proposition \ref{STBasis} (with $\tilde{\ee}(X)$ as the underlying basis). In Step 2, we arranged the columns of $\tilde{\ee}(X)$ in a way that all the nonvanishing elements in $d_1, \,\cdots, d_p$ are written in the upper left part of $c(X)$. In other words, $c(X)$ can be seen as a block diagonal matrix 
            \begin{equation*}
                c(X)=\mathrm{diag}{(c_\mathdutchcal{s}(X),I_{p-\mathdutchcal{s}})}
            \end{equation*}
            with
            \begin{equation}
                c_s(X)\eqdef\begin{pNiceMatrix}
                    1&0&\cdots &0\\
                    -d_2(X)&1&\cdots &0\\
                    \vdots & \vdots &\ddots &\vdots \\
                    -d_{\mathdutchcal{s}}(X)&0&\cdots &1\\
                \end{pNiceMatrix}.
            \end{equation} 
            Meanwhile, since $\delta^{-}_{j}(X)=\Delta(X)$ when $X\in \mathcal{V}\cap\mathbb{Y}$ for every $j\in\{1,\,\cdots, \mathdutchcal{s}\}$, 
            \begin{equation*}
                \delta(X)=\mathrm{diag}\left(\mat{\Delta}(X)I_\mathdutchcal{s}, I_{p-\mathdutchcal{s}},I_{n-p}\right).
            \end{equation*}
            
            We make use of $c(X)$ and its properties to define
            \begin{equation*}
                \mathfrak{s}(X)=\mathrm{diag}\left(c(X), I_{n-p}\right),
            \end{equation*}
            which is nonsingular and hence could be regarded as a legitimate change of variables. What is more, $\mathfrak{s}(X)\delta(X)\mathfrak{s}^{-1}(X)$ is the product of commuting diagonal blocks, from which it easily follows that $\delta(X)$ is invariant under conjugation by $\mathfrak{s}(X)$. Let $\upsilon_1(X)$ and $\upsilon_2(X)$ be square matrices of dimensions $\mathdutchcal{s}\times \mathdutchcal{s}$ such that
            \begin{equation*}
                \upsilon_1(X)\eqdef\begin{pmatrix}
                        1&0&\hdots &0\\
                        0&\mat{\Delta}(X) &\hdots&0 \\
                        \vdots&\vdots &\ddots &\vdots\\
                        0&0&\hdots&\mat{\Delta}(X)\\ 
                    \end{pmatrix} \quad\; \textrm{and} \quad\; \upsilon_2(X)\eqdef\begin{pmatrix}
                        \mat{\Delta}(X)&0&\hdots &0\\
                        0&1 &\hdots&0 \\
                        \vdots&\vdots &\ddots &\vdots\\
                        0&0&\hdots&1\\ 
                    \end{pmatrix}.
            \end{equation*}
            Note that $\delta(X)$ can be factored as
            \begin{align}\label{Factorization}
                \delta(X)=\begin{pmatrix}
                    \mat{\Delta}(X)I_{\mathdutchcal{s}}&&\\
                    &I_{p-\mathdutchcal{s}}&\\
                    &&I_{n-p}\\
                \end{pmatrix}=\begin{pmatrix}
                    \upsilon_1(X)\upsilon_2(X)&&\\
                    &I_{p-\mathdutchcal{s}}&\\
                    &&I_{n-p}\\
                \end{pmatrix},
            \end{align}
            so \small
            \begin{align}\label{delta_0b}
                \delta(X) &=\mathfrak{s}(X)\begin{pNiceMatrix}
                    \upsilon_1(X)&&\\
                    &I_{p-\mathdutchcal{s}}&\\
                    &&I_{n-p}\\ 
                \end{pNiceMatrix}\begin{pNiceMatrix}
                    \upsilon_2(X)&&\\
                    &I_{p-\mathdutchcal{s}}&\\
                    &&I_{n-p}\\
                \end{pNiceMatrix}\mathfrak{s}^{-1}(X)\\&=\begin{pNiceMatrix}
                    \begin{pmatrix}
                        c_\mathdutchcal{s}(X)\upsilon_1(X)&\\
                        &I_{p-\mathdutchcal{s}}\\
                    \end{pmatrix}&      \\
                    &I_{n-p}
                \nonumber\end{pNiceMatrix}\begin{pmatrix}
                    \begin{pmatrix}
                        v_2(X)&\\
                        & &I_{p-\mathdutchcal{s}}\\
                    \end{pmatrix}c^{-1}(X)& \\
                    &I_{n-p} 
                \end{pmatrix}\\&=\begin{pNiceMatrix}
                    \begin{pmatrix}
                        c_\mathdutchcal{s}(X)\upsilon_1(X)&\\
                        &I_{p-\mathdutchcal{s}}\\
                    \end{pmatrix}\mathfrak{p}^{-1}(X)&      \\
                    &I_{n-p}
                \nonumber\end{pNiceMatrix}\begin{pmatrix}
                    \mathfrak{p}(X)\begin{pmatrix}
                        \mat{\Delta}(X)& &\\
                        &1 & &\\
                        & &\ddots &\\
                        && &1\\ 
                    \end{pmatrix}c^{-1}(X)& \\
                    &I_{n-p} 
                \end{pmatrix}.
            \end{align}\normalsize
            In this manner, if  
            \begin{equation*}
                q(X)\eqdef \begin{pmatrix}
                        c_\mathdutchcal{s}(X)\upsilon_1(X)&\\
                        &I_{p-\mathdutchcal{s}}\\
                    \end{pmatrix}\mathfrak{p}^{-1}(X),
            \end{equation*}
            Proposition \ref{STBasis} and equation \eqref{delta_0b} enables to conclude that 
            \begin{equation*}
                \delta(X)=\mathrm{diag}\left(q(X)b^{-}(X),I_{n-p}\right).
            \end{equation*}
            To finalize, let 
             \begin{equation*}
                    b^{+}(X)\eqdef\begin{pmatrix}
                        b(X)\tilde{e}_{p+1}(X) &\cdots& b(X)\tilde{e}_{n}(X)
                    \end{pmatrix}.
             \end{equation*}
            If $m(X)\eqdef \begin{pmatrix}
                        I_p &q(X)b^{+}(X)
                    \end{pmatrix}$, then
            \begin{align*}
                q(X)b(X)=\begin{pmatrix}
                        q(X)b^{-}(X) &q(X)b^{+}(X)
                    \end{pmatrix}&=\begin{pmatrix}
                        I_p &q(X)b^{+}(X)
                    \end{pmatrix}\begin{pmatrix}
                        \delta^{-}(X)&\\
                        &I_{n-p}\\
                    \end{pmatrix}\\&=m(X)\delta(X).
            \end{align*}
            \textit{Step 6. $\delta(X)$ degenerates at critical points.} By construction, $\ker{\delta(X)}$ is nontrivial if and only if $\mat{\Delta}(X)=0$, i.e., if and only if $X\in \Gamma$. That $\ker{\delta(X)}$ is $\mathdutchcal{s}-$dimensional is evident from Step 2. 
        \end{proof}
        \begin{Definition}\label{LopWaveOperator}
            A collection of pseudodifferential operators $\Delta_{\gamma}(t,x,D_t,D_y)\in \mathrm{OPS}_{\gamma}^{0}(\R^{1+d}_{+}\times [\gamma_0, +\infty))$ is a \Lop family of operators if $\Delta_{\gamma}\eqdef \Op_{\gamma}(\delta)$ and $\delta(X)\equiv\delta_\gamma(t,x,\tau,\eta)$ satisfies  Theorem \ref{SymWR}.
        \end{Definition}
        
        Before continuing further, we present and prove a crucial lemma that simplifies the zeroth-order terms in Problem \ref{SystemPDO}. Morally, if $\mathcal{E}_{0,\gamma}$ is a pseudodifferential operator whose symbol is the nonsingular matrix $\tilde{\ee}_0(X)$ found in Theorem \ref{SymWR}, we look for a correction of $\mathcal{E}_{0,\gamma}$ by an operator of order $-1$, say $\mathcal{E}_{-1,\gamma}$, for which $D_d+\mathdutchcal{A}_{\gamma}$ is block diagonal up to an error of order $-1$. Specifically, we have:

        \begin{Lemma}[{{Lemma 1, \cite{coulombel_2004}}}]\label{LowerOrderTerms}
            Consider $(P_{\gamma},B_{\gamma})$ as in Problem \eqref{SystemPDO} with $a_1$ and $a_0$ being the first two elements of the asymptotic expansion of $a$. Under the notation of Theorem \ref{SymWR}, we can define a symbol $\ee_{-1}(X)$ on $\mathcal{V}$ such that $\ee_{-1}(X)$ is homogeneous of order $-1$ and 
            \begin{align*}
                &(\tilde{\ee}_0+\ee_{-1})(a_1+a_0) - (\dot{a}_1+\ddot{a}_0)(\tilde{\ee}_0+\ee_{-1}) + D_{d}\tilde{\ee}_0 + \frac{1}{i} \sum_{k=0}^{d-1}\left({\partial_{\eta_k}\tilde{\ee}_0\partial_{x_{k}}\dot{a}_1}-\partial_{\eta_k}\dot{a}_1\partial_{x_k}\tilde{\ee}_0\right) 
            \end{align*}
            is a symbol of order $-1$, where $\ddot{a}_0$ is a block diagonal symbol of order $0$ with blocks having dimensions $\alpha_1, \,\cdots, \alpha_q$ as those of $\dot{a}_1$. 
        \end{Lemma}
        \begin{proof}
            Let $\ee_{-1}$ be a symbol of order $-1$ to be determined. A first-order approximation of $(\tilde{\ee}_0+\ee_{-1})^{-1}$ shows that 
            \begin{align}\label{Parametrix}
                (\tilde{\ee}_0+\ee_{-1})(\tilde{\ee}_0^{-1}-\tilde{\ee}_0^{-1}\ee_{-1}\tilde{\ee}_0^{-1}) = I_n \mod S_{\gamma}^{-1},
            \end{align} 
            so $(\tilde{\ee}_0+\ee_{-1})(\xi_d I_n+a_1+a_0)(\tilde{\ee}_0+\ee_{-1})^{-1}$ can be estimated up to an error of order $-1$ by
            \begin{align*}
                &(\tilde{\ee}_0+\ee_{-1})(\xi_d I_n+a_1+a_0)(\tilde{\ee}_0^{-1}-\tilde{\ee}_0^{-1}\ee_{-1}\tilde{\ee}_0^{-1}) \\
                &\qquad = (\xi_d I_n + \tilde{\ee}_0a_1\tilde{\ee}_0^{-1} - \tilde{\ee}_0a_1 \tilde{\ee}_0^{-1}\ee_{-1}\tilde{\ee}_0^{-1} + \ee_{-1}a_1\tilde{\ee}_0^{-1} + \tilde{\ee}_0a_0 \tilde{\ee}_0^{-1}) \mod S_{\gamma}^{-1}.
            \end{align*} 
            Since $\tilde{\ee}_0a_1\tilde{\ee}_0^{-1} = \dot{a}_1$, it is true that
            \begin{align*}
                &\tilde{\ee}_0a_1 \tilde{\ee}_0^{-1}\ee_{-1}\tilde{\ee}_0^{-1} - \ee_{-1}a_1\tilde{\ee}_0^{-1} = \dot{a}_1 \ee_{-1}\tilde{\ee}_0^{-1} - \ee_{-1}\tilde{\ee}_0^{-1}\dot{a}_1 = [\dot{a}_1,\ee_{-1}\tilde{\ee}_0^{-1}],
            \end{align*} 
            and consequently
            \begin{align*}
                (\tilde{\ee}_0+\ee_{-1})(\xi_d I_n+a_1+a_0)(\tilde{\ee}_0^{-1}-\tilde{\ee}_0^{-1}\ee_{-1}\tilde{\ee}_0^{-1}) = \xi_d I_n + \dot{a}_1 - [\dot{a}_1,\ee_{-1}\tilde{\ee}_0^{-1}]+\tilde{\ee}_0a_0 \tilde{\ee}_0^{-1}
            \end{align*} 
            modulo $S_{\gamma}^{-1}$. As we always do with involved computations, we shall omit the parameter $\gamma$ to facilitate the exposition. That being so, let $\dot{A}_1\eqdef\Op(\dot{a}_1)$ and $\ddot{A}_0\eqdef\Op(\ddot{a}_0)$. We now utilize the usual symbolic calculus on the operator equation 
            \begin{align}\label{CompoOp}
                (\mathcal{E}_{0}+\mathcal{E}_{-1})(D_{d} + \mathdutchcal{A}) - (D_{d}+(\dot{A}_1+\dot{A}_0))(\mathcal{E}_{0} + \mathcal{E}_{-1}) = 0 \mod \Psi_{-1}
            \end{align} 
            to derive precise conditions on $\ddot{a}_0$. As a matter of fact, a first-order expansion of the symbol of \eqref{CompoOp} produces
            \begin{align}\label{expansion}
                \nonumber(\tilde{\ee}_0+\ee_{-1})(a_1+a_0) &- (\dot{a}_1+\ddot{a}_0)(\tilde{\ee}_0+\ee_{-1})\\
                &+D_{d} \tilde{\ee}_0+\frac{1}{i}\sum_{k=0}^{d-1}\left({\partial_{\eta_k}\tilde{\ee}_0\partial_{x_{k}}a_1}-\partial_{\eta_k}\dot{a}_1\partial_{x_k}\tilde{\ee}_0\right)=0,
            \end{align} 
            or more concisely,
            \begin{align}\label{IdentOne}
            -[\dot{a}_1,\ee_{-1}\tilde{\ee}_0^{-1}]\tilde{\ee}_0+\tilde{\ee}_0a_0-\ddot{a}_0\tilde{\ee}_0+D_{d} \tilde{\ee}_0+\frac{1}{i}\sum_{k=0}^{d-1}\left({\partial_{\eta_k}\tilde{\ee}_0\partial_{x_{k}}a_1}-\partial_{\eta_k}\dot{a}_1\partial_{x_k}\tilde{\ee}_0\right)=0 \hspace{-0.1cm} \mod \Psi_{-1}.
            \end{align} 
            Alternatively, if we multiply \eqref{IdentOne} from the right by $\tilde{\ee}_0^{-1}$ and put $\dot{a}_0=e_0a_0e^{-1}_0$, then
            \begin{align*}
            \nonumber\ddot{a}_0=-[\dot{a}_1,\ee_{-1}\tilde{\ee}_0^{-1}]+\dot{a}_0+(D_{d} \tilde{\ee}_0)\tilde{\ee}_0^{-1}+\frac{1}{i}\sum_{k=0}^{d-1}\left({\partial_{\eta_k}\tilde{\ee}_0\partial_{x_{k}}a_1}-\partial_{\eta_k}\dot{a}_1\partial_{x_k}\tilde{\ee}_0\right)\tilde{\ee}_0^{-1} \quad \mod \Psi_{-1}.
            \end{align*} 
            In general, there is no reason to expect that $\ddot{a}_0$ above is block diagonal. Yet, we can choose the off-diagonal entries of $[\dot{a}_1,\ee_{-1}\tilde{\ee}_0^{-1}]$ (it is worth remembering that $[\dot{a}_1,\ee_{-1}\tilde{\ee}_0^{-1}]$ has zero diagonal) to compensate those of
            \begin{align*}
                \dot{a}_0+(D_{d} \tilde{\ee}_0)\tilde{\ee}_0^{-1}+\frac{1}{i}\sum_{k=0}^{d-1}\left({\partial_{\eta_k}\tilde{\ee}_0\partial_{x_{k}}a_1}-\partial_{\eta_k}\dot{a}_1\partial_{x_k}\tilde{\ee}_0\right)\tilde{\ee}_0^{-1}.
            \end{align*} 
            As all the terms we have neglected so far are of lower order, the operator $D_{d} + \dot{A}_1 + \ddot{A}_0$ is a block diagonalization of $D_{d} + \mathdutchcal{A}_{\gamma}$ modulo an error of order $-1$. 
        \end{proof}   
        We continue with the main result of this paper, which we recall here for the benefit of the reader.
        \begin{customthm}{1.1}
            Let
            \begin{equation}\label{PseudoBVP}
            \left\{\begin{aligned}
                     P_{\gamma}u_{\gamma}(t,x)\eqdef\left(D_{d}+\mathdutchcal{A}_{\gamma}(t,y,x_d,D_t,D_y)\right)u(t,x)&=f(t,x) &\quad &\quad (t,x)\in \R^{1+d}_{+},\\
                     B_{\gamma}(t,y)u(t,y)|_{x_d=0}&=g(t,y) &\quad &\quad (t,y)\in \R^{d},
                \end{aligned}\right.
            \end{equation}
            where  $\mathdutchcal{A}_{\gamma} \in \mathrm{OPS}_\gamma^{1}(\R^{1+d}_{+}\times [1, +\infty))$ and $B_{\gamma}\in \mathrm{OPS}_\gamma^{0}(\R^{d}\times [1, +\infty))$ are classical pseudodifferential operators with matrix-valued symbols $a(X)$ and $b(X)$ of dimensions $n\times n$ and $p\times n$, respectively. Assume that $P_{\gamma}$ is hyperbolic in the sense of Definition \ref{Hyperbolicity}, $P_{\gamma}$ and $B_{\gamma}$ satisfy Property \textbf{(C)}, and $p=\dim\mathbb{E}^{-}(X)$. Then there exist
            \begin{enumerate}[label=\normalfont(\roman*)]
               \item $\gamma_0\geq 1$,
               \item a family of pseudodifferential operators $\Delta_{\gamma}(t,x,D_t,D_y)\in\mathrm{OPS}_\gamma^{0}(\R^{1+d}_{+}\times [\gamma_0, +\infty))$,
               \item function spaces 
                   \begin{align*}
                      L^2_{\Delta}&\eqdef\{v \in \mathcal{S}'(\R^{1+d}_+,\,\R^n): \Delta_{\gamma}v \in L_{\gamma}^2(\R^{1+d}_+,\,\C^n)\},\\
                      H^s_{\Delta}&\eqdef\{v \in \mathcal{S}'(\R^{1+d}_+,\,\R^n): \Lambda^s_{\gamma} v \in L_{\Delta}^2(\R^{1+d}_+,\,\C^n)\},
                   \end{align*} 
               \item and a positive constant $C$ such that,
            \end{enumerate}
            % \begin{enumerate}[label=\normalfont(\alph*)]
            %     \item a
            %     \item b
            % \end{enumerate}
            if $f\in L_{\gamma}^2(\R^{1+d}_{+})$ and $g\in L_{\gamma}^{2}(\R^{d})$, then for all $\gamma\geq \gamma_0$ and every $u\in \mathcal{D}(\R^{1+d}_{+})$ the following estimate holds
            \begin{align}
                \gamma\Vert{\Delta_{\gamma} u }\Vert_{0,\gamma}^2+\vert{\Delta_{\gamma} u(0)}\vert_{0,\gamma}^2&\leq C\left(\frac{1}{\gamma}\Vert{f}\Vert_{0,\gamma}^2+\vert{g}\vert_{0,\gamma}^2\right).
            \end{align}
            More generally, if $f\in H_{\gamma}^s(\R^{1+d}_{+})$ and $g\in H_{\gamma}^{s}(\R^{d})$,
            \begin{align}
                \gamma\Vert{\Delta_{\gamma} u }\Vert_{s,\gamma}^2+\vert{\Delta_{\gamma} u(0)}\vert_{s,\gamma}^2&\leq C\left(\frac{1}{\gamma}\Vert{f }\Vert_{s,\gamma}^2+\vert{g}\vert_{s,\gamma}^2\right).
            \end{align}
            
            Suppose, in addition, that the nullity of the principal symbol of $\Delta_{\gamma}$ is independent of $X\in \mathbb{X}_{S}$. If $f\in L_{\Delta}^2(\R^{1+d}_{+})$ and $g\in L_{\Delta}^{2}(\R^{d})$, there exists a pseudodifferential operator $Y_{\gamma}(t,y,D_t,D_y)\in \mathrm{OPS}_\gamma^{0}(\R^{d}_{+}\times [\gamma_0, +\infty))$ so that, for all $\gamma\geq \gamma_0$ and every $u\in \mathcal{D}(\R^{1+d}_{+})$,
            \begin{align}
                \gamma\Vert{\Delta_{\gamma} u }\Vert_{0,\gamma}^2+\vert{\Delta_{\gamma} u(0)}\vert_{0,\gamma}^2&\leq C\left(\frac{1}{\gamma}\Vert{\Delta_{\gamma}f}\Vert_{0,\gamma}^2+\vert{\Delta_{\gamma}Y_{\gamma}g}\vert_{0,\gamma}^2\right).
            \end{align}
            More generally, if $f\in H_{\Delta}^s(\R^{1+d}_{+})$ and $g\in H_{\Delta}^{s}(\R^{d})$,
            \begin{align}
                \gamma\Vert{\Delta_{\gamma} u }\Vert_{s,\gamma}^2+\vert{\Delta_{\gamma} u(0)}\vert_{s,\gamma}^2&\leq C\left(\frac{1}{\gamma}\Vert{\Delta_{\gamma}f }\Vert_{s,\gamma}^2+\vert{\Delta_{\gamma}Y_{\gamma}g}\vert_{s,\gamma}^2\right).
            \end{align}
        \end{customthm}
        \begin{proof}
            In the interest of not overloading the notation, we adhere to the following conventions:
            \begin{itemize}[label=$\circ$]
                \item We shall suppress the parameter $\gamma$ all through the calculations, except when its presence is relevant to the point being made (e.g. when we wish to emphasize the existence of a parameter-dependent family of pseudodifferential operators).
                % \item The parameter $\gamma$ will be suppressed all through the following calculations unless otherwise noted.
                \item $\Psi_{m}$ represents an error of order $m$ that may be different from line to line.
                \item When it comes to norms, we shall write $\Vert{\,\cdot\,}\Vert_{s,\gamma}=\Vert{\,\cdot\,}\Vert_s$ or $\Vert{\,\cdot\,}\Vert_{0,\gamma}=\Vert{\,\cdot\,}\Vert$ when $s=0$ (resp. $\vert{\,\cdot\,}\vert_{s,\gamma}=\vert{\,\cdot\,}\vert_s$ or $\vert{\,\cdot\,}\vert_{0,\gamma}=\vert{\,\cdot\,}\vert$ when $s=0$).
                \item We shall adopt $\mathdutchcal{A}_{\gamma}(x_d)$ and $\Delta_{\gamma}(x_d)$ as a substitute for
                \begin{equation*}
                    \mathdutchcal{A}_{\gamma}(t,y,x_d,D_t, D_y) \qquad \mathrm{and} \qquad \Delta_{\gamma}(t,y,x_d,D_t, D_y).
                \end{equation*}
            \end{itemize}
        
            Once again, we shall divide the analysis into steps for ease of explanation. 
            
            \textit{Step 1. We pick a pseudodifferential partition of unity}. Thanks to homogeneity and Property $\textbf{(C)}$, we can restrict our attention to a compact region $K\times S^d\subset\mathbb{X}_S$,  which we may cover with finitely many neighborhoods $\{\mathcal{V}_i\}_{i\in I}$ as shown in Theorem \ref{SymWR}. Let $\{\varphi_i\}_{i\in I}$ and $\{\theta_i\}_{i\in I}$ be systems of functions such that 
            \begin{enumerate}[label=\normalfont(\roman*)]
                \item $\varphi_i, \theta_i \in C_c(\mathcal{V}_i)$, $0\leq \varphi_i \leq 1, \; 0\leq \theta_i \leq 1$, 
                \item $\sum_{i\in I}\varphi^2_i\equiv 1$,
                \item $\theta_i\equiv 1$ in a vicinity of $\supp\varphi_i$.
            \end{enumerate}
            In addition, let us assume that $\Theta_i$ and $\Phi_i$ are pseudodifferential operators whose symbols are the extensions of $\theta_i$ and $\varphi_i$ to homogeneous functions of degree $0$ in $\zeta=(\tau, \eta, \gamma)$. 
            % Let $\{\varphi_i\}_{i\in I}$ be a partition of unity subordinate to $\{\underline{\mathcal{V}}_i\}_{i\in I}$ and let $\{\theta_i\}_{i\in I}$ be a system of functions such that $\theta_i\in C_c(\mathcal{V}_i)$, and $\theta_i\equiv 1$ in a vicinity of $\supp\varphi_i$.
            
            \textit{Step 2. We perform a change of variables in Problem \ref{PseudoBVP}}. To begin with, let us fix $\mathcal{V}\equiv\mathcal{V}_i$, $\theta\equiv\theta_i$, and take $\tilde{\ee}_0\equiv \tilde{\ee}_{0,i}$ as in Theorem \ref{SymWR}. Putting $\mathcal{E}_{0}\eqdef\Op{(\theta\tilde{\ee}_0)}$ and  $\mathcal{E}^{-1}_{0}\eqdef\Op{(\theta\tilde{\ee}_0^{-1})}$, it is readily verified that $\mathcal{E}_{0}\mathcal{E}^{-1}_{0}=I_n \mod\,\Psi_{-1}$, which justifies the abuse of notation in writing $\mathcal{E}^{-1}_{0}$ to refer to $\Op{(\theta\tilde{\ee}_0^{-1})}$ (in rigour, only the inverse of $\mathcal{E}_{0} \mod\,\Psi_{-1}$).  Applying $\mathcal{E}^{-1}_0$ on both sides of $(D_d+\mathdutchcal{A})u=f$ bring us to
            \begin{align}
               \mathcal{E}^{-1}_0D_{d}u+\mathcal{E}^{-1}_0\mathdutchcal{A}u=\mathcal{E}^{-1}_0f,
           \end{align}
           or alternatively to  
           \begin{align}\label{Pre-P}
               D_{d}\mathcal{E}^{-1}_0u+\mathcal{E}^{-1}_0\mathdutchcal{A}u+[\mathcal{E}^{-1}_0,D_d]u=\mathcal{E}^{-1}_0f. 
           \end{align}
           The equivalence 
           \begin{equation*}
               \mathcal{E}^{-1}_0\mathdutchcal{A}u=\mathcal{E}^{-1}_0\mathdutchcal{A}\mathcal{E}_0\mathcal{E}^{-1}_0u+\mathcal{E}^{-1}_0\mathdutchcal{A}\Psi_{-1}u=\mathcal{E}^{-1}_0\mathdutchcal{A}\mathcal{E}_0\mathcal{E}^{-1}_0u+\mathcal{E}^{-1}_0\mathdutchcal{A}\Psi_{-1}\mathcal{E}_0\mathcal{E}^{-1}_0u+\mathcal{E}^{-1}_0\mathdutchcal{A}\Psi_{-2}u           
           \end{equation*}
           modulo an error of order $-1$ let us recast \eqref{Pre-P} succinctly as
           \begin{equation*}
               (D_d+\dot{\mathdutchcal{A}}+E_0)\dot{u}=\dot{f} \qquad \mod\,\Psi_{-1},
           \end{equation*}
           with $\dot{u}\eqdef\mathcal{E}^{-1}_0u$, $\dot{\mathdutchcal{A}}\eqdef\mathcal{E}^{-1}_0\mathdutchcal{A}\mathcal{E}_0$, $E_0\eqdef[\mathcal{E}^{-1}_0,D_d]\mathcal{E}_0 +\mathcal{E}^{-1}_0\mathdutchcal{A}\Psi_{-1}\mathcal{E}_0$, and $\dot{f}\eqdef\mathcal{E}^{-1}_0f$. If we think of $E_0$ as part of $\dot{\mathdutchcal{A}}$, 
           Lemma \ref{LowerOrderTerms} implies the existence of a refined basis $\mathcal{E}=\mathcal{E}_{0}+\mathcal{E}_{-1}$ with respect to which $D_d+\underline{\mathdutchcal{A}}$ is a block diagonalization of $D_d+\dot{\mathdutchcal{A}}+E_0$ modulo $\Psi_{-1}$. In the same vein, if $\dot{B}\eqdef B\mathcal{E}$, we notice that $\dot{B}\mathcal{E}^{-1}\eqdef (B\mathcal{E})\mathcal{E}^{-1}$ differs from $B$ by an error of order $-1$. The resulting boundary problem is
           \begin{equation}\label{dot_u}
            \left\{\begin{aligned}
                \underline{P}\dot{u}(t,x)\eqdef\left(D_{d}+\underline{\mathdutchcal{A}}\right)\dot{u}(t,x)&=\dot{f}(t,x) \\
                \dot{B}\dot{u}(t,y)&=\dot{g}(t,y),
            \end{aligned}\right.
            \end{equation}
            with $\dot{g}\equiv g$, and $\underline{\mathdutchcal{A}}$ being a classical pseudodifferential operator with symbol 
            \begin{equation*}
                \underline{a}\sim \dot{a}_1+\ddot{a}_0+\cdots,
            \end{equation*}
            where $\dot{a}_1$ and $\ddot{a}_0$ are block diagonal.
           
           \textit{Step 3. We localize $u$ by means of $\underline{\Phi}_i\eqdef \Phi_i^*\Phi_i$}. To do so, we fix $i$ such that $\underline{\Phi}\equiv\underline{\Phi}_i$ and observe that the commutator relations
           \begin{align*}
               (D_d+\underline{\mathdutchcal{A}})\underline{\Phi} &=\underline{\Phi}(D_d+\underline{\mathdutchcal{A}})+[(D_d+\underline{\mathdutchcal{A}}), \underline{\Phi}],\\
               \dot{B}\underline{\Phi} &=\underline{\Phi} \dot{B}+[\dot{B},\underline{\Phi}],
           \end{align*}
           enable  us to formulate \eqref{dot_u} in terms of $\tilde{u}\eqdef\underline{\Phi} \dot{u}$ at the expense of a zeroth-order term $[\underline{P}\,, \underline{\Phi}]$ and a harmless error $[\dot{B},\underline{\Phi}]$ of order $-1$ to be analyzed shortly. Thus, we are left with 
           %   that we can afford thanks to Proposition \ref{DeltaHminus}. 
           \begin{equation}\label{ProblemAfterDiag}
            \left\{\begin{aligned}
                \underline{P}\tilde{u}=\left(D_{d}+\underline{\mathdutchcal{A}}\right)\tilde{u}(t,x)&=\tilde{f}(t,x) \\
                \dot{B}\tilde{u}(t,y)&=\tilde{g}(t,y),
            \end{aligned}\right.
            \end{equation}
            where the $\smallsim$ everywhere refers to the application of the operator $\underline{\Phi}$. In this setting, we can define pseudodifferential operators $Q\eqdef \Op(\theta q)$, $\Delta\eqdef \Op(\theta\delta)$, and $M\eqdef \Op(\theta m)$ with $q, \delta$, and $m$ as in Theorem \ref{SymWR}, and use them to set the auxiliary system 
            \begin{equation}\label{StrongSysWR}
                \left\{\begin{aligned}
                    \left(D_{d}+\underline{\mathdutchcal{A}}\right)\tilde{w}(t,x)&=\Delta \tilde{f}(t,x), \\
                    M\tilde{w}(t,y)&={Q}\tilde{g}(t,y).
                \end{aligned}\right.
           \end{equation}
           Interestingly, \eqref{StrongSysWR} satisfies the uniform \Lop condition, for $m(X)$ restricted to the stable subspace of $\dot{a}_1$ is the identity $I_p$ (see Step 5 in Theorem \ref{SymWR}). Hence, from the standard theory of linear hyperbolic boundary value problems (see Chapter 7 in \cite{Chazarain2011} or Chapter 4 in \cite{benzoni-gavage_serre_2007}), it is possible to find a (strong) symmetrizer with the features described in Definition \ref{SFS}, for which we momentarily reintegrate $\gamma$ and $x_d$ to leave no room for ambiguity. 
           \begin{Definition}\label{SFS}
               A (strong) symmetrizer for \eqref{StrongSysWR} is a family $R_{\gamma}$ of $C^1$ operator-valued maps parameterized by $x_d$ so that, for $\gamma\geq \gamma_0\geq 1$,
               \begin{enumerate}[label=\roman*)]
                    \item \label{SFS2} $R_{\gamma}(x_d)$ and $\partial_d R_{\gamma}(x_d)$ are $L^2-$bounded operators with uniform bounds in $x_d$ and $\gamma$.
                    \item \label{SFS1} $R_{\gamma}(x_d)$ is self-adjoint.
                    \item \label{SFS3} There is a positive constant $c$, independent of $x_d$ and $\gamma$, such that
                    \begin{equation*}
                        \ima{\langle{R_{\gamma}(x_d)\underline{\mathdutchcal{A}}_{\gamma}(x_d)v,v}\rangle}\geq c\gamma \vert{v}\vert_{0,\gamma}^2
                    \end{equation*}
                    for every $v\in L^2(\R_+^{1+d})$.
                    \item \label{SFS4} There exist positive constants $\alpha$ and $\beta$ so that
                    \begin{equation*}
                        \langle{R_{\gamma}(0)v,v}\rangle\geq \alpha\vert{v}\vert_{0,\gamma}^2-\beta \vert{M_{\gamma}v}\vert_{0,\gamma}^2
                    \end{equation*}
                    is valid for each $v\in L^2(\R^d)$.
                \end{enumerate}
           \end{Definition}
           As we shall soon confirm, $\Delta_{\gamma}(x_d)$ together with $R_{\gamma}(x_d)$ pave the way for the rest of the proof.
           
           \textit{Step 4. We establish the existence of a $\WR$ symmetrizer}.  We introduce the spaces
           \begin{align*}
              L^2_{\Delta}(\R^{1+d}_{+})&\eqdef\{v \in \mathcal{S}'(\R^{1+d}_+,\,\R^n): \Delta_{\gamma}v \in L_{\gamma}^2(\R^{1+d}_+,\,\C^n)\},\\
              H^s_{\Delta}(\R^{1+d}_{+})&\eqdef\{v \in \mathcal{S}'(\R^{1+d}_+,\,\R^n): \Lambda^s_{\gamma} v \in L_{\Delta}^s(\R^{1+d}_+,\,\C^n)\},
           \end{align*} 
           and the next definition. 
           \begin{Definition}\label{WRFS}
               A $\WR$ symmetrizer for Problem \eqref{ProblemAfterDiag} is a family of pseudodifferential operators $\Sigma_{\gamma}\in \mathrm{OPS}_{\gamma}^{0}(\R^{1+d}_{+}\times [1, +\infty))$ for which there exists $\gamma_0\geq 1$ such that, for all $\gamma\geq \gamma_0$,
               \begin{enumerate}[label=\roman*)]
                    \item \label{SigmaSelfAdjoint} $\Sigma_{\gamma}(x_d)$ is self-adjoint,
                    \item  \label{WRFS0} for every $v_1,v_2\in L^2_{\Delta}$, there is a positive constant $C$ satisfying 
                    \begin{equation*}
                        \langle{\Sigma_{\gamma}(x_d) v_1,v_2}\rangle\leq C\vert{\Delta_{\gamma}(x_d) v_1}\vert\vert{\Delta_{\gamma}(x_d) v_2}\vert,
                    \end{equation*}
                    \item \label{WRFS1} there is a positive constant $c$, independent of $x_d$, so that
                    \begin{equation*}
                        \langle{\partial_d \Sigma_{\gamma}(x_d) v,v}\rangle+2\ima{\langle\Sigma_{\gamma}(x_d)\underline{\mathdutchcal{A}}_{\gamma}(x_d)v,v\rangle}\geq c\gamma \vert{\Delta_{\gamma}(x_d) v}\vert^2
                    \end{equation*}
                    for each $v\in L^2_{\Delta}(\R^{1+d}_{+})$,
                    \item \label{WRFS2} there exist positive constants $\alpha$ and $\beta$ for which
                    \begin{equation*}
                        \langle{\Sigma_{\gamma}(0)v,v}\rangle\geq \alpha\vert{\Delta_{\gamma}(0) v}\vert^2-\beta \vert{Q_{\gamma} \dot{B}_{\gamma}v}\vert^2
                    \end{equation*}
                    holds true for every $v\in L^2_{\Delta}(\R^d)$.
                \end{enumerate}
            \end{Definition}
            Let us drop $x_d$ and $\gamma$ once more to make the idea smoother. We claim that ${\Sigma}={\Delta}^{*}R{\Delta}$ meets (\ref{SigmaSelfAdjoint} to (\ref{WRFS2} in Definition \ref{WRFS}. That $\Sigma$ is self-adjoint follows directly from the self-adjointness of $R$. Condition (\ref{WRFS0}, on the other hand, stems from the elementary computation
            \begin{equation*}
                \langle{\Sigma v_1, v_2}\rangle=\langle{\Delta^{*}R\Delta v_1, v_2}\rangle=\langle{R\Delta v_1, \Delta v_2}\rangle\leq C \vert{\Delta v_1}\vert\vert{\Delta v_2}\vert,
            \end{equation*}
           valid for some constant $C>0$ and test functions $v_1, v_2$ supported on $\mathcal{V}$. The remaining properties can be obtained from those of $R$ and $\Delta$ as explained hereafter. Firstly,
           \begin{align*}
                \langle{\partial_d{\Sigma}v,v}\rangle+2\ima{\langle{{\Sigma} \underline{A}v,v}\rangle}&=\langle{(\partial_d R)\Delta v,\Delta v}\rangle+2\real{\langle{R(\partial_d\Delta)v,\Delta v}\rangle}+2\ima{\langle{R\Delta \underline{A}v,\Delta v}\rangle}\\&=\langle{(\partial_d R)\Delta v,\Delta v}\rangle+2\real{\langle{R(\Op{(\partial_d\delta)})v,\Delta v}\rangle}+2\ima{\langle{R\Delta \underline{A}v,\Delta v}\rangle}.
            \end{align*}
            Secondly, 
            \begin{align}\label{SymProp_2}
                \nonumber\Delta \underline{A}&=\underline{A}\Delta +[\Delta,\underline{A}]\\&=\underline{A}\Delta+\Op([\delta,\dot{a}_1]+[\delta,\ddot{a}_0]-i\{\dot{a}_1,\delta\})+\Psi_{-1}=\underline{A}\Delta-i\Op(\{\dot{a}_1,\delta\}) + \Psi_{-1},
            \end{align}
            given that $[\delta,\dot{a}_1]$ and $[\delta,\ddot{a}_0]$ vanish identically in light of Lemma \ref{LowerOrderTerms}. Inserting \eqref{SymProp_2} into $2\ima{\langle{R\Delta \underline{A}v,\Delta v}\rangle}$ gives
            \begin{align*}
                2\ima{\langle{R\Delta \underline{A}v,\Delta v}\rangle}&= 2\ima{\langle{R\underline{A}\Delta v,\Delta v}\rangle}+2\ima{\langle{R[\Delta, \underline{A}] v,\Delta v}\rangle}\\
                &= 2\ima{\langle{R\underline{A}\Delta v,\Delta v}\rangle}-2\real{\langle{R\Op(\{\dot{a}_1,\delta\})v,\Delta v}\rangle} +\langle{\Psi_{-1}v,\Delta v}\rangle,\\
                &= 2\ima{\langle{R\underline{A}\Delta v,\Delta v}\rangle}+2\real{\langle{R\Op(\{\delta, \dot{a}_1\})v,\Delta v}\rangle} +\langle{\Psi_{-1}v,\Delta v}\rangle,
            \end{align*}
            and eventually, 
            \begin{align}\label{SymProp_3}
                \langle{\partial_d{\Sigma}v,v}\rangle+2\ima{\langle{{\Sigma} \underline{A}v,v}\rangle}=\langle{(\partial_d R)\Delta v,\Delta v}\rangle&+2\ima{\langle{R\underline{A}\Delta v,\Delta v}\rangle}\\\nonumber&+2\real{\langle{R\Op(\partial_d\delta+\{\delta, \dot{a}_1\})v,\Delta v}\rangle},
            \end{align}
            modulo a negligible error $\langle{\Psi_{-1}v,\Delta v}\rangle$ (to be seen!). Now, recalling that $R$ is a strong functional symmetrizer, the second term in \eqref{SymProp_3} obeys the inequality
            \begin{equation*}
                2\ima{\left(R\underline{A}\Delta v,\Delta v\right)}\geq C\gamma \vert{\Delta v}\vert^2,
            \end{equation*}
            whereas the last bracket is null because $\partial_d\delta+\{\delta, \dot{a}_1\}=0$ by construction. 
            Finally, combining the identity $M\Delta=Q\dot{B} \mod\,\Psi_{-1}$ from Theorem \ref{SymWR} and Condition (\ref{SFS4} in Definition \ref{SFS} yields 
            \begin{equation*}
                \langle{\Sigma(0)v(0),v(0)}\rangle\geq \alpha\vert{\Delta(0) v(0)}\vert^2-\beta \vert{M \Delta v(0)}\vert^2=\alpha\vert{\Delta(0) v(0)}\vert^2-\beta \vert{Q\dot{B}v(0)}\vert^2-\vert{\Psi_{-1}v(0)}\vert^2,
            \end{equation*}
            for some positive constants $\alpha, \beta$, and a perturbation $\vert{\Psi_{-1}v(0)}\vert^2$ to be absorbed.
                
            \textit{Step 5. We deduce energy estimates for $\tilde{u}$ via $\Sigma$}. We proceed analogously as in the model case:
            \begin{align}
                \nonumber\frac{d}{dx_d}\langle{\Sigma \tilde{u},\tilde{u}}\rangle&=\langle{\partial_{d}\Sigma \tilde{u}, \tilde{u}}\rangle+\langle{\Sigma\partial_{d}\tilde{u},\tilde{u}}\rangle+\langle{\Sigma \tilde{u},\partial_{d}\tilde{u}}\rangle\\
                \nonumber&=\langle{\partial_{d}\Sigma \tilde{u}, \tilde{u}}\rangle+2\real{\langle{\Sigma iD_{d}\tilde{u},\tilde{u}}\rangle}\\
                \nonumber&=\langle{\partial_{d}\Sigma  \tilde{u},\tilde{u}}\rangle+2\real{\langle\Sigma i(\tilde{f}-\underline{A}\tilde{u}),\tilde{u}\rangle}\\
                \nonumber&=\langle{\partial_{d}\Sigma  \tilde{u},\tilde{u}}\rangle+2\ima{\langle{\Sigma \underline{A}\tilde{u},\tilde{u}}\rangle}-2\ima{\langle{\Sigma \tilde{f},\tilde{u}}\rangle}.
                % -2\ima{\langle{\Sigma E_1u,\tilde{u}}\rangle}.
            \end{align}
            Integrating in $x_d$ over $[0,\infty)$ and bearing in mind that $\tilde{u}$ vanishes at infinity leads to
            \begin{align*}
                \langle{\Sigma \tilde{u}, \tilde{u}}\rangle|_{x_d=0}=&-\int_{0}^{\infty}{\left(\langle{\partial_{d}\Sigma  \tilde{u},\tilde{u}}\rangle+2\ima{\langle{\Sigma \underline{A}\tilde{u},\tilde{u}}\rangle}\right)}\,dx_d+2\int_{0}^{\infty}{\ima{\langle{\Sigma \tilde{f},\tilde{u}}\rangle}}\,dx_d.
            \end{align*}
            From Condition (\ref{WRFS0} in Definition \ref{WRFS}, it is clear that
            \begin{align*}
                2\ima{\langle\Sigma \tilde{f},\tilde{u}\rangle}&\leq 2|\langle\Sigma \tilde{f},\tilde{u}\rangle|\leq C_1 \vert{\Delta \tilde{f}}\vert\vert{\Delta \tilde{u}}\vert,
            \end{align*}
            while from Part (\ref{WRFS1} in Definition \ref{WRFS}, 
            \begin{align*}
                \langle{\Sigma \tilde{u}, \tilde{u}}\rangle|_{x_d=0}&\leq -c\gamma\int_{0}^{\infty}{\vert{\Delta \tilde{u}}\vert^2}\,dx_d+C_1\int_{0}^{\infty}{\vert{\Delta \tilde{f}}\vert\vert{\Delta \tilde{u}}\vert}\,dx_d.
            \end{align*}
            Meanwhile, if we use (\ref{WRFS2} in Definition \ref{WRFS} plus Young's inequality, we arrive at the conclusion that
            \begin{align*}
                \alpha|\Delta \tilde{u}(0)|^2-\beta|Q\dot{B}u(0)|^2&\leq  \left(-c\gamma+\varepsilon\gamma\right)\int_{0}^{\infty}{\vert{\Delta \tilde{u}}\vert^2}\,dx_d+\frac{C_1}{4\varepsilon\gamma}\int_{0}^{\infty}{\vert{\Delta \tilde{f}}\vert^2}\,dx_d
            \end{align*}
            for all $\varepsilon>0$. By making the parameter $\varepsilon$ small enough, we infer that
            \begin{align*}
                \alpha|\Delta \tilde{u}(0)|^2-\beta|Q\dot{B}\tilde{u}(0)|^2&\leq  -C_2\gamma \int_{0}^{\infty}{\vert{\Delta \tilde{u}}\vert^2}\,dx_d+\frac{C_3}{\gamma}\int_{0}^{\infty}{\vert{\Delta \tilde{f}}\vert^2}\,dx_d,
            \end{align*}
            and then by rescaling constants if necessary, 
            \begin{align}\label{EEu_i}
                \gamma\int_{\R^{1+d}_{+}}{e^{-2\gamma t}\vert{\Delta \tilde{u}}\vert^2}\,dt\, dx &+\int_{\R^{d}}{e^{-2\gamma t}\vert{\Delta \tilde{u}(0)}\vert^2}\,dt\, dy 
               \\\nonumber&\leq C\left(\frac{1}{\gamma}\int_{\R^{1+d}_{+}}{e^{-2\gamma t}\vert{\Delta \tilde{f}}\vert^2}\,dt\, dx \right. \left. +\int_{\R^{d}}{e^{-2\gamma t}\vert{Q\dot{B}\tilde{u}(0)}\vert^2\,dt\, dy}\right)
            \end{align}
            for some $C>0$.
            
            \textit{Step 6. We embed $L^2_{\Delta}$ into $H^{s}_{\gamma}$}. More precisely:
            \begin{customproposition}{1.1}\label{DeltaHminus}
                The norms $\Vert{\Delta_{\gamma}\,\cdot\,}\Vert_{0,\gamma}$ lie between $L_{\gamma}^2$ and $H_{\gamma}^{-1}$ for a sufficiently large $\gamma$, that is to say, for $\gamma_0\geq 1$ there exist positive constants $C_1$, $C_2$ such that
                \begin{equation}\label{Embedding}
                    C_1\Vert{\,\cdot\,}\Vert_{-1,\gamma}\leq\frac{1}{\gamma_0} \Vert{\Delta_{\gamma}\,\cdot\,}\Vert_{0,\gamma} \leq C_2\Vert{\,\cdot\,}\Vert_{0,\gamma},
                \end{equation}
                for every $\gamma\in (\gamma_0,+\infty)$.
            \end{customproposition}
            \begin{proof}
                The upper inequality in \eqref{Embedding} is straightforward as soon as one realizes that $\Delta_{\gamma}$ is a family of pseudodifferential operators of order $0$, so only the lower inequality needs to be checked. Let us fix $\gamma_0\geq 1$. We know that  
                \begin{equation*}
                    \lambda_{\gamma}(\zeta)\equiv\lambda(\zeta)=(\gamma^2+\tau^2+\eta^2)^{1/2},
                \end{equation*}
                and 
                \begin{equation*}
                    \mat{\Delta}(X)\equiv\mat{\Delta}(t,y,\zeta)=\frac{\gamma+i\omega(t,y,\zeta)}{\lambda(\zeta)},
                \end{equation*}
                for $\gamma\geq \gamma_1> \gamma_0$. Let us study the cases $x_d=0$ and $x_d>0$ independently for greater clarity. If $x_d=0$, then 
                \begin{equation}\label{Delta_gamma_0}
                    \frac{1}{\lambda(\zeta)}=\frac{\gamma_0}{\gamma_0\lambda(\zeta)}\leq\frac{\sqrt{\gamma^2+\omega^2(t,y,\zeta)}}{\gamma_{0}\lambda(\tau,\eta,\gamma)}=\frac{\vert{\mat{\Delta}(X)}\vert}{\gamma_0}.
                \end{equation}
                In the same spirit, 
                \begin{equation}\label{1overgamma}
                    \frac{1}{\lambda}\leq \frac{1}{\gamma}< \frac{1}{\gamma_0},
                \end{equation}
                and a direct comparison reveals that
                \begin{align*}
                    \frac{1}{\lambda^2}\vert{\hat{u}}\vert^2=\frac{1}{\lambda^2}\hat{u}^2_1+\cdots+\frac{1}{\lambda^2}\hat{u}^2_p+\frac{1}{\lambda^2}\hat{u}^2_{p+1}+\cdots+\frac{1}{\lambda^2}\hat{u}^2_n<& \frac{1}{\gamma^2_0}\vert{\mat{\Delta}(X)\hat{u}_1}\vert^2+\cdots+\frac{1}{\gamma^2_0}\vert{\mat{\Delta}(X)\hat{u}_p}\vert^2+\\&\frac{1}{\gamma_0^2}\hat{u}^2_{p+1}+\cdots+\frac{1}{\gamma_0^2}\hat{u}^2_n=\frac{1}{\gamma^2_0}\vert{\delta(X)\hat{u}}\vert^2.
                \end{align*}
                % If we integrate, take square root, and apply Plancherel's theorem on both sides, we obtain 
                % \begin{equation*}
                %     \Vert{\,\cdot\,}\Vert_{-1,\gamma}\leq\frac{1}{\gamma_0} \Vert{\Delta_{\gamma}\,\cdot\,}\Vert_{0,\gamma}.
                % \end{equation*}
                Moreover, since
                \begin{equation*}
                    c_\delta\eqdef \frac{1}{\gamma_0^2}\delta^{*}(X)\delta(X)-\frac{1}{\lambda^2}
                \end{equation*}
                is diagonal and positive definite, a quick inspection shows that $\sqrt{c_\delta}$ is well-defined, smooth and bounded, as are all of its derivatives, so $\sqrt{c_\delta}\in S^0_{\gamma}(\R^{d}\times \R^{d}\times [\gamma_0, +\infty))$ \footnote{To be exact, $\lambda^{-1}\lesssim \sqrt{c_\delta} \lesssim 1.$}. If $\mathdutchcal{B}_0\eqdef \Op(\sqrt{c_\delta})$, the residue
                \begin{equation*}
                    R_1\eqdef \Delta_{\gamma}^*\Delta_{\gamma}-\Lambda^{-2}_{\gamma}-\mathdutchcal{B}_0^*\mathdutchcal{B}_0
                \end{equation*}
                is a pseudodifferential operator of order $-1$. Arguing inductively, we can determine $\mathdutchcal{B}_{j}=\Op_{\gamma}(\mathdutchcal{b}_j)\in \mathrm{OPS}_{\gamma}^{-j}(\R^{1+d}_{+}\times[\gamma_0, +\infty))$, $j\in \{0,\,\cdots, k-1\}$, so that
                \begin{equation*}
                    R_k\eqdef \Delta_{\gamma}^*\Delta_{\gamma}-\Lambda^{-2}_{\gamma}-(\mathdutchcal{B}_0+\cdots+\mathdutchcal{B}_{k-1})^*(\mathdutchcal{B}_0+\cdots+\mathdutchcal{B}_{k-1})
                \end{equation*}
                is a pseudodifferential operator of order $-k$ by imposing the symbol restriction
                \begin{equation*}
                    \overline{b_k}\sqrt{c_\delta}+\sqrt{c_\delta}b_k=\sigma(R_k).
                \end{equation*}
                It should be stressed that $\sigma(R_k)$ is real as $R_k$ is a self-adjoint operator, so $b_k$ can be chosen as $\frac{1}{2}\sigma(R_k)/\sqrt{c_\delta}$. Lastly, after iterating several times and dominating $R_k$ by $\Lambda^{-2}_{\gamma}$, we conclude that
                \begin{equation*}
                    \Vert{\,\cdot\,}\Vert_{-1,\gamma}\leq\frac{1}{\gamma_0} \Vert{\Delta_{\gamma}\,\cdot\,}\Vert_{0,\gamma}.
                \end{equation*}
                When $x_d>0$, the situation is more intricate and require some effort. Certainly, as $\delta^{-}_j(X)$ solves the transport equation \eqref{TransportEq}, $\delta^{-}_j(X)$ is the composition of $\mat{\Delta}(X)$ with the inverse of the Hamiltonian flow map $\phi_{x_d.j}$ associated with the eigenvalue $a_{1,j}(X)$ (see the proof of Theorem \ref{SymWR}). This means that if $X_{\flat}=(t_\flat,y_\flat,0,\tau_\flat,\eta_\flat,\gamma_\flat)\in \mathbb{Y}$ is such that 
                \begin{equation}\label{FlowComposition}
                    \delta^{-}_j(X)=(\phi^{*}_{-x_d,\,j}\,\mat{\Delta})(X)=\mat{\Delta}(X_\flat),
                \end{equation}
                then 
                \begin{align}\label{delta_gamma_0}
                    \nonumber\frac{1}{\lambda(\zeta)}=\frac{\gamma_0}{\gamma_0\lambda(\zeta)}\leq\frac{\sqrt{\gamma_{\flat}^2+\omega^2(t_\flat,y_{\flat},\zeta_{\flat})}}{\gamma_{0}\lambda(\zeta)}&=\frac{\sqrt{\gamma_{\flat}^2+\omega^2(t_\flat,y_{\flat},\zeta_{\flat})}}{\gamma_{0}\lambda(\zeta_\flat)}\frac{\lambda(\zeta_{\flat})}{\lambda(\zeta)}\\&=\frac{\vert{\Delta(X_{\flat})}\vert}{\gamma_0}\frac{\lambda(\zeta_{\flat})}{\lambda(\zeta)}=\frac{\vert{\delta^{-}_j(X)}\vert}{\gamma_0}\frac{\lambda(\zeta_{\flat})}{\lambda(\zeta)}.
                \end{align}
                In an attempt to control \eqref{delta_gamma_0} appropriately, we must find an upper bound for $\lambda(\zeta_{\flat})/\lambda(\zeta)$. This is addressed in the lemma below. 
                \begin{Lemma}
                    Let $\phi^{-1}\equiv\phi_{-x_d,\,j}$ be the inverse of the Hamiltonian map $\phi_{x_d}$ encoded in equation \eqref{bicharacteristics}. Under the assumptions of Theorem \ref{SymWR}, there exists a constant $C>0$ such that for every $X=(t,y,x_d,\zeta)\in \mathcal{V}$,
                    \begin{equation*}
                        \frac{\lambda(\phi^{-1}(X))}{\lambda(\zeta)}\leq C.
                    \end{equation*}
                \end{Lemma}
                \begin{proof}
                    Let us argue by contradiction and suppose that for every $n\in \N$, there is a $X_n=(t_n,y_n,x_{d,n},\zeta_n)\in \mathcal{V}$ such that 
                    \begin{equation}\label{SeqN}
                        \frac{\lambda(\phi^{-1}(t_n,y_n,x_{d,n},\zeta_n))}{\lambda(\zeta_n)}> n.
                    \end{equation}
                    Considering $\lambda$ and $\phi^{-1}$ are homogeneous of degree $1$ in $\zeta$, we can equally write \eqref{SeqN} as
                    \begin{equation}\label{SeqNHom}
                        \lambda(\phi^{-1}(t_n,y_n,x_{d,n},{\zeta'}_n))> n,
                    \end{equation}
                    where $\zeta'_n\eqdef \zeta_n/\vert{{\zeta}_n}\vert=\zeta_n/\lambda(\zeta_n)$. If we look at the covariables as elements on the sphere $S^d$, the new neighborhood $\mathcal{V}'\subset\mathbb{X}_S$ is compact (as $\mathcal{V}$ can be taken compact in $(t,y,x_d)$ in view of Property \textbf{(C)}), so the sequence $\{X_n\}_{n\in\N}$ has a convergent subsequence, say, $\{X_k\}_{k\in\N}$, so that $X_k\to \underline{X}=(\underline{t},\underline{y},\underline{x}_{d},\underline{\zeta'})$ as $k$ goes to infinity. Then 
                    \begin{align}\label{contra}
                        \nonumber k&<\vert{\lambda(\phi^{-1}(t_k,y_k,x_{d,k},{\zeta'}_k))}\vert\\&=\vert{\lambda(\phi^{-1}(t_k,y_k,x_{d,k},{\zeta'}_k))-\lambda(\phi^{-1}(\underline{t},\underline{y},\underline{x}_{d},\underline{\zeta'}))}\vert+\lambda(\phi^{-1}(\underline{t},\underline{y},\underline{x}_{d},\underline{\zeta'})).
                    \end{align}
                    For $k$ sufficiently large, the difference on the right-hand side of \eqref{contra} can be made arbitrarily small because of the continuity of $\phi^{-1}$ and $\lambda$, meaning that the whole expression can be bounded by some constant $C'$ for $k$ large enough, which is a contradiction. 
                \end{proof}
                Returning to equations \eqref{1overgamma} and \eqref{delta_gamma_0}, we see that 
                \begin{equation*}
                    \frac{1}{\lambda(\zeta)}\leq \frac{C}{\gamma_0}\vert{\delta^{-}_j(X)}\vert
                \end{equation*}
                and 
                \begin{equation*}
                    \frac{1}{\lambda}\leq \frac{C}{\gamma_0}
                \end{equation*}
                for some positive constant $C$. The rest of the proof is identical to that for $x_d=0$.
            \end{proof}
            Proposition \ref{DeltaHminus} allows us to handle $\langle{\Psi_{-1}v,\Delta v}\rangle$ and $\vert{\Psi_{-1}v(0)}\vert^2$ effectively using the philosophy illustrated in the next step.

            \textit{Step 7. We glue the pieces}. Recall that $\Delta$, $Q$, $\tilde{u}$, $\tilde{f}$, and $\tilde{g}$ in \eqref{EEu_i} are indexed by $i$, so we can add the pieces to get
            \begin{align}
                \gamma\sum_{i\in I}\Vert{\Delta_i \tilde{u}_i }\Vert^2+\sum_{i\in I}\vert{\Delta_i \tilde{u}_i(0)}\vert^2&\leq C\left(\frac{1}{\gamma}\sum_{i\in I}\Vert{\Delta_i\tilde{f}_i }\Vert^2+\sum_{i\in I}\vert{Q_i\tilde{g}_i}\vert^2\right).
            \end{align}
            Let 
            \begin{equation*}
                \Delta\eqdef \sum_{i\in I}\mathcal{E}_i\Delta_i\mathcal{E}_i^{-1}\underline{\Phi}_i=\sum_{i\in I}\mathcal{E}_i\Delta_i\underline{\Phi}_i\mathcal{E}_i^{-1} \mod \Psi^{-1}.
            \end{equation*}
            From Proposition \ref{DeltaHminus}, the triangle inequality, and the convexity of the power function $x\to x^2$, it follows that
            \begin{align}\label{EE_lowerbound}
                \Vert{\Delta u}\Vert^2\lesssim \sum_{i\in I}\Vert{\Delta_i \tilde{u}_i}\Vert^2 \qquad \mathrm{and} \qquad
                \vert{\Delta u(0)}\vert^2\lesssim \sum_{i\in I}\vert{\Delta_i \tilde{u}_i(0)}\vert^2.
            \end{align}
           Consequently,
           \begin{align}\label{EE_prefinal}
                \gamma\Vert{\Delta u }\Vert^2+\vert{\Delta u(0)}\vert^2&\leq C\left(\frac{1}{\gamma}\sum_{i\in I}\Vert{\Delta_i\tilde{f}_i }\Vert^2+\sum_{i\in I}\vert{Q_i\tilde{g}_i}\vert^2\right),
            \end{align}
            and then the obvious relations $\Vert{\Delta_i\tilde{f}_i}\Vert\lesssim\Vert{\tilde{f}_i}\Vert\lesssim\Vert{f}\Vert$ and $\vert{Q_i\tilde{g}}\vert\lesssim \vert{\tilde{g}}\vert\lesssim \vert{g}\vert$ yield 
            \begin{align}\label{EE_Final}
                \gamma\Vert{\Delta u }\Vert^2+\vert{\Delta u(0)}\vert^2&\leq C\left(\frac{1}{\gamma}\Vert{f}\Vert^2+\vert{g}\vert^2\right).
            \end{align}
            
            Formula \eqref{EE_Final} is not a strict generalization of \eqref{Delta_Estimates_CC} as $\Delta$ is not present on the right-hand side of the inequality. However, if we further assume that $\mathdutchcal{s}$ remains constant for each $i\in I$, then a variable coefficient version of \eqref{Delta_Estimates_CC} holds true. To understand why, we first resort to the ideas discussed around equation \eqref{BXDY} to obtain
            \begin{align}\label{EE_Final_2}
                \gamma\Vert{\Delta u }\Vert^2+\vert{\Delta u(0)}\vert^2&\leq C\left(\frac{1}{\gamma}\sum_{i\in I}\Vert{\Delta_i\tilde{f}_i}\Vert^2+\sum_{i\in I}\vert{\Delta_i Y_i\tilde{g}_i}\vert^2\right).
            \end{align}
            At this point, consider two charts $\mathcal{V}_i$ and $\mathcal{V}_j$ and the identity
            \begin{equation*}
                \Vert{\mathcal{E}_i\Delta_i\mathcal{E}_i^{-1}\tilde{f}_i+\mathcal{E}_j\Delta_j\mathcal{E}_j^{-1}\tilde{f}_j}\Vert^2=\Vert{\mathcal{E}_i\Delta_i\mathcal{E}_i^{-1}\tilde{f}_i}\Vert^2+2\real{\mathcal\langle\mathcal{E}_i\Delta_i\mathcal{E}_i^{-1}\tilde{f}_i,\mathcal{E}_j\Delta_j\mathcal{E}_j^{-1}\tilde{f}_j\rangle}+\Vert{\mathcal{E}_j\Delta_j\mathcal{E}_j^{-1}\tilde{f}_j}\Vert^2.
            \end{equation*}
            If $\mathcal{V}_i$ and $\mathcal{V}_j$ are disjoint,  $2\real{\mathcal\langle\mathcal{E}_i\Delta_i\mathcal{E}_i^{-1}\tilde{f}_i,\mathcal{E}_j\Delta_j\mathcal{E}_j^{-1}\tilde{f}_j\rangle}$ vanishes identically and
            \begin{equation}\label{Diff}
                \Vert{\mathcal{E}_i\Delta_i\mathcal{E}_i^{-1}\tilde{f}_i+\mathcal{E}_j\Delta_j\mathcal{E}_j^{-1}\tilde{f}_j}\Vert^2=\Vert{\mathcal{E}_i\Delta_i\mathcal{E}_i^{-1}\tilde{f}_i}\Vert^2+\Vert{\mathcal{E}_j\Delta_j\mathcal{E}_j^{-1}\tilde{f}_j}\Vert^2.
            \end{equation}
            In contrast, when $\mathcal{V}_i$ and $\mathcal{V}_j$ are adjacent,  $2\real{\mathcal\langle\mathcal{E}_i\Delta_i\mathcal{E}_i^{-1}\tilde{f}_i,\mathcal{E}_j\Delta_j\mathcal{E}_j^{-1}\tilde{f}_j\rangle}$ is not necessarily zero and must be examined carefully. In particular, it is worth noting that
            \begin{align*}
                2\real{\mathcal\langle\mathcal{E}_i\Delta_i\mathcal{E}_i^{-1}\tilde{f}_i,\mathcal{E}_j\Delta_j\mathcal{E}_j^{-1}\tilde{f}_j\rangle}&=2\real{\mathcal\langle\mathcal{E}_i\Delta_i\mathcal{E}_i^{-1}\underline{\Phi}_if,\mathcal{E}_j\Delta_j\mathcal{E}_j^{-1}\underline{\Phi}_j f\rangle}\\&=2\real{\mathcal\langle\mathcal{E}_i\Delta_i\mathcal{E}_i^{-1}\Phi^*_i\Phi_if,\mathcal{E}_j\Delta_j\mathcal{E}_j^{-1}\Phi^*_j\Phi_j f\rangle}\\&=2\real{\mathcal\langle\mathcal{E}_i\Delta_i\mathcal{E}_i^{-1}\Phi_i\Phi_jf,\mathcal{E}_j\Delta_j\mathcal{E}_j^{-1}\Phi_i\Phi_j f\rangle}+E,
            \end{align*}
            with 
            \begin{equation}\label{ErrorMixed}
                E\eqdef 2\real{\mathcal\langle\mathcal{E}_i\Delta_i\mathcal{E}_i^{-1}\underline{\Phi}_i f,\Psi_{-1}f\rangle}+2\real{\mathcal\langle\Psi_{-1}f,\mathcal{E}_j\Delta_j\mathcal{E}_j^{-1}\underline{\Phi}_j f\rangle}+\Vert{\Psi_{-1}f}\Vert^2
            \end{equation}
            to be controlled later. In order to analyze $2\real{\mathcal\langle\mathcal{E}_i\Delta_i\mathcal{E}_i^{-1}\Phi_i\Phi_jf,\mathcal{E}_j\Delta_j\mathcal{E}_j^{-1}\Phi_i\Phi_j f\rangle}$, let us choose initially any $X\in \supp\varphi_1\cap\supp\varphi_2\subset \mathcal{V}_i\cap\mathcal{V}_j$ so that $x_d=0$. Given that the first $p$ columns of $\mathcal{E}_i(X)$ and $\mathcal{E}_j(X)$ both span $\mathbb{E}^{-}(X)$, the \Lop determinants $\mat{\Delta}_i(X)$ and $\mat{\Delta}_j(X)$ are equivalent and related by a smooth, nonvanishing \footnote{in a neighborhood of $\mathcal{V}_i\cap\mathcal{V}_j$}, complex-valued factor $\mathdutchcal{v}$ as shown
            \begin{equation*}
                \mat{\Delta}_i(X)=\mathdutchcal{v}_{ij}(X)\mat{\Delta}_j(X).
            \end{equation*}
            Hence, if $\underline{\mathdutchcal{v}}_{ij}(X)\eqdef \mathrm{diag}\left(\mathdutchcal{v}_{ij}(X)I_\mathdutchcal{s}, I_{p-\mathdutchcal{s}},I_{n-p}\right)$,
            \begin{align}\label{LopConnection}
                 \delta_i(X)&=\nonumber\mathrm{diag}\left(\mat{\Delta}_i(X)I_\mathdutchcal{s}, I_{p-\mathdutchcal{s}},I_{n-p}\right)=\mathrm{diag}\left(\mathdutchcal{v}_{ij}(X)\mat{\Delta}_j(X)I_\mathdutchcal{s}, I_{p-\mathdutchcal{s}},I_{n-p}\right)\\&\nonumber=\mathrm{diag}\left(\mathdutchcal{v}_{ij}(X)I_\mathdutchcal{s}, I_{p-\mathdutchcal{s}},I_{n-p}\right)\mathrm{diag}\left(\mat{\Delta}_j(X)I_\mathdutchcal{s}, I_{p-\mathdutchcal{s}},I_{n-p}\right)\\&=\underline{\mathdutchcal{v}}_{ij}(X)\delta_j(X).
            \end{align}
            % \textemdash which are known to be invariant under a change of coordinates\textemdash
            According to \eqref{bicharacteristics}, the bicharacteristic curves depend exclusively on the eigenvalues of $a_1(X)$, \textemdash which are known to be invariant under change of coordinates\textemdash, and thus,  as long as their order is respected, the flow maps in \eqref{FlowComposition} coincide, making equation \eqref{LopConnection} valid for $x_d>0$. As a result, if by abuse of notation we think of $\underline{\mathdutchcal{v}}_{ij}$ as the homogeneous extension of degree 0 of itself, we can write $\mathdutchcal{V}_{ij}=\Op(\underline{\mathdutchcal{v}}_{ij})$ and
            \begin{equation}\label{Delta_ij}
                \Delta_i=\Op(\theta_i\delta_i)=\Op(\theta_i\underline{\mathdutchcal{v}}_{ij}\delta_j)=\mathdutchcal{V}_{ij}\Delta_j+\Psi_{-1},
            \end{equation}
            provided that $\Delta_i$ and $\Delta_j$ act on elements of the form $\Phi_i\Phi_j f$. After all, the latter implies that
            \begin{align*}
                \mathcal{E}_{i}\Delta_1\mathcal{E}^{-1}_{i}=\mathcal{E}_{i}(\mathdutchcal{V}_{ij}\Delta_j+\Psi_{-1})\mathcal{E}^{-1}_{i}&=\mathcal{E}_{i}\mathdutchcal{V}_{ij}\Delta_j\mathcal{E}^{-1}_{i}+\mathcal{E}_{i}\Psi_{-1}\mathcal{E}^{-1}_{i}\\&=\mathcal{E}_{i}\mathdutchcal{V}_{ij}(\mathcal{E}^{-1}_{i}\mathcal{E}_{i}+\Psi_{-1})\Delta_j\mathcal{E}^{-1}_{i}+\mathcal{E}_{i}\Psi_{-1}\mathcal{E}^{-1}_{i}\\&=\mathcal{E}_{i}\mathdutchcal{V}_{ij}\mathcal{E}^{-1}_{i}\mathcal{E}_i\Delta_j\mathcal{E}^{-1}_{i}+\Psi_{-1}\\&=\mathcal{E}_{i}\mathdutchcal{V}_{ij}\mathcal{E}^{-1}_{i}\mathcal{E}_{j}\Delta_j\mathcal{E}_{j}^{-1}+\Psi_{-1},
            \end{align*}
            where we have exploited the equality $\ee_{0,i}(X)\delta_j(X)\ee^{-1}_{0,i}(X)=\ee_{0,j}(X)\delta_j(X)\ee^{-1}_{0,j}(X)$, which stems directly from the definition of $\delta_j(X)$. Note that $\dot{\mathdutchcal{V}_{ij}}\eqdef\mathcal{E}_{1}\mathdutchcal{V}_{ij}\mathcal{E}_1^{-1}$ is an elliptic operator when applied to $\Phi_i\Phi_j f$, so 
            \begin{align}\label{MixTerms1}
                \nonumber2\real{\mathcal\langle\dot{\Delta}_i\Phi_i\Phi_jf,\dot{\Delta}_j\Phi_i\Phi_j f\rangle}&=2\real{\mathcal\langle\dot{\mathdutchcal{V}_{ij}}\dot{\Delta}_j\Phi_i\Phi_jf,\dot{\Delta}_j\Phi_i\Phi_j f\rangle}+2\real{\langle{\Psi_{-1}\Phi_i\Phi_j f,\dot{\Delta}_j\Phi_i\Phi_j f}\rangle}\\&\geq c\Vert{\dot{\Delta}_j\Phi_i\Phi_j f}\Vert^2-2\Vert{\Psi_{-1}\Phi_i\Phi_j f}\Vert\Vert{\dot{\Delta}_j\Phi_i\Phi_j f}\Vert,
            \end{align}
            the closing line being justified by Gårding and Cauchy-Schwarz inequalities. In addition, the second term in \eqref{MixTerms1} enjoys the estimate 
            \begin{align}\label{MixTerms2}
                \nonumber\Vert{\Psi_{-1}\Phi_i\Phi_j f}\Vert\Vert{\dot{\Delta}_j\Phi_i\Phi_j f}\Vert&\lesssim \Vert{\Psi_{-1}\mathcal{E}^{-1}_{1}\Phi_i\Phi_j f}\Vert\Vert{\dot{\Delta}_j\Phi_i\Phi_j f}\Vert\\\nonumber&\lesssim \sqrt{\gamma_0}\Vert{\Psi_{-1}\mathcal{E}^{-1}_{1}\Phi_i\Phi_j f}\Vert^2+\frac{1}{4\sqrt{\gamma_0}}\Vert{\dot{\Delta}_j\Phi_i\Phi_j f}\Vert^2\\&\lesssim \frac{1}{\sqrt{\gamma_0}}\Vert{\dot{\Delta}_j\Phi_i\Phi_j f}\Vert^2+\frac{1}{4\sqrt{\gamma_0}}\Vert{\dot{\Delta}_j\Phi_i\Phi_j f}\Vert^2.
            \end{align}
            In the end, combining \eqref{MixTerms1} and \eqref{MixTerms2} and taking $\gamma_0$ (and thereby $\gamma$) sufficiently large, it is easily verified that $2\real{\mathcal\langle\dot{\Delta}_i\Phi_i\Phi_jf,\dot{\Delta}_j\Phi_i\Phi_j f\rangle}>0$. Likewise, from \eqref{MixTerms2} and Proposition \ref{DeltaHminus}, we can absorb $E$ from \eqref{ErrorMixed} too. In this manner, proceeding similarly with $\Delta_iY_i\tilde{g}_i$ and $\Delta_jY_j\tilde{g}_j$, we deduce that
            \begin{equation*}
                \gamma\Vert{\Delta u }\Vert^2+\vert{\Delta u(0)}\vert^2\leq \frac{1}{\gamma}\Vert{\Delta f}\Vert^2+\Vert{\Delta Yg}\Vert^2,
            \end{equation*}
            provided that we pick $\gamma_0$ large enough.
            
            % \textcolor{red}{Since the eigenvalues of $a_1(X)$ are invariant under change of coordinates, flow maps associated \eqref{bicharacteristics}}
            % \textcolor{red}{According to \eqref{bicharacteristics}, the bicharacteristic curves depend solely on the eigenvalues of $a_1$ and therefore are invariant }

            % \textcolor{red}{We need to analyze $2\real{\mathcal\langle\mathcal{E}_i\Delta_i\mathcal{E}_i^{-1}\tilde{f}_i,\mathcal{E}_j\Delta_j\mathcal{E}_j^{-1}\tilde{f}_j\rangle}$. Performing an asymptotic expansion of}
            % \begin{equation}\label{AsyExp}
            %     \real{(\mathcal{E}_j\Delta_j\mathcal{E}_j^{-1})^*\mathcal{E}_i\Delta_i\mathcal{E}_i^{-1}},    
            % \end{equation}
            % \textcolor{red}{we see that the symbol of \eqref{AsyExp} is supported on $\supp\varphi_1\cap\supp\varphi_2$. If $x_d=0$, observe that on $\supp\varphi_1\cap\supp\varphi_2$, the \Lop determinants $\mat{\Delta}_i(X)$ and $\mat{\Delta}_j(X)$ differ by a smooth complex-valued function $v(X)$ that does not vanish in a neighborhood of $X\in \mathcal{V}_i\cap\mathcal{V}_j$, then } 
            % \begin{equation*}
            %     \mat{\Delta}_i(X)=v(X)\mat{\Delta}_j(X).
            % \end{equation*}
            % \textcolor{red}{Note that}
            % \begin{align*}
            %      \delta_i(X)&=\mathrm{diag}\left(\mat{\Delta}(X)_iI_\mathdutchcal{s}, I_{p-\mathdutchcal{s}},I_{n-p}\right)=\mathrm{diag}\left(v(X)\mat{\Delta}_j(X)I_\mathdutchcal{s}, I_{p-\mathdutchcal{s}},I_{n-p}\right)\\&=\mathrm{diag}\left(v(X)I_\mathdutchcal{s}, I_{p-\mathdutchcal{s}},I_{n-p}\right)\mathrm{diag}\left(\mat{\Delta}(X)_jI_\mathdutchcal{s}, I_{p-\mathdutchcal{s}},I_{n-p}\right)\\&=w(X)\delta_j(X).
            % \end{align*}
            
            \textit{Step 8. We extend the argument to Sobolev spaces}. 
            For the sake of clarity, let us remove the index $i$ and the parameter $\gamma$ once again. Since $\Lambda^s_{\gamma} \tilde{u} \in L^2_{\Delta}$ when $\tilde{u}\in H_{\Delta}^s$, $\Lambda^s_{\gamma} \tilde{u}$ fulfills  
             \begin{align}\label{EE_u_Sobolev}
                \gamma\Vert{\Delta_i\Lambda^s\tilde{u}_i}\Vert^2+\vert{\Delta_i\Lambda^s\tilde{u}_i(0)}\vert^2&\leq C\left(\frac{1}{\gamma}\Vert{\Delta_i \underline{P}_i\Lambda^s\tilde{u}_i }\Vert^2+\vert{Q_i\dot{B}_i\Lambda^s\tilde{u}_i}\vert^2\right).
            \end{align}
            Notice that
            % \begin{align*}
            %     \gamma\Vert{\Delta \tilde{u}}\Vert^2+\vert{\Delta \tilde{u}(0)}\vert^2&\leq C\left( \frac{1}{\gamma}\Vert{\Delta\underline{P} \tilde{u}}\Vert^2+\vert{Q\dot{B}\tilde{u}(0)}\vert^2\right)=C\left( \frac{1}{\gamma}\Vert{\Delta \underline{P}\Phi\dot{u}}\Vert^2+\vert{Q\dot{B}\Phi\dot{u}(0)}\vert^2\right),
            % \end{align*}
            % and 
            \begin{align*}
                \Vert{\Delta \underline{P}\Lambda^s\tilde{u}}\Vert^2&\lesssim \Vert{\Lambda^s\Delta \underline{P}\tilde{u}}\Vert^2+\Vert{[\Delta,\Lambda^s]\underline{P}\tilde{u}}\Vert^2+\Vert{\Delta[\underline{P},\Lambda^s]\tilde{u}}\Vert^2\\&\lesssim \Vert{\Lambda^s\Delta \underline{P}\tilde{u}}\Vert^2+\Vert{\underline{P}\tilde{u}}\Vert^2_{-1}+\Vert{\Delta[\underline{P},\Lambda^s]\tilde{u}}\Vert^2\\&\lesssim \Vert{\Delta \tilde{f}}\Vert^2+\frac{1}{\gamma^2_0}\Vert{\Delta\tilde{f}}\Vert^2+\Vert{\Delta[\underline{P},\Lambda^s]\tilde{u}}\Vert^2,
            \end{align*}
            where we have appealed to Proposition \ref{DeltaHminus} in the last part. It remains to check that we can deal with $\Delta[\underline{P}\,,\Lambda^s]$ satisfactorily. In essence, taking into account that $\underline{P}$ is block diagonal modulo $\Psi_{-1}$, $[\underline{P}\,, \Lambda^s]$ is an operator of order $0$ with principal symbol $i\{\Lambda^s, \xi_d+\dot{a}_1\}$, also diagonal, so
            \begin{equation*}
                \Delta[\underline{P}\,,\Lambda^s]=[\underline{P}\,,\Lambda^s]\Delta +\Psi_{-1},
            \end{equation*}
            and therefore, 
            \begin{align}\label{PPhi}
                \Vert{\Delta \underline{P}\Lambda^s\tilde{u}}\Vert^2&\lesssim \Vert{\Delta\tilde{f}}\Vert^2+\frac{1}{\gamma^2_0}\Vert{\Delta\tilde{f}}\Vert^2+\Vert{[\underline{P},\Lambda^s]\Delta\tilde{u}}\Vert^2+\Vert{\tilde{u}}\Vert^2_{-1}\\\nonumber&\lesssim \Vert{\Delta \tilde{f}}\Vert^2+\frac{1}{\gamma^2_0}\Vert{\Delta\tilde{f}}\Vert^2+\Vert{\Delta\tilde{u}}\Vert^2+\frac{1}{\gamma_0^2}\Vert{\Delta \tilde{u}}\Vert^2\\\nonumber&\lesssim \Vert{\Delta \tilde{f}}\Vert^2+\Vert{\Delta\tilde{u}}\Vert^2.
            \end{align}
            Let us now pay attention to $Q\dot{B}\Lambda^s\tilde{u}$. In this case, 
            \begin{align}\label{QBPhi}
                \vert{Q\dot{B}\Lambda^s\tilde{u}(0)}\vert^2&\lesssim \vert{\Lambda^s Q\dot{B}\tilde{u}(0)}\vert^2+\vert{[Q,\Lambda^s]\dot{B}\tilde{u}(0)}\vert^2+\vert{Q[\dot{B},\Lambda^s]\tilde{u}(0)}\vert^2\\\nonumber&\lesssim \vert{\Lambda^s Qg}\vert^2+\vert{\tilde{u}(0)}\vert_{-1}^2\\\nonumber&\lesssim \vert{Qg}\vert^2+\frac{1}{\gamma_0^2}\vert{\Delta\tilde{u}(0)}\vert^2,
            \end{align}
            the second inequality being a consequence of the fact that $[Q,\Lambda^s]\dot{B}$ and $Q[\dot{B},\Lambda^s]$ are operators of order $-1$. In summary, we have
            % In summary, we have 
            % \begin{align*}
            %     \gamma\Vert{\Delta \tilde{u}}\Vert^2+\vert{\Delta \tilde{u}(0)}\vert^2&\leq C\left(\frac{1}{\gamma}\Vert{\Delta\tilde{f}}\Vert^2+\frac{1}{\gamma}\Vert{\Delta\tilde{u}}\Vert^2+\vert{Q\dot{g}}\vert^2+\frac{1}{\gamma_0^2}\vert{\Delta\tilde{u}(0)}\vert^2\right),
            % \end{align*}
            % or 
            % \begin{align}\label{EE_u_l}
            %     \gamma\Vert{\Delta \tilde{u} }\Vert^2+\vert{\Delta \tilde{u}(0)}\vert^2&\leq C\left(\frac{1}{\gamma}\Vert{\Delta \tilde{f} }\Vert^2+\vert{Q\dot{g}}\vert^2\right),
            % \end{align}
            % provided that we pick $\gamma_0$ (and thereby $\gamma$) sufficiently large. 
            \begin{align}
                \Vert{\Delta \underline{P}\Lambda^s\tilde{u}}\Vert^2&\lesssim\Vert{\Delta \tilde{f}}\Vert_{s}^2+\Vert{\Delta\tilde{u}}\Vert_s^2,\\
                \vert{Q\dot{B}\Lambda^s\tilde{u}(0)}\vert^2&\lesssim\vert{Q\tilde{g}}\vert_s^2+\frac{1}{\gamma_0^2}\vert{\Delta\tilde{u}(0)}\vert_s^2.
            \end{align} 
           For the left-hand side of \eqref{EE_u_Sobolev}, we utilize $\Lambda^s\Delta=\Delta\Lambda^s+[\Lambda^s, \Delta]$ and conclude that
           \begin{align*}
                \Vert{\Lambda^s\Delta\tilde{u}}\Vert^2&\lesssim\Vert{\Delta\Lambda^s\tilde{u}}\Vert^2+\Vert{\Psi_{s-1}\tilde{u}}\Vert^2,\\
                \Vert{\Delta\tilde{u}}\Vert_s^2-\frac{1}{\gamma^2_0}\Vert{\Delta\tilde{u}}\Vert_s^2&\lesssim\Vert{\Delta\Lambda^s\tilde{u}}\Vert^2.
            \end{align*}
            Eventually, collecting terms and choosing $\gamma_0$ large enough produces
            \begin{align}
                \gamma\Vert{\Delta \tilde{u} }\Vert_s^2+\vert{\Delta \tilde{u}(0)}\vert_s^2&\leq C\left(\frac{1}{\gamma}\Vert{\Delta\tilde{f} }\Vert_s^2+\vert{Q\tilde{g}}\vert_s^2\right).
            \end{align}
            From this point on, we continue in the same way as we did for $\Vert{\,\cdot\,}\Vert$.
        \end{proof}
    % \section{The well-posedness of $\mathcal{WR}$ problems}\label{wellposedness}
    \section{The well-posedness of WR problems}\label{wellposedness}
        \subsection{Construction of a parametrix.}
        In this section, we address the classical questions of existence, uniqueness, and regularity for solutions of $\mathcal{WR}$ problems using the energy estimates inferred in previous sections. The key idea is to verify that $\tilde{w}_i=\Delta_{\gamma,i}\tilde{u}_i$ is the unique solution of of a problem satisfying the uniform \Lop condition, after which the properties of $\tilde{u}_i$ can be investigated from the analysis of the pseudodifferential equation $\tilde{w}_i=\Delta_{\gamma,i}\tilde{u}_i$.
        
        To start with, let us revisit Step 3 in the proof of Theorem \ref{main}. Under the conditions and notation specified there, $\tilde{w}=\Delta_{\gamma}\tilde{u}$ is a solution of the auxiliary problem
        \begin{equation}\label{StrongSysWRCh4}
            \left\{\begin{aligned}
                \left(D_{d}+\underline{\mathdutchcal{A}}\right)\tilde{w}(t,x)&=\Delta \tilde{f}(t,x), \\
                M\tilde{w}(t,y)&={Q}\tilde{g}(t,y),
            \end{aligned}\right.
       \end{equation}
        which follows from applying $\Delta_{\gamma}$ to both sides of \eqref{ProblemAfterDiag} and from the next identities modulo $\Psi_{-1}$:
        \begin{align}\label{SymProp_2Ch4}
            \nonumber\Delta_{\gamma} \underline{A}_{\gamma}&=\underline{A}_{\gamma}\Delta_{\gamma}-i\Op_{\gamma}(\{\dot{a}_1,\delta\}),
        \end{align}
        \begin{equation*}
            \Delta_{\gamma} D_d=D_d\Delta_{\gamma}-i\Op{\delta},
        \end{equation*}
        and 
        \begin{equation*}
            M_{\gamma}\Delta_{\gamma}=Q_{\gamma}B_{\gamma}.
        \end{equation*}
        If $\Delta_{\gamma} \tilde{f}\in H^s_{\gamma}(\overline{\R}_{+}^{1+d})$ and $\Delta_{\gamma} \tilde{g}\in H^s_{\gamma}(\overline{\R}^{d})$, Theorem \ref{ExistenceUniquenessRegularity} asserts that there exists a unique $\tilde{w}\in H^s_{\gamma}(\overline{\R}_{+}^{1+d})$ with $\tilde{w}|_{x_d=0}\in H^s_{\gamma}(\R^{d})$ that fulfills \eqref{StrongSysWRCh4}. From this perspective, there is some hope of establishing existence, uniqueness, and regularity for $\tilde{u}$ by studying $\Delta \tilde{u}=\tilde{w}$ equipped with $u|_{t=0}=0$. The major steps in this direction are given by the two statements below.
            
        \begin{Theorem}[Lemma 3.2, \cite{taylor1979rayleigh}]\label{parametrix}
            Let $\Omega$ be a compact manifold and suppose that $\R\times \Omega$ has coordinates $(t,x)$. Assume that $P$ is a scalar zeroth-order pseudodifferential operator with real principal symbol $p(t,x,\tau,\xi)$.  If $f\in \mathcal{E}'(\R\times\Omega)$ is supported in $\{t>0\}$ and $\partial_{\tau}p$ is nonvanishing whenever $p=0$, then the pseudodifferential equation 
            \begin{equation}\label{PseudoEquation}
                Pu=f
            \end{equation}
            has a unique solution modulo $C^{\infty}$ that vanishes for $\{t<0\}$. Moreover, $\mathrm{WF}(u)$ is contained in the union of $\mathrm{WF}(Pu)$ and the set of positively time-oriented null bicharacteristics of $p$ passing over $\mathrm{WF}(u)$.
        \end{Theorem}
            
       \begin{Proposition}\label{DeltaPT}
            Suppose that $\delta(X)=\mathrm{diag}(\delta^{-}(X), I_{n-p})$ with
            \begin{equation}
                \delta^{-}(X) \eqdef\begin{pNiceMatrix}
                    \begin{pmatrix}
                        \delta_{1}^{-}(X)&\hdots&0\\
                        \vdots&\ddots&\vdots\\
                        0&\hdots&\delta_{\mathdutchcal{s}}^{-}(X)\\
                    \end{pmatrix}&      \\
                    &I_{p-\mathdutchcal{s}}
                \end{pNiceMatrix},
            \end{equation}
            where $\delta_1(X),\,\cdots, \delta_{\mathdutchcal{s}}(X)$ are as described in Theorem \ref{SymWR}.  Then the pseudodifferential operators 
            \begin{equation*}
                \Op_{\gamma}{(\delta_1)}, \,\cdots, \Op_{\gamma}{(\delta_{\mathdutchcal{s}})}
            \end{equation*}
            are of real principal type.
        \end{Proposition}
        % Let $X=(t,y,x_d,\tau,\eta,\gamma)$,  $X'=(t,y,0,\tau,\eta,\gamma)$, and $\tilde{X}=(t,y,x_d,\tau,\eta,0)$. 
        \begin{proof}
            % Suppose that $K\times S^d\subset\mathbb{X}_S$ be the compact region introduced in Theorem \ref{main}, Step 1. 
            Recall that the unique solution of the transport equation
            \begin{align}
                \left\{\begin{aligned}
         		       \partial_{d}\delta_{i}^{-}+\{\delta_{j}^{-},a_{1,j}\}&=0,\\
         		       \delta_{j}^{-}|_{x_d=0}&=\mat{\Delta},\\
    		    \end{aligned}\right.
            \end{align}  
            %  As before, we may concentrate on a compact region $K\times S^d$. \textcolor{magenta}{Also,}
            is obtained by composing $\mat{\Delta}$ with the inverse of the Hamiltonian flow map $\phi_{x_d}$, that is to say, $\delta_j(X)=(\phi^{*}_{-x_d}\,\mat{\Delta})(X)$ for $X$ in $K\times S^d$, the compact region introduced in Theorem \ref{main}, Step 1. In the sequel, we distinguish two scenarios, namely, when $\gamma>0$ and when $\gamma=0$. In the former, $\mat{\Delta}$  and $\phi^{*}_{-x_d}$ are never vanishing and so $\delta_j(X)$ is elliptic. In the latter, we verify the definition for $-i\delta_j(X)$ rather than for $\delta(X)$. Certainly, 
            \begin{equation*}
                -i\delta_j(\tilde{X})=-i(\phi^{*}_{-x_d}\,\underline{\Delta})(\tilde{X})=\omega(t_\flat,y_{\flat},\zeta_{\flat})
            \end{equation*}
            is real-valued and its roots are simple, as we see now. For $x_d>0$, $\mathrm{Char}(\Op_{\gamma}(\delta_j))$ is the orbit of $\Gamma$ under $\phi_{x_d}$, i.e.,
            \begin{equation*}
                \mathrm{Char}(\delta_j)=\Gamma_j\eqdef\{\Gamma \;\textrm{transported along the flow}\; \phi_{x_d}\}.
            \end{equation*}
            Therefore, we have to show that $\partial_{\tau}\delta_j(X)$ is non-vanishing on $\Gamma_j$, which amounts to proving that
            \begin{equation}\label{derdelta}
                \partial_{\tau}\delta_j=(\partial_{\tau}\phi^{*}_{-x_d}\underline{\Delta})=(\partial_{\tau}\underline{\Delta})(\phi_{-x_d})\partial_{\tau}\phi_{-x_d}\neq 0
            \end{equation}
            on $\Gamma_i$. The first factor in \eqref{derdelta} is different from zero by the very definition of the $\WR$ class, whereas $\partial_{\tau}\phi_{-x_d}$ can be computed via 
            \begin{equation*}
                \partial_{\tau}\phi_{-x_d}=\exp\left(\int_0^{t}{\partial_{\tau}H_{\phi}(s)(\phi_{-s})}\,ds\right)\neq 0,
            \end{equation*}
            where $H_{\phi}(x_d)(\cdot)$ stands for the vector field associated with the flow $\phi_{x_d}$ (see \cite{arnold1992ordinary}). 
        \end{proof}

    \subsection{Existence, uniqueness, and regularity.}  
        Let us focus on the pseudodifferential system
        \begin{equation}\label{DeltaDiag}
            \Delta_{\gamma}\tilde{u} =\begin{pNiceMatrix}
                \begin{pmatrix}
                    \Op_{\gamma}{(\delta_1)}&&\\
                    &\ddots&\\
                    &&\Op_{\gamma}{(\delta_{\mathdutchcal{s}})}\\
                \end{pmatrix}&      \\
                &I_{n-\mathdutchcal{s}}
            \end{pNiceMatrix}\begin{pmatrix}
                    \tilde{u}_1\\
                    \vdots\\
                    \tilde{u}_n\\
                \end{pmatrix}=\begin{pmatrix}
                    \tilde{w}_1\\
                    \vdots\\
                    \tilde{w}_n\\
                \end{pmatrix}=\tilde{w}
        \end{equation}
        equipped with homogeneous Cauchy data. Reasoning component-wise, we get $\mathdutchcal{s}$ (nontrivial) initial value problems that can be solved on $K\times S^d$ according to Theorem \ref{parametrix}. More precisely, there holds
       
       \begin{Theorem}\label{PolarisationWR}
           Consider
           \begin{equation}\label{CauchyDelta}
            \left\{\begin{aligned}
                \Op_{\gamma}{(\delta_j)}\tilde{u}_j &=\tilde{w}_j, \\
                \tilde{u}_j|_{t=0}&=0,
            \end{aligned}\right.
            \end{equation}
            where $\tilde{w}\in H^s_{\gamma}(\overline{\R}_{+}^{1+d})$ and $\Op_{\gamma}{(\delta_j)}$ is defined as in Proposition \ref{DeltaPT}. Then there exists a unique solution of \eqref{CauchyDelta} modulo $C^{\infty}$ for which 
            \begin{equation}\label{PropagationOfSingularitiesII}
               \mathrm{WF}_{s-1}(\tilde{u}_j)\setminus \mathrm{WF}_s(\Op_{\gamma}{(\delta_j)}\tilde{u}_j) \subset\delta_j^{-1}(0).
           \end{equation}
       \end{Theorem}
       \begin{proof}
            A direct application of Theorem \ref{parametrix} and Proposition \ref{DeltaPT}.
       \end{proof}

       Formula \eqref{PropagationOfSingularitiesII} and Proposition \ref{DeltaHminus} allow us to characterize $H^s_{\Delta}$ alternatively.

       \begin{Proposition}\label{H^s'}
             Let 
             \begin{align*}
                  H^s_{\Delta}&\eqdef\{v \in \mathcal{S}'(\R^{1+d}_+,\,\R^n): \Lambda^s_{\gamma} v \in L_{\Delta}^2(\R^{1+d}_+,\,\C^n)\}
              \end{align*} 
              be the function space introduced in Theorem \ref{main}. Then 
              \begin{equation*}
                  H^s_{\Delta}=\{u\in H_{\gamma}^{s-1}: \Delta_{\gamma} u \in H_{\gamma}^s\}.
              \end{equation*}
        \end{Proposition}
        \begin{proof}
             Let $u\in \underline{H}^s_{\Delta}\eqdef\{u\in H_{\gamma}^{s-1}: \Delta_{\gamma} u \in H_{\gamma}^s\}$. By definition, $u \in H_{\gamma}^{s-1}$ and $\Lambda^{s}\Delta_{\gamma} u \in L^2_{\gamma}$, so we need to prove that $\Delta_{\gamma}\Lambda_{\gamma}^su\in L^2_{\gamma}$. Indeed, from
             \begin{equation}\label{commutatornorm}
                 \Delta_{\gamma}\Lambda_{\gamma}^su= \Lambda_{\gamma}^s\Delta_{\gamma} u+[\Delta_{\gamma},\Lambda_{\gamma}]u,
             \end{equation} 
             it will be enough to check that $[\Delta_{\gamma},\Lambda_{\gamma}]u\in L_{\gamma}^2$. But this is a straightforward consequence of the fact that $[\Delta_{\gamma},\Lambda^s_{\gamma}]$ is an operator of order $s-1$ and $u\in H^{s-1}_{\gamma}$. 
             
             Suppose now that $u\in H^s_{\Delta}$, meaning that $\Delta_{\gamma}\Lambda_{\gamma}^su\in L^2_{\gamma}$. From Proposition \ref{DeltaHminus}, 
             \begin{equation*}
                 \Vert{u}\Vert_{s-1}\leq \Vert{\Lambda^s_{\gamma}}u\Vert_{-1}\leq \Vert{\Delta_{\gamma}\Lambda^s_{\gamma}}u\Vert< \infty,
             \end{equation*}
             and thus it is clear that $\Lambda_{\gamma}^s\Delta_{\gamma} u\in L^2_{\gamma}$  from \eqref{commutatornorm} and the same argument as before.
        \end{proof}  
        We can interpret Theorem \ref{PolarisationWR} and Proposition \ref{H^s'} as follows: given that $\tilde{w}\in H_{\gamma}^s$, the first $\mathdutchcal{s}$ components experience a loss of regularity of one derivative, while the others remain unchanged. This supports Serre's observation in \cite{serre2005solvability} that the solution $\tilde{u}$ exhibits a polarization effect around the critical set $\Gamma$.  
            
        \textbf{Acknowledgments.} This work was supported by the DAAD and the Research Training Group 2491 ``Fourier Analysis and Spectral Theory" from the University of Göttingen. The author thanks Professor Ingo Witt for his valuable input and generous sharing of his time and expertise. Special thanks also go to Dr. Christian Jäh for his decisive help with one of the lemmas of this paper.
        
        \bibliographystyle{alpha}
        \bibliography{Main}
        \newpage

    % \bibliography{Main}
    % \bibliographystyle{alpha}
\end{document}